\documentclass[12pt]{article}
\RequirePackage{amsthm,amsmath,amsfonts,amssymb,mathrsfs}
\RequirePackage[numbers]{natbib}
\usepackage{graphicx}
\usepackage[export]{adjustbox}
\usepackage{float}
\usepackage{caption} 
\usepackage{subcaption}
\usepackage{booktabs, makecell}
\usepackage{multirow}
\usepackage{hyperref}
\usepackage{comment}
\usepackage{array}
\usepackage{multirow}
\usepackage{booktabs} 
\usepackage{bm}
\usepackage{mathtools}
\usepackage{booktabs}
\usepackage{amssymb}
\usepackage{accents}
\usepackage{amsthm}
\usepackage{cases}
\usepackage[english]{babel} 
\usepackage[protrusion=true,expansion=true]{microtype} 
\usepackage{microtype} 
\usepackage{tabularx}
\usepackage[font = small,labelfont=bf,textfont=it]{caption} %
\usepackage[toc,page]{appendix}
\usepackage{dsfont}
\usepackage{booktabs} 
\usepackage{natbib}
\usepackage{setspace}

\usepackage{soul}
\usepackage{enumitem}

\usepackage{multirow}
\usepackage{algorithm}
\usepackage[noend]{algpseudocode}
\usepackage{etoolbox}
\algblock{ParFor}{EndParFor}

\usepackage{graphicx,epstopdf}
\usepackage{booktabs}
\usepackage{siunitx}
\makeatletter

\newcommand{\blind}{1}

\theoremstyle{plain}
\newtheorem{theorem}{Theorem}
\newtheorem{proposition}{Proposition}
\newtheorem{corollary}{Corollary}
\newtheorem{lemma}{Lemma}
\theoremstyle{remark}
\newtheorem{definition}{Definition}
\newtheorem{remark}{Remark}

\newtheorem{assumption}{Assumption}
\newcommand\mcc[1]{\multicolumn{2}{>{$}c<{$}}{#1}}

\def\spacingset#1{\renewcommand{\baselinestretch}
{#1}\small\normalsize} \spacingset{1.1}

\begin{document}






\def\spacingset#1{\renewcommand{\baselinestretch}
{#1}\small\normalsize} \spacingset{1.1}

\if1\blind
{
 
 \centering{\bf\Large Asymptotic Theory for Linear Functionals of Kernel Ridge Regression}\\
  \vspace{0.2in}
  \centering{Rui Tuo$^{a}$, Lu Zou$^{b}$\vspace{0.2in}\\
  $^{a}$Department of Industrial and Systems Engineering, Texas A\&M University, ruituo@tamu.edu \\
        $^{b}$ School of Management, Shenzhen Polytechnic University, lzou@connect.ust.hk \\
        }
    \date{\vspace{-7ex}}
 
} \fi

\bigskip

\begin{abstract}
An asymptotic theory is established for linear functionals of the predictive function given by kernel ridge regression, when the reproducing kernel Hilbert space is equivalent to a Sobolev space. The theory covers a wide variety of linear functionals, including point evaluations, evaluation of derivatives, $L_2$ inner products, etc. We establish the upper and lower bounds of the estimates and their asymptotic normality. 
We show the asymptotic normality of these estimators under mild conditions, which enables uncertainty quantification of a wide range of frequently used plug-in estimators. The theory also implies that the minimax $L_\infty$ error of kernel ridge regression can be attained under $\lambda\sim n^{-1}\log n$. 
\end{abstract}

\noindent%



\section{Introduction}

Consider a nonparametric regression model 
\begin{eqnarray}\label{model}
    y_i=f(x_i)+e_i
\end{eqnarray}
with $e_i$'s being independent and identically distributed random errors with mean zero and a finite variance $\sigma^2$. Here $x_i$'s can be deterministic or random inputs independent of $e_i$'s. Nonparametric regression aims to estimate $f$ from data $(x_i,y_i),i=1,\ldots,n$.

Kernel ridge regression (KRR) is defined as
\begin{eqnarray}\label{KRR:1}
    \hat{f}:=\operatorname*{argmin}_{v\in\mathcal{H}}\frac{1}{n}\sum_{i=1}^n(y_i-v(x_i))^2+\lambda \|v\|^2_{\mathcal{H}},
\end{eqnarray}
given data $(x_i,y_i)_{i=1}^n$, where $\mathcal{H}$ is the reproducing kernel Hilbert space generated by a kernel function $K$, and $\lambda>0$ is called the smoothing parameter. We use the notation $\|\cdot\|_\mathcal{H}$ and $\langle\cdot,\cdot\rangle_\mathcal{H}$ to denote the norm and the inner product of $\mathcal{H}$, respectively. It is well known that $\hat{f}$ is a good estimator for $f$ under mild conditions.

In many real-world problems, the quantity of interest is a linear functional of $f$, denoted by $l(f)$, such as an  evaluation or a derivative of $f$ at a pre-specified point, or an integral of $f$. Sometimes, the quantity of interest is nonlinear in $f$ by itself, but is closely related to a linear functional. For instance, the maximizer of $f$ is the zero point of the gradient function of $f$. Plug-in estimators are widely used in practice, that is, to estimate $l(f)$ by $l(\hat{f})$. This work aims at providing theoretical justification and a framework of uncertainty quantification for these plug-in estimators.

\subsection{Problem of Interest and Overview of Our Results}

In this work, we consider the asymptotic properties of a linear functional of $\hat{f}-f$ defined as general as 
\begin{eqnarray}\label{PoI}
    l(\hat{f}-f):=\langle \hat{f}-f,g\rangle_\mathcal{H},
\end{eqnarray}
for some $g\in\mathcal{H}$.
This includes many examples of practical interest, e.g., 
$L_2$ inner products $\int_\Omega (\hat{f}-f)(x)h(x)dx=\left\langle \hat{f}-f,\int_\Omega K(\cdot,x)h(x)dx\right\rangle_\mathcal{H}$,
point evaluations $(\hat{f}-f)(x)=\left\langle \hat{f}-f,K(\cdot,x)\right\rangle_\mathcal{H}$, 
point evaluations of derivatives $\frac{\partial}{\partial x_i}(\hat{f}-f)(x)=\left\langle \hat{f}-f,\frac{\partial}{\partial x_i}K(\cdot,x)\right\rangle_\mathcal{H}$.

As we shall study theoretical properties as $n\rightarrow \infty$, the input and output data, the minimizer $\hat{f}$, and the tuning parameter $\lambda$ should all naturally be dependent on $n$. In addition, unless otherwise specified, the true function $f$ can depend on $n$ as well. While keeping this fact in mind, we shall omit the subscript $n$ for the sake of notational convenience throughout this article.
Below is a summary of our major contributions.

\begin{enumerate}
    \item 
We develop a new method to investigate the asymptotic properties of a single linear functional of the form $\langle \hat{f},g\rangle_{\mathcal{H}}$ to answer the following questions: 1) How large is the bias and variance of $\langle \hat{f},g\rangle_\mathcal{H}$ as an estimator of $\langle f,g\rangle_\mathcal{H}$; 2) What is an appropriate rate of $\lambda$ to facilitate the estimation of $\langle f,g\rangle_\mathcal{H}$; and 3) Is $\langle \hat{f},g\rangle_\mathcal{H}$ asymptotically normal?

While our theory depicts a more general picture, we give Table \ref{tab:summary} to highlight a few cases of particular practical interests. It can be seen that our theory gives the exact rate of convergence and the central limit theorem for these statistics under a wide range of $\lambda$. It also shows that $\lambda\sim n^{-1}$  balances the variance and the worst-case bias regardless of the specific linear functional.

\begin{table}[h]
\resizebox{0.99\textwidth}{!}{
    \centering
    \begin{tabular}{c|c|c|c|c}
    \hline
    \multirow{2}{*}{Functional} & \multicolumn{2}{c|}{Upper \& lower rates} & \multirow{2}{*}{Range of $\lambda$} & \multirow{2}{*}{Central limit theorem}\\
    \cline{2-3}
    & Variance & Worst-case bias & &\\
        \hline
        Point evaluation & $n^{-1}\lambda^{-\frac{d}{2m}}$ & $\lambda^{\frac{1}{2}-\frac{d}{4m}}$ & \multirow{ 3}{*}{\makecell{$\lambda=O(1)$ \\ $\lambda^{-1}=O(n^{\frac{2m}{d}})$}} & \multirow{ 3}{*}{\makecell{Valid if $\lambda=o(1)$ and \\ $\lambda^{-1}=o(n^{\frac{2m}{d}})$}}  \\
         \cline{1-3}
        Derivative evaluation &  $n^{-1}\lambda^{-\frac{d+2|\alpha|}{2m}}$ & $\lambda^{\frac{1}{2}-\frac{d+2|\alpha|}{4m}}$  &  &   \\
         \cline{1-3}
        $L_2$ inner product &  $n^{-1}$ & \makecell{No more than\\ $\lambda^{\frac{1}{2}}$}  &  &   \\
         \hline
    \end{tabular}}
    \caption{Summary of asymptotic properties of linear functionals of practical interest, where $d=$ input dimension, $m=$ smoothness, $|\alpha|=$ total order of derivatives. Exact upper and lower rates of convergence are given, except for the worst-case bias for the $L_2$ inner product. Discussions regarding this matter is made in Section \ref{sec:semiparametric} of Supplementary Material.}
    \label{tab:summary}
\end{table}

\item Our asymptotic theory for linear functionals can be employed to find upper and lower bounds for uniform errors as well. In this work, we examine the global error of the KRR regression as well as the derivatives, in terms of $\sup_{x\in\Omega}|D^\alpha \hat{f}(x)-D^\alpha f(x)|$. An exact rate of convergence is given when the noise is normally distributed. We show that with $\lambda\sim n^{-1}\log n$, the resulting rate of convergence is $(n^{-1}\log n)^{\frac{1}{2}-\frac{d+2|\alpha|}{4m}}$, matching the known minimax rate in \cite{nemirovski2000topics}.
This result implies that $\lambda$ reaches the $L_\infty$-minimax rate differs from the one that reaches the $L_2$-minimax rate.

\item Our theory can be leveraged to cover some non-linear functionals that can be linearized asymptotically, such as $\max_{x\in\Omega}f(x)$.
\end{enumerate}

The remainder of this article is organized as follows. We review the related work in Section \ref{subsec:literature}. In Section \ref{sec:BnV}, we introduce the bias-variance decomposition of the problem. 

The main results of our theory are presented in Section \ref{sec:improved}, in terms of the general theory of the upper and lower bounds and asymptotic normality. In Section \ref{sec:examples}, we present several examples to illustrate the scope of the proposed framework. In Section \ref{sec:other}, we employ our theory to obtain some uniform error bounds for KRR and investigate a nonlinear problem to further demonstrate the applicability of our theory. 

Numerical studies and an analysis of real-world data are presented in Section \ref{sec:NS}.
The Supplementary Materials provide a more in-depth review of the literature, other related results, detailed discussions of a key assumption, and all technical proofs.

\subsection{Related Work}\label{subsec:literature}
KRR was initially introduced in the context of spline models \citep{wahba1978improper} and support vector machines \citep{boser1992training}, due to its innate capacity to accommodate complex patterns and nonlinear relationships. 

\textit{Error bounds for KRR.} The minimax convergence rates for KRR in $L_2$  are well established in the existing literature; see, e.g., \cite{caponnetto2007optimal,smale2007learning,steinwart2009optimal,mendelson2010regularization}, among many others.  Although there has been rich literature on the theoretical guarantees of KRR, theory on functionals of KRR estimators is scarce. The closely related work is \cite{liu2023estimation}, which offers a non-asymptotic analysis of the plug-in KRR estimator for its partial mixed derivatives. This paper develops a general theory on rates of convergence and statistical inference covering a diverse set of linear functionals, which includes derivatives considered in \cite{liu2023estimation}.  Another series of work related to this paper delves into linear functional regression \citep{cai2012minimax,yuan2010reproducing}. Nevertheless, this literature often assumes the linear functional as the $L_2$ inner product of the input data with a slope function, and primarily focuses on the asymptotic properties of the slope function.  Some linear functionals in terms of the $L_2$ inner product fall into the semiparametric regime, see \cite{kosorok2008introduction,geer2000empirical}. Our theory also extends these results by weakening the requirements for the smoothness of the function in the $L_2$ inner product. 

\textit{Statistical inference for KRR.}  Another approach uses KRR for statistical inference, often investigating Gaussian approximation for KRR and its variants. Starting with \cite{huang2003local}, which established pointwise asymptotic normality for the polynomial B-spline estimator, several works have studied constructing uniform confidence bands assuming the objective function lies in an RKHS; see \citep{shang2013local,cheng2013joint, zhao2021statistical}. The uniform asymptotic inference results in this literature rely on expressing the KRR estimator through an orthonormal basis. Our result yields pointwise asymptotic normality for KRR under weaker conditions.  Furthermore, we demonstrate that many other linear functionals of KRR also exhibit  asymptotic normality under both fixed and random designs. The existing literature on statistical inference for KRR has mainly focused on regression functions. The relevant work in this area is \cite{liu2023estimationanova}, which introduced a plug-in KRR estimator to estimate derivatives of a smoothing spline ANOVA model and provided convergence rates and asymptotic normality. Their estimation and inference theorem relies on the tensor structure and the equivalent kernel technique \citep{messer1993new,silverman1984spline}. However, this method cannot be directly applied to non-tensor product structures like the Mat\'ern kernels. Instead, we do not assume a tensor product structure and our analysis also covers derivatives of more general orders. A more detailed discussion of related literature is deferred to the Supplementary Material.

\section{Bias and Variance}\label{sec:BnV}

For simplicity, we introduce the following notation. For any $A=(a_1,\ldots,a_m)^T$ and $B=(b_1,\ldots,b_l)^T$, denote $K(A,B)=(K(a_i,b_j))_{i j}$. Denote $X=(x_1,\ldots,x_n)^T$ and $Y=(y_1,\ldots,y_n)^T$.
Then the representer's theorem \citep{scholkopf2001generalized,wahba1990spline} provides an explicit expression of $\hat{f}$ in (\ref{KRR:1}) as 
$\hat{f}(x)=K(x,X)(K(X,X)+\lambda n I)^{-1}Y$.
Thus, we have
$\langle \hat{f},g\rangle_\mathcal{H}=g^T(X)(K(X,X)+\lambda n I)^{-1}Y$,
where $g^T(X)=(g(x_1),\ldots,g(x_n))$. Now split $Y=F+E=:(f(x_1),\ldots,f(x_n))^T+(e_1,\ldots,e_n)^T$. Then
\[\langle \hat{f},g\rangle_\mathcal{H}=g^T(X)(K(X,X)+\lambda n I)^{-1}F+g^T(X)(K(X,X)+\lambda n I)^{-1}E.\]
Let $\mathbb{E}_E$ and $\operatorname{Var}_E$ be the expectation and variance operators with respect to $E$, respectively. Note that $X$ is independent of $E$, if $X$ is random at all. Taking expectation or variance with respect to $E$ will leave $X$ as is. 
We call the quantity in (\ref{key:bias}) the \textit{bias}, denoted as $\operatorname{BIAS}$:
\begin{eqnarray}\label{key:bias}
 \operatorname{BIAS}:=\mathbb{E}_E\langle \hat{f}-f,g\rangle_\mathcal{H}=g^T(X)(K(X,X)+\lambda n I)^{-1}F-\langle f,g\rangle_\mathcal{H}. 
\end{eqnarray}
We call (\ref{key:variance}) the \textit{variance term}.
\begin{eqnarray}\label{key:variance}
    \langle \hat{f}-\mathbb{E}_E\hat{f},g\rangle_\mathcal{H}=g^T(X)(K(X,X)+\lambda n I)^{-1}E.
\end{eqnarray}
The term (\ref{key:VAR}) is called the \textit{variance}, denoted as $\operatorname{VAR}$:
\begin{eqnarray}\label{key:VAR}
   \operatorname{VAR}:=\operatorname{Var}_E \langle \hat{f}-f,g\rangle_\mathcal{H}=\sigma^2g^T(X)(K(X,X)+\lambda n I)^{-2}g(X).
\end{eqnarray}
A primary objective of this study is to quantify  $\operatorname{BIAS}$ and $\operatorname{VAR}$ as the sample size tends to infinity. It is important to note that, unlike $\operatorname{VAR}$, $\operatorname{BIAS}$ is dependent on the underlying true function $f$. Sometimes, we want to emphasize this dependency by denoting the bias as $\operatorname{BIAS}_f$, when the interest lies in understanding the lower bounds of the \textit{worst case bias} over the RKHS unit ball, defined as $\sup_{\|f\|_\mathcal{H}\leq 1}|\operatorname{BIAS}_f|$.
To analyze the bias and variance, this work introduces an innovative tool called \textit{noiseless kernel ridge regression}, which is detailed in Section \ref{sec:proofs_main} of the Supplementary Materials.

\section{Main Results}\label{sec:improved}

In this section, we will present three types of major theoretical results: the upper bounds in Section \ref{sec:upper}, the lower bounds in Section \ref{sec:lower}, and the asymptotic normality results in Section \ref{sec:AN}.
First, we introduce a set of assumptions in Section \ref{sec:assumptions}.

\subsection{Assumptions}\label{sec:assumptions}

While the proposed techniques can be applied in other settings, in this work, we only consider the situations when $\mathcal{H}$ is equivalent to a (fractional) Sobolev space (see Section \ref{sec:spaces} of the Supplementary Materials), leading to Assumption \ref{assum:matern}.

\begin{assumption}\label{assum:matern}
    The input domain $\Omega$ is a convex and compact subset of $\mathbb{R}^d$ with a non-empty interior.
    In addition, $\mathcal{H}$ is equal to a (fractional) Sobolev space with order $m$ (satisfying $m>d/2$), denoted by $H^m$, with equivalent norms. 
\end{assumption}

The condition $m>d/2$ is to ensure that $H^m$ is embedded into the space of continuous functions, according to the Sobolev embedding theorem. This embedding is necessary because otherwise, the point evaluation $f(x)$ is mathematically not well-defined.
The spaces $\mathcal{H}$ and $H^m$ are equivalent if $K$ is an isotropic Mat\'ern kernel with smoothness $\nu=m-d/2$, under the regularity conditions for $\Omega$ in Assumption \ref{assum:matern}; see \citep{wendland2004scattered}.

Now we formally introduce the smoothness requirement of $g$. The intuition behind Assumption \ref{assum:f} is that $g$ has to be smoother than the baseline smoothness of $\mathcal{H}$. More discussion is deferred to Sections \ref{lemma:relationship}-\ref{sec:A2} in the Supplementary Material.

\begin{assumption}\label{assum:f}
    There exist constants $C_g>0$ and $\delta\in(0,1]$, such that for each $v\in\mathcal{H}$,
    \begin{eqnarray}\label{smoothness}
        |\langle g,v\rangle_{\mathcal{H}}|\leq C_g\|v\|^{\delta}_{L_2}\|v\|^{1-\delta}_{\mathcal{H}}.
    \end{eqnarray}
\end{assumption}

Note that (\ref{smoothness}) is always true if $\delta=0$, by plugging in $C_g=\|g\|_{\mathcal{H}}$, which imposes no extra conditions. This is why we need $\delta>0$. As $\|\cdot\|_\mathcal{H}$ is stronger than $\|\cdot\|_{L_2}$, a larger $\delta$ fulfilling Assumption \ref{assum:f} can imply that Assumption \ref{assum:f} is also true for a smaller $\delta$. As we will see later, the larger $\delta$ is, we can expect the more improvements in the rates of convergence.
In Section \ref{sec:examples}, we will give the corresponding $\delta$ value for each of the aforementioned linear functionals.

We also need regularity conditions for the input sites. In this work, the design points can be either random or fixed, provided that Assumption \ref{assum:norms} holds.

\begin{assumption}\label{assum:norms}
If $X$ is random, $X$ is independent of $E$.
    Besides, there exists $C_1>0$, and for each $\epsilon>0$, there exists $C_\epsilon>0$, both independent of $n$ and $X$, such that $\mathbb{P}(\Xi_\epsilon)\geq 1-\epsilon$, where $\Xi_\epsilon$ denotes the event
    \begin{eqnarray}
        &&\|v\|_{L_2}\leq \max\left\{C_1 \|v\|_{n},C_\epsilon n^{-m/d}\|v\|_{\mathcal{H}}\right\},\label{l2_to_n}\\
        &&\|v\|_n\leq \max\left\{C_1\|v\|_{L_2},C_\epsilon n^{-m/d}\|v\|_{\mathcal{H}}\right\}.\label{n_to_l2}
    \end{eqnarray}
    for all $v\in \mathcal{H}$.
\end{assumption}

 In Section \ref{sec:sufficient} of the Supplementary Material, we give some sufficient conditions for Assumption \ref{assum:norms}. Specifically, Assumption \ref{assum:norms} holds for 1) \textit{random designs} whose points are independent and identically distributed samples from a probability density bounded away from zero and infinity, and 2) \textit{fixed designs} that are quasi-uniform.

It is worth noting that in Assumption \ref{assum:norms}, the probability is taken with regard to the randomness of $X$, and in case $X$ is deterministic, the norm inequalities (\ref{l2_to_n}) and (\ref{n_to_l2}) should hold unconditionally.
To obtain the improved rates and the upper bounds, condition (\ref{l2_to_n}) alone suffices. The lower bounds and the asymptotic normality will also need condition (\ref{n_to_l2}).

Connecting the $\|\cdot\|_n$ and the $\|\cdot\|_{L_2}$ norms is crucial in the theory of a variety of nonparametric regression methods; see \cite{huang2003local,geer2000empirical} for example. In Assumption \ref{assum:norms}, the event $\Xi_\epsilon$ serves as a set of high probability such that $\|\cdot\|_n$ and $\|\cdot\|_{L_2}$ are comparable. Lemma \ref{lemma:var} shows a simple but important consequence of Assumption \ref{assum:norms}.

\begin{lemma}\label{lemma:var}
    With Assumption \ref{assum:norms} and the conditions $\sigma^2\neq 0$ and $g\neq 0$, we have $\operatorname{VAR}\neq 0$ with probability tending to one, as $n\rightarrow\infty$.
\end{lemma}

\subsection{Upper Bounds}\label{sec:upper}

We shall use the following notation for asymptotic orders. For (possibly random) sequences $a_n,b_n>0$, we denote $a_n\lesssim b_n$ if $a_n/b_n$ is bounded in probability; denote $a_n\gtrsim b_n$ if $b_n\lesssim a_n$; and $a_n\asymp b_n$ if $a_n\lesssim b_n$ and $a_n\gtrsim b_n$.

\begin{theorem}\label{Coro:rate}
Suppose $\lambda\gtrsim n^{-{2m}/d}$. Under Assumptions \ref{assum:matern}-\ref{assum:norms}, we have
\begin{eqnarray}
    |\operatorname{BIAS}|&=&O_\mathbb{P}(\lambda^{\frac{\delta}{2}}\|f\|_\mathcal{H}),\label{BIASupper}\\
    \operatorname{VAR}&=&O_\mathbb{P}(\sigma^2n^{-1}\lambda^{\delta-1}).\label{VARupper}
\end{eqnarray}
\end{theorem}

\subsection{Lower Bounds}\label{sec:lower}

It is not surprising that $\operatorname{VAR}$ should have a lower bound, in view of the classic statistical theory such as the Cram\'er-Rao lower bound.
Here we would like to pursue a lower bound as close as possible to the upper bound in Theorem \ref{Coro:rate}.

Note that the upper bounds of the rate of convergence depend on the best $\delta$ value that ensures Assumption \ref{assum:f}. Intuitively, a lower bound should rely on a $\delta$ value that disallows for (\ref{smoothness}) in Assumption \ref{assum:f}. 
To elaborate on the condition to be introduced, we first present an equivalent statement of Assumption \ref{assum:f}. For notational simplicity, we use the convention $\frac{0}{0}=0$ throughout this article. 

\begin{proposition}\label{Prop:equivalence}
    Under Assumption \ref{assum:matern}, given $g\in\mathcal{H}$ and $\delta\in(0,1]$, $\sup_{v\in\mathcal{H}}\frac{|\langle g,v\rangle_\mathcal{H}|}{\|v\|_{L_2}^\delta\|v\|^{1-\delta}_\mathcal{H}}$ is finite if and only if  for each $R>0$, 
    \begin{eqnarray}\label{delta_eqv}
        \sup_{\|v\|_\mathcal{H}\leq R\|v\|_{L_2}}\frac{|\langle g,v\rangle_\mathcal{H}|}{\|v\|_{L_2}}\leq C R^{1-\delta},
    \end{eqnarray}
    for some constant $C>0$ independent of $R$.
\end{proposition}

Our lower bounds rely on the reversed direction of the inequality (\ref{delta_eqv}), showing in Assumption \ref{assum:tau}.

\begin{assumption}\label{assum:tau}
    For some $\tau\in (0,1]$, there exist constants $C_0>0$ and $R_0>0$ such that
      $ \sup_{\|v\|_\mathcal{H}\leq R\|v\|_{L_2}}\frac{\langle g,v\rangle_\mathcal{H}}{\|v\|_{L_2}}> C_0 R^{1-\tau}$,
    for each $R\geq R_0$.
\end{assumption}

It is worth noting that Assumption \ref{assum:tau} implies that $g\neq 0$.
In view of Proposition \ref{Prop:equivalence}, if Assumptions \ref{assum:f} and \ref{assum:tau} are both true, we clearly have $\delta\leq \tau$. As opposed to Assumption \ref{assum:f}, a smaller $\tau$ fulfilling Assumption \ref{assum:tau} can imply that Assumption \ref{assum:tau} is also true for a larger $\tau$. The case $\tau=1$ is trivially true provided that $g\neq 0$, for $R_0= \|g\|_\mathcal{H}/\|g\|_{L_2}$ and $C_0=\|g\|^2_\mathcal{H}/\|g\|_{L_2}$. It is not hard to imagine that $\tau$ plays an important role in characterizing our lower bound of the rate of convergence in Theorem \ref{Th:lower}.

\begin{theorem}\label{Th:lower}
   Suppose Assumptions \ref{assum:matern}-\ref{assum:tau} hold. Then for each $\epsilon>0$, there exist constants $A_1,A_2,A_3>0$ depending only on $C_0,C_1,C_g,C_\epsilon,R_0,\delta$, and $\tau$, such that, on the event $\Xi_\epsilon$ introduced in Assumption \ref{assum:norms}, for any $n$ and $\lambda$ satisfying $A_1n^{-2m/d}\leq \lambda \leq A_2$, we have
   $\operatorname{VAR}\geq A_3\sigma^2n^{-1}\lambda^{\frac{\delta(\tau-1)}{\tau}}$.
  
\end{theorem}

The trivial case $\tau=1$ leads to a ``parametric-rate'' lower bound $\operatorname{VAR}\gtrsim \sigma^2 n^{-1}$, which is not surprising. Besides, it is particularly interesting when $\delta=\tau$, as the lower rate in Theorem \ref{Th:lower} coincides with the upper rate in Theorem \ref{Coro:rate}. This leads to Theorem \ref{Coro:varrate}. We will show in Section \ref{sec:examples} that $\delta=\tau$ is indeed true for many examples of practical interest.

\begin{theorem}\label{Coro:varrate}
    Suppose $g\in \mathcal{H}$ satisfies Assumptions \ref{assum:f} and \ref{assum:tau} with $\delta=\tau$. Besides, Assumptions \ref{assum:matern} and \ref{assum:norms} hold. Then for each $\epsilon>0$, there exist constants $A_1,A_2,A_3,A_4>0$ depending only on $C_0,C_1,C_g,C_\epsilon,R_0,\delta$, and $\tau$, such that, on the event $\Xi_\epsilon$ introduced in Assumption \ref{assum:norms}, for any $n$ and $\lambda$ satisfying $A_1n^{-2m/d}\leq \lambda \leq A_2$, we have
   $A_3\sigma^2n^{-1}\lambda^{\delta-1}\leq \operatorname{VAR}\leq A_4\sigma^2n^{-1}\lambda^{\delta-1}$.
\end{theorem}

Now we consider the bias term. First, we note that the bias depends on the underlying true function $f$. If $f\equiv 0$, we can clearly see $\operatorname{BIAS}=0$. A more meaningful study of the lower bounds for bias is to consider the \textit{worst-case bias}. To define a worst-case bias, we imagine the application of KRR to a family of
models having the form of equation (\ref{model}), but with different $f$. Nevertheless, the same $g$ and parameter $\lambda$ are used for each model. For each $f$, denote the corresponding bias by $\operatorname{BIAS}_f$. we Theorem \ref{Coro:biasrate} provides a lower bound for the worst-case bias over the unit ball of $\mathcal{H}$.

\begin{theorem}\label{Coro:biasrate}
    Suppose Assumptions \ref{assum:matern}-\ref{assum:tau} hold. Then for each $\epsilon>0$, there exist constants $A_1,A_2,A_3>0$ depending only on $C_0,C_1,C_g,C_\epsilon,R_0,\delta$, and $\tau$, such that, on the event $\Xi_\epsilon$ introduced in Assumption \ref{assum:norms}, for any $n$ and $\lambda$ satisfying $A_1n^{-2m/d}\leq \lambda \leq A_2$, we have
    \begin{eqnarray}\label{supbiaslower1}
        \sup_{\|f\|_\mathcal{H}\leq 1}|\operatorname{BIAS}_f|\geq\begin{cases}
               A_3 \lambda^{\frac{2\tau-2\delta+\delta^2-\delta^2\tau}{2\tau(1-\delta)}} & \text{ if } \delta<1\\
               A_3 \lambda & \text{ if } \delta=1
           \end{cases};
    \end{eqnarray}
    and in particular, if $\delta=\tau<1$, 
    \begin{eqnarray}\label{supbiaslower}
        A_3\lambda^\frac{\delta}{2}\leq \sup_{\|f\|_\mathcal{H}\leq 1}|\operatorname{BIAS}_f|\leq A_4\lambda^\frac{\delta}{2},
    \end{eqnarray}
    for some $A_4$ depending only on $C_0,C_1,C_g,C_\epsilon,R_0$, and $\delta$.
\end{theorem}

\begin{remark}
There is a sharp transition in the lower bounds (\ref{supbiaslower1}) between the case $\delta<1$ and $\delta=1$, 
showing completely different rates of convergence. Despite the weird appearance, this gap in the rate of convergence is genuine! When $\delta=1$, there exists a \textit{semiparametric effect} that may significantly boost the rate of convergence of the bias so that $\sup_{\|f\|_\mathcal{H}\leq 1}|\operatorname{BIAS}_f|$ can become much smaller than the lower bound suggested in (\ref{supbiaslower}). It is implied in the literature concerning the semiparametric properties of KRR (e.g., \cite{mammen1997penalized,tuo2015efficient,geer2000empirical}) that there exist cases with $\delta=1$, such that $\operatorname{BIAS}=o(n^{-1/2})$ whenever $n^{-1}\lesssim\lambda=o(n^{-1/2})$, which definitely violates (\ref{supbiaslower}). The semiparametric effect improves the bias rate of convergence through a mechanism different from what we have discussed.
Further investigations in Section \ref{sec:semiparametric} of the Supplementary Materials also show that the lower bound (\ref{KRRlowerHequal}) for $\delta=1$ cannot be improved in general.     
\end{remark}

\subsection{Discussion on the choice of $\lambda$}

In view of Theorems \ref{Coro:rate}, \ref{Coro:varrate} and \ref{Coro:biasrate} we may choose  $\lambda\asymp n^{-1}$ to balance the worst-case bias and the variance  when $\delta=\tau<1$. For $\delta=1$, the variance becomes $O(n^{-1})$, the parametric rate, regardless of the choice of $\lambda$. From Theorem \ref{Coro:rate}, a suitable choice of $\lambda$ in this case would be $n^{-{2m}/d}\lesssim\lambda\lesssim n^{-1}$. 
Note that this differs from $\lambda\asymp n^{-\frac{2m}{2m+d}}$,  the optimal order of magnitude of $\lambda$ for $\|\hat{f}-f\|_{L_2}$ to reach the minimax rate of convergence \citep{Stone80}.
Of course, we would also expect that the actual $|\operatorname{BIAS}_f|$ for a specific $f$ can be much smaller than the worst-case bias. 

Theorem \ref{Th:oBIAS} shows that $\operatorname{BIAS}$ decays faster than the rate indicated by Theorem \ref{Coro:rate} for fixed $f$.
\begin{theorem}\label{Th:oBIAS}
    If $f$ is fixed across all $n$ and $\lambda=o(1)$, under the conditions of Theorem \ref{Coro:rate}, $|\operatorname{BIAS}|=o_\mathbb{P}(\lambda^{\delta/2})$.
\end{theorem}

More explicit improved rates for $\operatorname{BIAS}$ are given in Section \ref{sec:further} of the Supplementary Materials under extra smoothness conditions of $f$. In view of these results, when $\lambda\asymp n^{-1}$ is used, the bias will become negligible compared with the variance term. This, however, may not be disadvantageous when the statistical inference is of interest. We will see in Section \ref{sec:AN} that the variance term is asymptotically normal. In this case, an asymptotically negligible bias enables us to construct an asymptotically unbiased confidence interval.

\subsection{Asymptotic Normality}\label{sec:AN}

In this section, we provide sufficient conditions under which the statistic $\langle \hat{f},g\rangle_\mathcal{H}$ is asymptotically normal. Because the bias is nonrandom given $X$, we only consider the asymptotic distribution of the variance term $g^T(X)(K(X,X)+\lambda n I)^{-1}E$. 
We use the notion ``$\xrightarrow{\mathscr{L}}$'' to denote the convergence in distribution.

\begin{theorem}\label{Th:AN}
Suppose $\sigma^2\in(0,\infty)$ is independent of $n$, and $g\neq 0$. The design points $X$ are either deterministic, or random but independent of the random error $E$.
    Under Assumptions \ref{assum:matern}-\ref{assum:tau}, we have the central limit theorem
   \begin{eqnarray}\label{CLT}
       \frac{1}{\sqrt{\operatorname{VAR}}}g^T(X)(K(X,X)+\lambda n I)^{-1}E\xrightarrow{\mathscr{L}} N(0,1),  \text{ as } n\rightarrow\infty,
   \end{eqnarray} 
   provided that $\lambda=o(1)$ and
    \begin{eqnarray}\label{conditionlambda}
        \lambda^{-1}=o\left(n^{\frac{2m}{d+2m(1-\delta/\tau)}}\right).
    \end{eqnarray}
    In particular, if $\delta=\tau$, (\ref{conditionlambda}) becomes $\lambda^{-1}=o(n^{\frac{2m}{d}})$.
\end{theorem}

Theorem \ref{Th:AN} conveys two important messages. First, $\lambda\asymp n^{-\frac{2m}{2m+d}}$, the optimal order of magnitude of $\lambda$ to reach the minimax rate of $\|\hat{f}-f\|_{L_2}$, always entails the asymptotic normality of the variance term. Second, if $\delta=\tau$, the variance term enjoys asymptotic normality for almost all choices of $\lambda$ under the assumption of Theorem \ref{Coro:rate}.

The asymptotic normality (\ref{CLT}) can be used to construct an asymptotic confidence interval for the ``biased true value'' $\mathbb{E}_E\langle \hat{f},g\rangle_\mathcal{H}$. In practice, more interest lies in building confidence intervals for the true value $\langle {f},g\rangle_\mathcal{H}$. This can be done if the bias is asymptotically negligible compared with the variance term. In view of Theorem \ref{Th:oBIAS}, when $\delta=\tau$, $\operatorname{BIAS}^2/{\operatorname{VAR}}\xrightarrow{p} 0$ as $n\rightarrow \infty$, under the choice $\lambda\asymp n^{-1}$. Suppose $\hat{\sigma}^2$ is a consistent estimate of $\sigma^2$, such as 
$\hat{\sigma}^2=\frac{1}{n}\sum_{i=1}^n(y_i-\hat{f}(x_i))^2$.
Then we can estimate $\operatorname{VAR}$ with
$\widehat{\operatorname{VAR}}=\hat{\sigma}^2g^T(X)(K(X,X)+\lambda n I)^{-2}g(X)$.
So the suggested $1-\alpha$ confidence interval for $\langle f,g\rangle_\mathcal{H}$ is
$\left[\langle \hat{f},g\rangle_\mathcal{H}-z_{\alpha/2}\sqrt{\widehat{\operatorname{VAR}}},\langle \hat{f},g\rangle_\mathcal{H}+z_{\alpha/2}\sqrt{\widehat{\operatorname{VAR}}},\right]$,
where $z_{\alpha/2}$ denotes the $\alpha/2$ upper quantile of the standard normal distribution.

\section{Examples}\label{sec:examples}
In this section, we present several examples to demonstrate the breadth of the proposed framework, including special cases of practical interest. 

\subsection{Point Evaluations}\label{sec:point}
Consider the point evaluation $l(f)=f(x_0)$ for some $x_0\in\Omega$. We have
\begin{eqnarray}\label{VARpoint}
    \operatorname{VAR}=\sigma^2 K(x_0,X)(K(X,X)+\lambda n I)^{-2}K(X,x_0).
\end{eqnarray}

We use the interpolation inequality (Theorem 3.8 of \cite{adams2003sobolev}; also see \cite{brezis2019sobolev} for non-integer $m$)
\begin{eqnarray}\label{interpolation}
    \|v\|_{L_\infty}\leq A\|v\|^{1-\frac{d}{2m}}_{L_2}\|v\|^{\frac{d}{2m}}_{H^m},
\end{eqnarray}
which holds for all $v\in H^m$ and some constant $A>0$, provided that $m>d/2$. Because $f(x_0)\leq \|f\|_{L_\infty}$, the interpolation inequality implies that Assumption \ref{assum:f} is true with $\delta=1-\frac{d}{2m}$. On the other hand, it can also be shown that $\tau=1-\frac{d}{2m}$ if $x_0$ is an interior point of $\Omega$.
Hence, we have the following result.

\begin{theorem}\label{Th:point}
    Suppose Assumptions \ref{assum:matern} and \ref{assum:norms} are true. Suppose $\lambda=o(1)$ and $\lambda^{-1}=o(n^{\frac{2m}{d}})$. Let $x_0$ be an interior point of $\Omega$ and $\operatorname{VAR}$ be as in (\ref{VARpoint}). Then, we have
    \begin{enumerate}
        \item $\operatorname{VAR}\asymp \sigma^2 n^{-1}\lambda^{-\frac{d}{2m}}$.
        \item $\sup_{\|f\|_\mathcal{H}\leq 1}\left|\mathbb{E}_E\hat{f}(x_0)-f(x_0)\right| \asymp \lambda^{\frac{1}{2}-\frac{d}{4m}}$.
        \item Regarding $\sigma^2$ as a positive constant, under the optimal order $\lambda\asymp n^{-1}$, 
        \[\sup_{\|f\|_\mathcal{H}\leq 1}|\hat{f}(x_0)-f(x_0)|\asymp n^{-\frac{1}{2}+\frac{d}{4m}}.\]
        \item In addition, if $\sigma^2>0$ and $\lambda=o(n^{-1})$, $(\operatorname{VAR})^{-\frac{1}{2}}(\hat{f}(x_0)-f(x_0))\xrightarrow{\mathscr{L}}N(0,1)$.
    \end{enumerate}
\end{theorem}

\begin{remark}
    For point evaluations of KRR, \cite{shang2013local,zhao2021statistical} obtained the rate of convergence and the asymptotic normality of the variance term, using a device called the functional Bahadur representation \citep{shang2010convergence}.
   
    The results presented in this work are under broader situations and weaker conditions: both random and deterministic designs are allowed, with wider ranges for $\lambda$ and $m$, and there is no uniform boundedness requirement for the eigenfunctions of the kernel. Besides, we give the order of magnitude of the worst-case bias together with the best order of magnitude of $\lambda$.
\end{remark}

\subsection{Derivatives}

Let $\alpha=(\alpha_1,\ldots,\alpha_d)^T\in\mathbb{N}^d$ be a multi-index and $|\alpha|=\alpha_1+\cdots+\alpha_d$. Denote $D^\alpha f=\frac{\partial^{|\alpha|}}{\partial \chi_1^{\alpha_1}\cdots\partial \chi_d^{\alpha_d}}f$ with $x=:(\chi_1,\ldots,\chi_d)^T$. Note that the zeroth order derivative stands for the identity mapping. (Thus, the point evaluation is a special case here.) The goal is to study the asymptotic properties of $D^\alpha \hat{f}(x_0)$ for $x_0\in\Omega$, as an estimator of $D^\alpha f(x_0)$. First, we have
\begin{eqnarray}\label{VARderivative}
    \operatorname{VAR}=\sigma^2 D^\alpha K(x_0,X)(K(X,X)+\lambda n I)^{-2}D^\alpha K(X,x_0),
\end{eqnarray}
where $D^\alpha K$ stands for the $\alpha$-th derivative of $K$ with respect to the first argument (or the second argument, as $K$ is symmetric.)
The Sobolev embedding theorem asserts that the linear operator $l(f)=D^\alpha f(x_0)$ is bounded provided that $m>d/2+|\alpha|$.
A different version of the interpolation inequality says that
\begin{eqnarray}\label{interpolationD}
    \|D^\alpha v\|_{L_\infty}\leq A\|v\|^{1-\frac{d+2|\alpha|}{2m}}_{L_2}\|v\|^{\frac{d+2|\alpha|}{2m}}_{H^m},
\end{eqnarray}
some constant $A>0$, provided that $m>d/2+|\alpha|$. This shows $\delta=1-\frac{d+2|\alpha|}{2m}$. Similarly, we have $\tau=1-\frac{d+2|\alpha|}{2m}$ for each interior point $x_0\in\Omega$, giving the following result.

\begin{theorem}\label{Th:derivative}
    Suppose Assumptions \ref{assum:matern} and \ref{assum:norms} are true, and $m>d/2+|\alpha|$. Suppose $\lambda=o(1)$ and $\lambda^{-1}=o(n^{\frac{2m}{d}})$. Let $x_0$ be an interior point of $\Omega$ and $\operatorname{VAR}$ be as in (\ref{VARderivative}). Then, we have
    \begin{enumerate}
        \item $\operatorname{VAR}\asymp \sigma^2 n^{-1}\lambda^{-\frac{d+2|\alpha|}{2m}}$.
        \item $\sup_{\|f\|_\mathcal{H}\leq 1}\left|\mathbb{E}_E D^\alpha \hat{f}(x_0)-D^\alpha f(x_0)\right| \asymp \lambda^{\frac{1}{2}-\frac{d+2|\alpha|}{4m}}$.
        \item Regarding $\sigma^2$ as a positive constant, under the optimal order $\lambda\asymp n^{-1}$, 
        \[\sup_{\|f\|_\mathcal{H}\leq 1}|D^\alpha \hat{f}(x_0)-D^\alpha f(x_0)|\asymp n^{-\frac{1}{2}+\frac{d+2|\alpha|}{4m}}.\]
        \item In addition, if $\sigma^2>0$ and $\lambda=o(n^{-1})$, $(\operatorname{VAR})^{-\frac{1}{2}}(D^\alpha \hat{f}(x_0)-D^\alpha f(x_0))\xrightarrow{\mathscr{L}}N(0,1)$.
    \end{enumerate}
\end{theorem}

Frequently, it is imperative to establish a multivariate central limit theorem for the variance term concerning various locations or partial derivatives. For example, the joint asymptotic normality of the gradient is needed in the example introduced in Section \ref{sec:nonlinear}.

Specifically, given locations $z_1,\ldots,z_{d_0}\in\Omega$ and multi-indices $\alpha_1,\ldots,\alpha_{d_0}\in \mathbb{N}^d$ for some $d_0\in\mathbb{N}_+$. Then the variance term of $D^\alpha_i \hat{f}(z_i)$ is
$D^{\alpha_i} K(z_i,X)(K+\lambda n I)^{-1}E$.
Thus the $d_0\times d_0$ covariance matrix of the vector of the variance terms is
\begin{eqnarray}\label{cov}
    \operatorname{COV}:=\left(\sigma^2D^{\alpha_i} K(z_i,X)(K+\lambda n I)^{-2}D^{\alpha_j}(X,z_j)\right)_{i,j}.
\end{eqnarray}
Theorem \ref{Th:multivariateAN} shows a multivariate central limit theorem for the variance term when $\alpha_i$'s are homogeneous, in the sense that $|\alpha_1|=\cdots=|\alpha_{d_0}|$.

\begin{theorem}\label{Th:multivariateAN}
Suppose Assumption \ref{assum:matern} is true. The covariance matrix $\operatorname{COV}$ defined in (\ref{cov}) is invertible with probability tending to one, provided that the pairs $(\alpha_1,z_1),\ldots,(\alpha_{d_0},z_{d_0})$ are distinct and $\sigma^2>0$.
    In addition, if Assumption \ref{assum:norms} is true, $|\alpha_1|=\cdots=|\alpha_{d_0}|=k$, $m>k+d/2$, and $z_i$'s are interior points of $\Omega$,  let $\lambda=o(1)$ and $\lambda^{-1}=o(n^{\frac{2m}{d}})$, then we have 
    \[\operatorname{COV}^{-\frac{1}{2}}
    \begin{pmatrix}
        D^{\alpha_1} K(z_1,X)\\
        \vdots\\
        D^{\alpha_{d_0}}K(z_{d_0},X)
    \end{pmatrix}(K+\lambda n I)^{-1}E\xrightarrow{\mathscr{L}}N(0,I),\]
\end{theorem}

\subsection{\texorpdfstring{$L_2$}{Lg}  Inner Products}

As shown in Proposition \ref{prop:L2}, if $\delta=1$, the linear functional $\langle g,\cdot\rangle_\mathcal{H}$ must be an $L_2$ inner product.

\begin{proposition}\label{prop:L2}
Suppose Assumption \ref{assum:matern} holds.
    If $g\in\mathcal{H}$ satisfies Assumption \ref{assum:f} with $\delta=1$, under Assumption \ref{assum:matern}, there exists a unique $h\in L_2$, such that $\langle g,v\rangle_{\mathcal{H}}=\langle h,v\rangle_{L_2}$ for each $v\in\mathcal{H}$.
\end{proposition}

Let $l(f)=\int_\Omega f(x)h(x)dx$. We have
\begin{eqnarray}\label{VARL2}
    \operatorname{VAR}=\int_\Omega\int_\Omega h(s) K(s,X)(K(X,X)+\lambda n I)^{-2} K(X,t) h(t)ds dt,
\end{eqnarray}
Set $\delta=\tau=1$. Corollary \ref{Coro:L2} follows immediately.

\begin{corollary}\label{Coro:L2}
    Suppose Assumptions \ref{assum:matern} and \ref{assum:norms} are true. Suppose $\lambda=o(1)$ and $\lambda^{-1}=o(n^{\frac{2m}{d}})$. Let $\operatorname{VAR}$ be as in (\ref{VARL2}). Then, we have
    \begin{enumerate}
        \item $\operatorname{VAR}\asymp \sigma^2 n^{-1}$.
        \item $|\int_\Omega (\hat{f}-f)(x)h(x)dx| =O_\mathbb{P}( \lambda^{\frac{1}{2}}\|f\|_\mathcal{H}+\sigma n^{-\frac{1}{2}})$.
        \item In addition, if $\sigma^2>0$ and $\lambda=o(n^{-1})$, $(\operatorname{VAR})^{-\frac{1}{2}}\int_\Omega (\hat{f}-f)(x)h(x)dx\xrightarrow{\mathscr{L}}N(0,1)$.
    \end{enumerate}
\end{corollary}

\begin{remark}
    \cite{tuo2015efficient} considered the $L_2$ inner product and demonstrated its impact on the calibration of computer models. The techniques adopted in \cite{tuo2015efficient} were available in much earlier literature to study the parametric part of smoothing splines and partial linear models. All these results show a root-$n$ rate of convergence and the asymptotic normality. The existing approach cannot deal with general $h\in L_2$, but under extra smoothness conditions of $h$, the theory gives the rate of convergence $O_\mathbb{P}(\lambda\|f\|_\mathcal{H}+\sigma n^{-1/2})$; see Section \ref{sec:semiparametric} of the Supplementary Materials for further discussion.
\end{remark}

\subsubsection{Expressions in terms of the Eigensystem}\label{sec:series}

A more abstract, but potentially general statement starts with an equivalent representation of $\mathcal{H}$ \citep{wendland2004scattered}. The discussion is deferred to Section \ref{sec: eigen_series} of the Supplementary Materials.

\section{Other applications of the linear functional theory}\label{sec:other}

Our theory of the linear functionals of KRR can be leveraged to handle other problems. Two prominent cases would be: 1) supremum over a set of linear functionals, e.g., the uniform error, and 2) nonlinear functionals that can be linearized asymptotically, e.g., the maximum point of a function. In this section, we outline our findings. The full technical details are deferred to Sections \ref{sec:uniformSup} and \ref{sec:nonlinearSup} of the Supplementary Materials.

\subsection{Uniform Bounds}\label{sec:uniform}

The methodology introduced in Section \ref{sec:improved} can be extended to study the uniform errors in terms of $\sup_{g\in\mathscr{G}}|\langle \hat{f}-f,g\rangle_\mathcal{H}|$. We are particularly interested in the uniform error of the partial derivatives, i.e.,
\begin{eqnarray}\label{uniform}
    \sup_{x\in\Omega}\left|D^\alpha \hat{f}(x)-D^\alpha f(x)\right|,
\end{eqnarray}
for some $\alpha\in\mathbb{N}^d$. Note that (\ref{uniform}) includes the $L_\infty$ error by setting $\alpha=0$. Following the idea in Section \ref{sec:BnV}, we break (\ref{uniform}) into two terms.
\begin{eqnarray}
  (\ref{uniform}) \leq \sup_{x\in\Omega}\left|\mathbb{E}_E D^\alpha \hat{f}(x)-D^\alpha f(x)\right|+\sup_{x\in\Omega}\left|D^\alpha \hat{f}(x)-\mathbb{E}_E D^\alpha \hat{f}(x)\right|.\label{uniformBnV}
\end{eqnarray}
With some abuse of terminology, we call the first term in (\ref{uniformBnV}) the \textit{uniform bias} and the second term the \text{uniform variance term}.

Our analysis shows the upper bound for the uniform bias
\begin{eqnarray}\label{uniformbiasbound}
    \text{uniform bias}=O_\mathbb{P}(\lambda^{\frac{1}{2}-\frac{d+2|\alpha|}{4m}}\|f\|_\mathcal{H}),
\end{eqnarray}
which is attainable in the worst-case scenario. The magnitude of the variance term would depend on the random noise's tail property. When the noise has a sub-Gaussian tail, i.e., $\mathbb{E}\exp\{\vartheta e_1\}\leq \exp\{\vartheta^2\varsigma^2/2\}$ for all $\vartheta\in \mathbb{R}$ and some $\varsigma^2>0$, we have the bound
\begin{eqnarray}\label{uniformvariancebound}
    \text{uniform variance term}=O_\mathbb{P}\left(\varsigma n^{-\frac{1}{2}}\lambda^{-\frac{d+2|\alpha|}{4m}}\sqrt{\log\left(\frac{C}{\lambda}\right)}\right).
\end{eqnarray}
Compared with the pointwise bound given by Theorem \ref{Th:derivative}, (\ref{uniformvariancebound}) is inflated only by a logarithmic factor $\sqrt{\log(C/\lambda)}$. This factor cannot be improved in general, as the bound is shown to be sharp when the noise follows a normal distribution.

The bias and variance terms in (\ref{uniformbiasbound}) and (\ref{uniformvariancebound}) can be balanced by choosing $\lambda\sim n^{-1}\log n$ which is independent of $m,d$, and $\alpha$,
and the resulting rate of convergence is
\begin{eqnarray}\label{minimax}
    \sup_{x\in\Omega}\left|D^\alpha \hat{f}(x)-D^\alpha f(x)\right|=O_\mathbb{P}\left((n^{-1}\log n)^{\frac{1}{2}-\frac{d+2|\alpha|}{4m}}\right).
\end{eqnarray}

\begin{remark}\label{remark:uniform}
  
The rate of convergence shown in (\ref{minimax}) matches the classic  $L_\infty$ minimax rate.  \cite{nemirovski2000topics} demonstrates that, under grid-based designs, the lower bounds for the minimax risk under the   $L_\infty$  norm of $D^\alpha \hat{f}(x)-D^\alpha f(x)$ in  a unit ball of a Sobolev space with smoothness m, as stated in Theorem 2.1.1, is $(n/\log n)^{\frac{1}{2}-\frac{2|\alpha|+d}{4m}}$.  
\end{remark}

\subsection{A Nonlinear Problem}\label{sec:nonlinear}
Although this work primarily focuses on linear functionals of $f$, the results can help study certain nonlinear functionals if they can be linearized. In this section, we consider the nonlinear functionals
$\min_{x\in\Omega} f(x) \text{ and } \operatorname*{argmin}_{x\in\Omega}f(x)$.
Consider the plug-in estimators of $\min_{x\in\Omega} f(x)$ and $\operatorname*{argmin}_{x\in\Omega}f(x)$, defined as 
$\hat{f}_{\min}:=\min_{x\in\Omega} \hat{f}(x) \text{ and } \hat{x}_{\min}:=\operatorname*{argmin}_{x\in\Omega}\hat{f}(x),$
respectively. 
To linearize $\hat{x}_{\min}-x_{\min}$, intuitively, we use a Taylor expansion argument
$        0=\frac{\partial\hat{f}}{\partial x}(\hat{x}_{\min})\approx\frac{\partial\hat{f}}{\partial x}(x_{\min})+\frac{\partial^2\hat{f}}{\partial x\partial x^T}(x_{\min})(\hat{x}_{\min}-x_{\min}),
$
    which implies
    $\hat{x}_{\min}-x_{\min}\approx -H^{-1}\frac{\partial\hat{f}}{\partial x}(x_{\min})$.
This inspires us to consider the linear functional $l(\hat{f}-f)=\frac{\partial (\hat{f}-f)}{\partial x}(x_{\min})$. The covariance matrix of the variance term is
\begin{eqnarray}\label{Covnonlinear}
    \operatorname{COV}=\sigma^2\frac{\partial K}{\partial x}(x_{\min},X)(K(X,X)+\lambda n I)^{-2}\frac{\partial K}{\partial x^T}(X,x_{\min}).
\end{eqnarray}
Because both $H$ and $\operatorname{COV}$ contain unknown parameters, we consider estimators
\begin{eqnarray}
    \hat{H}&:=&\frac{\partial^2 \hat{f}}{\partial x\partial x^T}(\hat{x}_{\min}),\\
    \widehat{\operatorname{COV}}&:=&\hat{\sigma}^2\frac{\partial K}{\partial x}(\hat{x}_{\min},X)(K(X,X)+\lambda n I)^{-2}\frac{\partial K}{\partial x^T}(X,\hat{x}_{\min}),\label{Covestimate}
\end{eqnarray}
where $\hat{\sigma}^2$ is a consistent estimator of ${\sigma}^2$. 

Under the optimal tuning parameter $\lambda\asymp n^{-1}$, we show that
    \begin{enumerate}
        \item $\|\hat{x}_{\min}-x_{\min}\|=O_\mathbb{P}(n^{-\frac{1}{2}+\frac{d+2}{4m}}),f(\hat{x}_{\min})-f(x_{\min})=O_\mathbb{P}(n^{-1+\frac{d+2}{2m}})$;
        \item $\widehat{\operatorname{COV}}^{-\frac{1}{2}}\hat{H}(\hat{x}_{\min}-x_{\min})\xrightarrow{\mathscr{L}} N(0,I)$.
    \end{enumerate}

\section{Numerical Studies}\label{sec:NS}

In this section, we conduct numerical studies to examine both the pointwise asymptotic confidence interval (CI) for the estimated optimal point $\hat{x}_{\min}$ and the finite-sample coverage probability of the proposed derivative estimator. We begin by evaluating the performance of the proposed estimator for estimating the optimal point using both a toy example and real data, focusing on the accuracy of the pointwise CIs for $\hat{x}_{\min}$.
Next, we compare the finite-sample coverage probability of the proposed derivative estimator with several alternative methods in a toy example. The results provide numerical evidence supporting the theoretical asymptotic properties of the proposed estimator.

\subsection{Asymptotic Confidence Interval for Optimal Point}\label{subsection:optimal_point}
We conduct numerical studies to examine the pointwise asymptotic CI for the estimated optimal point $\hat{x}_{min}$ in the objective function. Three test regression functions are considered:
\begin{enumerate}
 \setlength\itemsep{3mm}
 \item $f_{1}(x)=1.8[\beta_{10,5}(x)+\beta_{7,7}(x)+\beta_{5,10}(x)]$,
 \item $f_{2}(x)=2.4\beta_{30,17}(x)+2.8\beta_{4,11}(x)$,
\item $f_{3}(x)=\frac{7}{5}\beta_{15,30}(x)+8\sin(32\pi x -\frac{4\pi}{3})-6\cos(16\pi x)-\frac{1}{5}\cos(64\pi x)$,
\end{enumerate}
where $\beta_{a,b}(x)$ stands for the density function of a Beta($a,b$) distribution. In all cases, we generate independent and identically distributed input data $X$ from the uniform distribution over $[0,1]$. The response $y$ is given by model (\ref{model}) after adding an independent and identically distributed noise. Two types of noise distributions are used: the normal distribution with a variance of 3 and the student's $t$-distribution with degrees of freedom $\nu=3$. Each distribution type is used under the mean zero and two different variance ($\sigma^2$) levels.  

In all simulation experiments, we choose the Mat\'ern kernel with $\nu=3$ and choose both its hyperparameters and the regularization parameter $\lambda$, where $\lambda$ is set near the order of $O(n^{-1})$, through cross-validation. We then construct CIs for each $\hat{x}_{\min}$ at a 95\% nominal level following the result in Section \ref{sec:nonlinear}. The coverage probability (CP) is estimated as the proportion of the CIs that cover the true value in a total of 800 replications. In addition, we present the Q-Q plots of the test statistics $\hat{x}_{min}$ to visualize their empirical distributions versus the normal distributions. The test functions are plotted as solid curves in Figure \ref{fig:testfun_plot} in the supplementary material. As shown in the plots, all three test functions are smooth, but have an increasing number of local optimal points. 

\begin{figure}
\centering
\begin{subfigure}{.37\hsize}
\centering
\includegraphics[width=1.1\linewidth,height=1.8\linewidth]{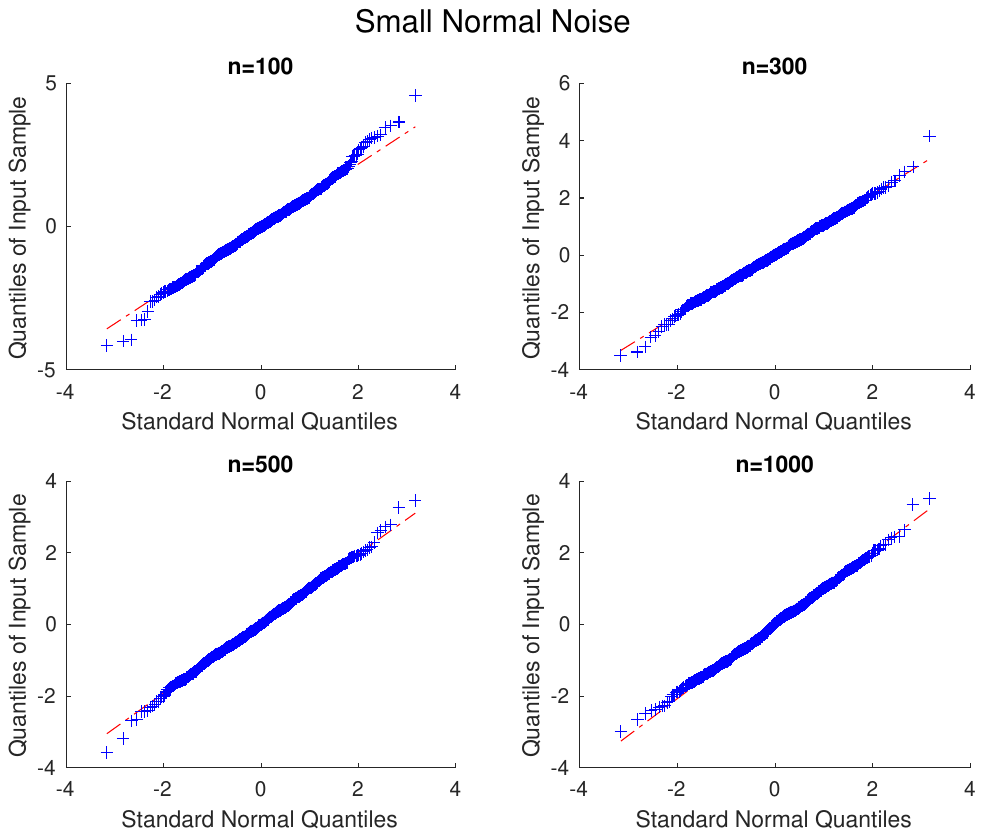}
\end{subfigure}%
\begin{subfigure}{.37\hsize}
\includegraphics[width=1.1\linewidth,height=1.8\linewidth]{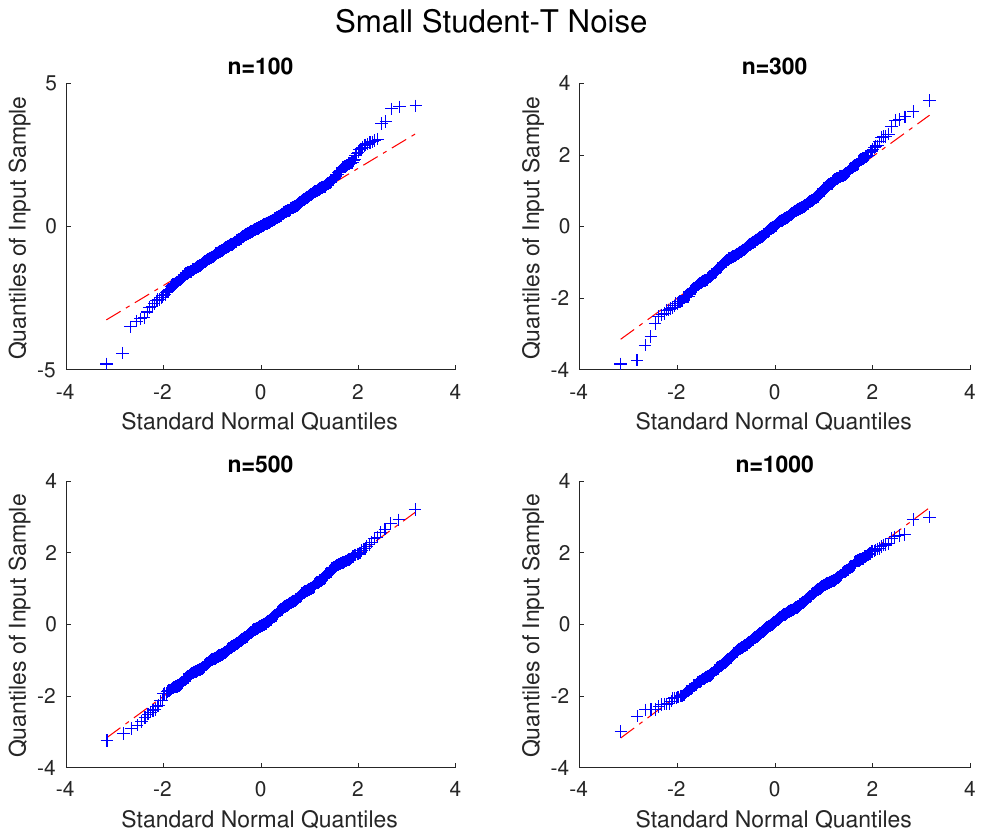}
\end{subfigure}
 \vspace*{-26mm}
\caption{Results for Test Function $f_1$ with low-level noise $\sigma=0.5$ }
\label{fig:test1_small}
\end{figure}

\begin{figure}
\centering
\begin{subfigure}{.37\hsize}
  \centering
  \includegraphics[width=1.1\linewidth,height=1.8\linewidth]{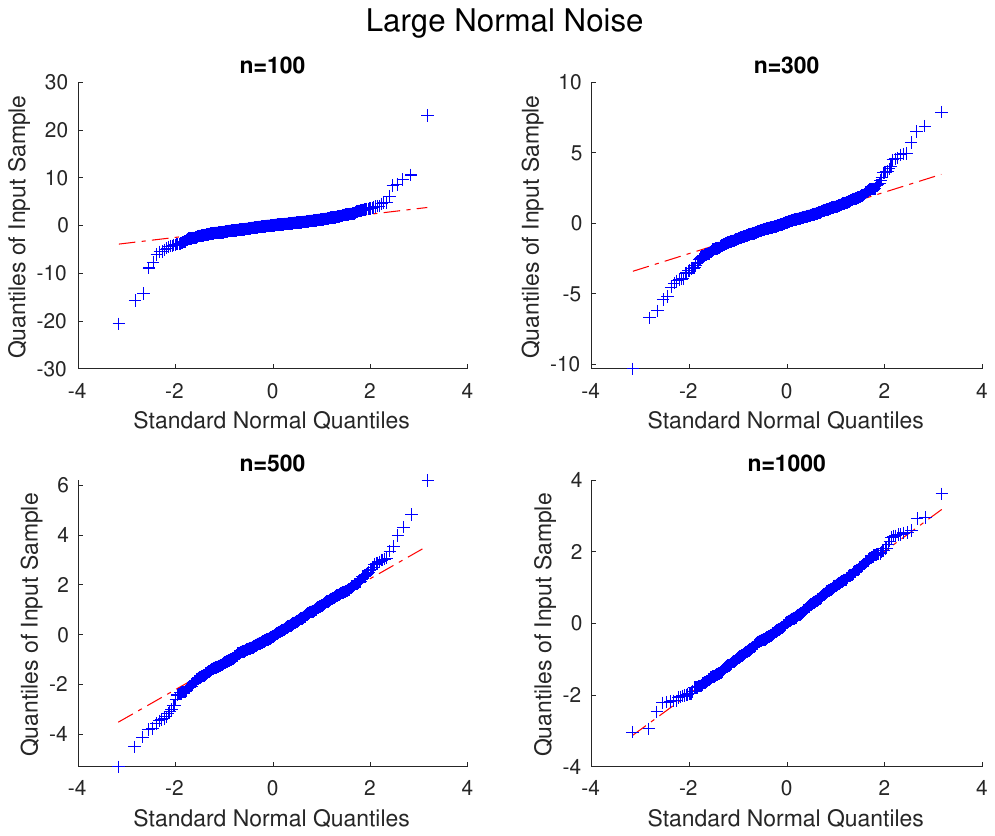}
\end{subfigure}%
\begin{subfigure}{.37\hsize}
  \centering
\includegraphics[width=1.1\linewidth,height=1.8\linewidth]{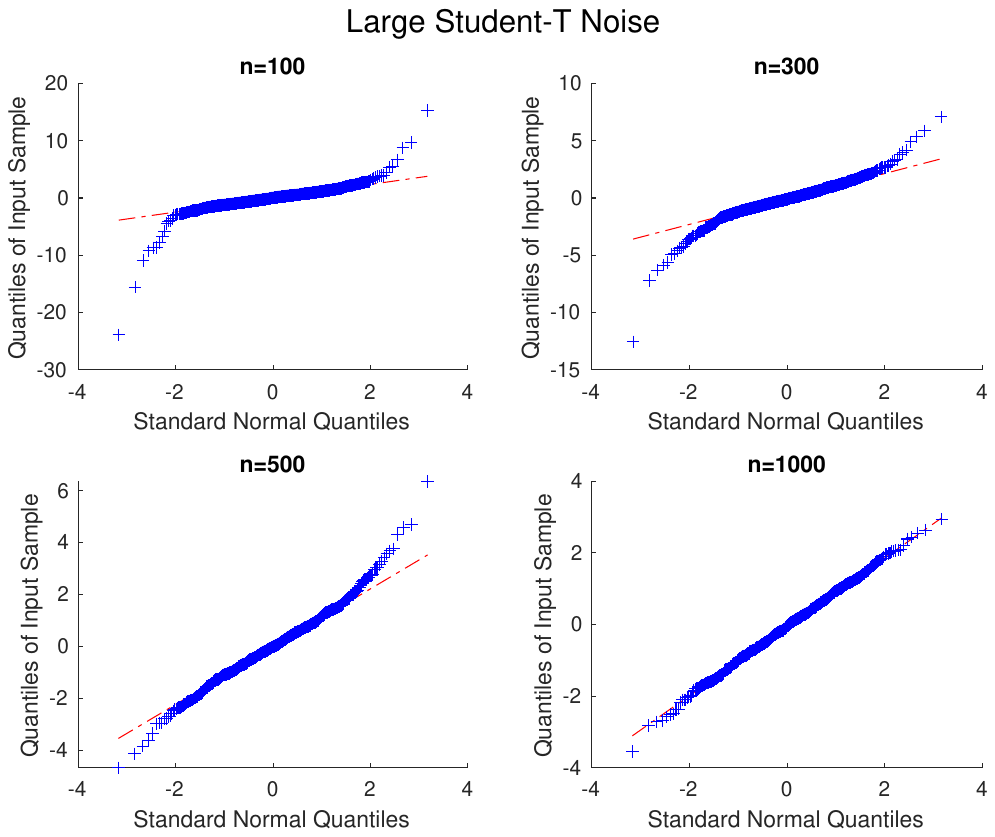}
\end{subfigure}
\vspace*{-26mm}
\caption{Results for Test Function $f_1$ with high-level noise $\sigma=5$ }
\label{fig:test1_large}
\end{figure}

\begin{figure}
\centering
\begin{subfigure}{.37\hsize}
  \centering
\includegraphics[width=1.1\linewidth,height=1.8\linewidth]{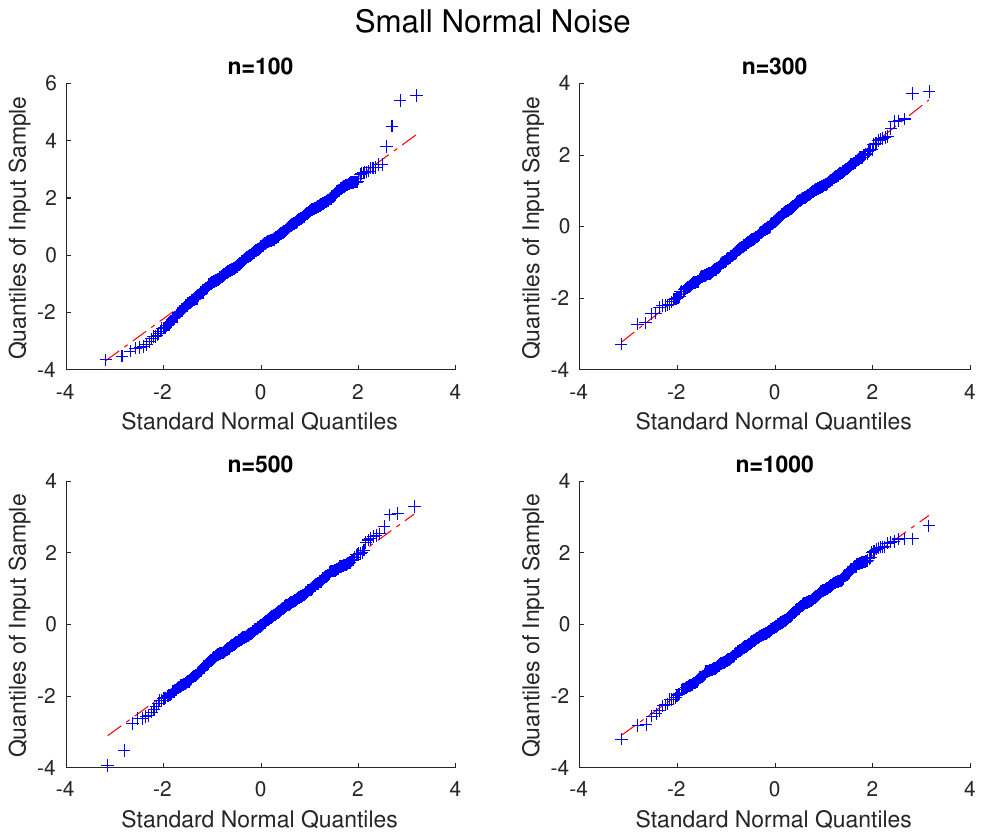}
\end{subfigure}%
\begin{subfigure}{.37\hsize}
  \centering
\includegraphics[width=1.1\linewidth,height=1.8\linewidth]{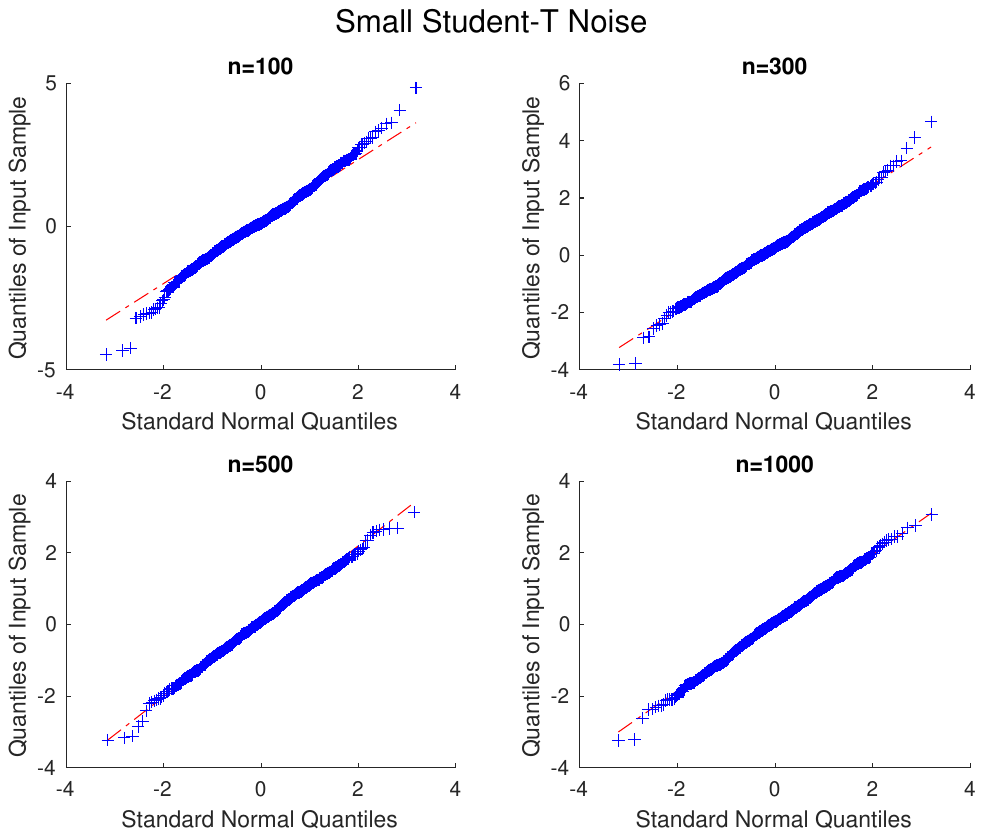}
\end{subfigure}
\vspace*{-26mm}
\caption{Results for Test Function $f_2$ with low-level noise $\sigma=0.5$ }
\label{fig:test2_small}
\end{figure}

\begin{figure}
\centering
\begin{subfigure}{.37\hsize}
  \centering
\includegraphics[width=1.1\linewidth,height=1.8\linewidth]{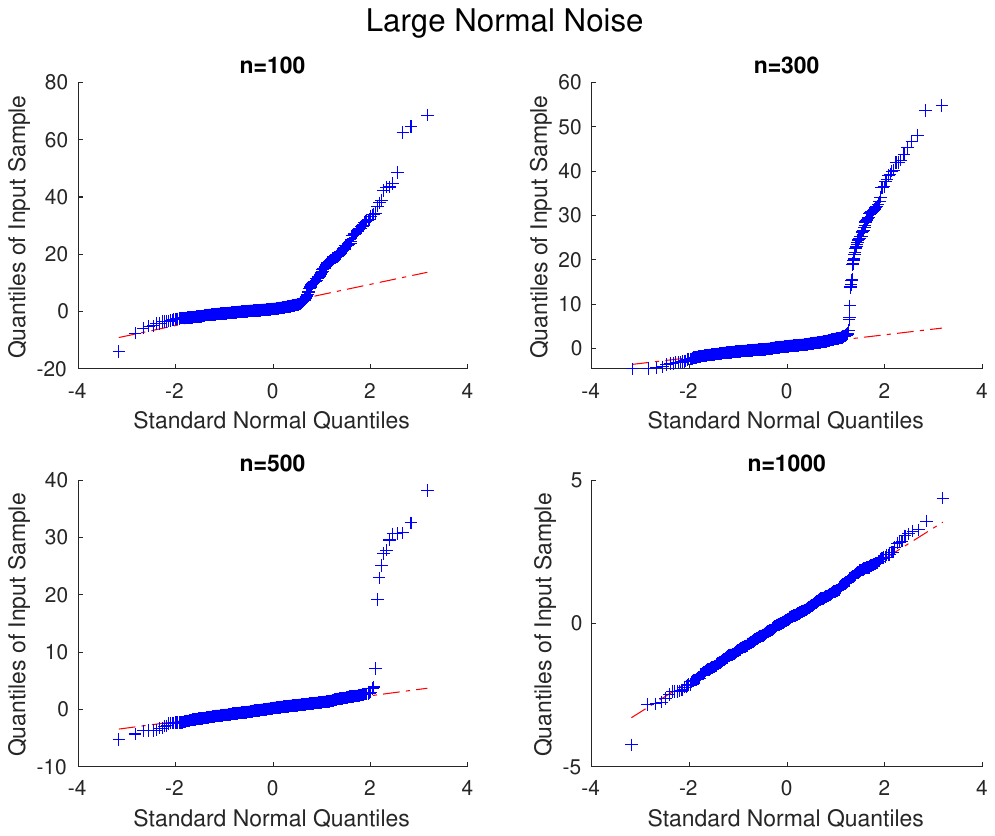}
\end{subfigure}%
\begin{subfigure}{.37\hsize}
  \centering
  \includegraphics[width=1.1\linewidth,height=1.8\linewidth]{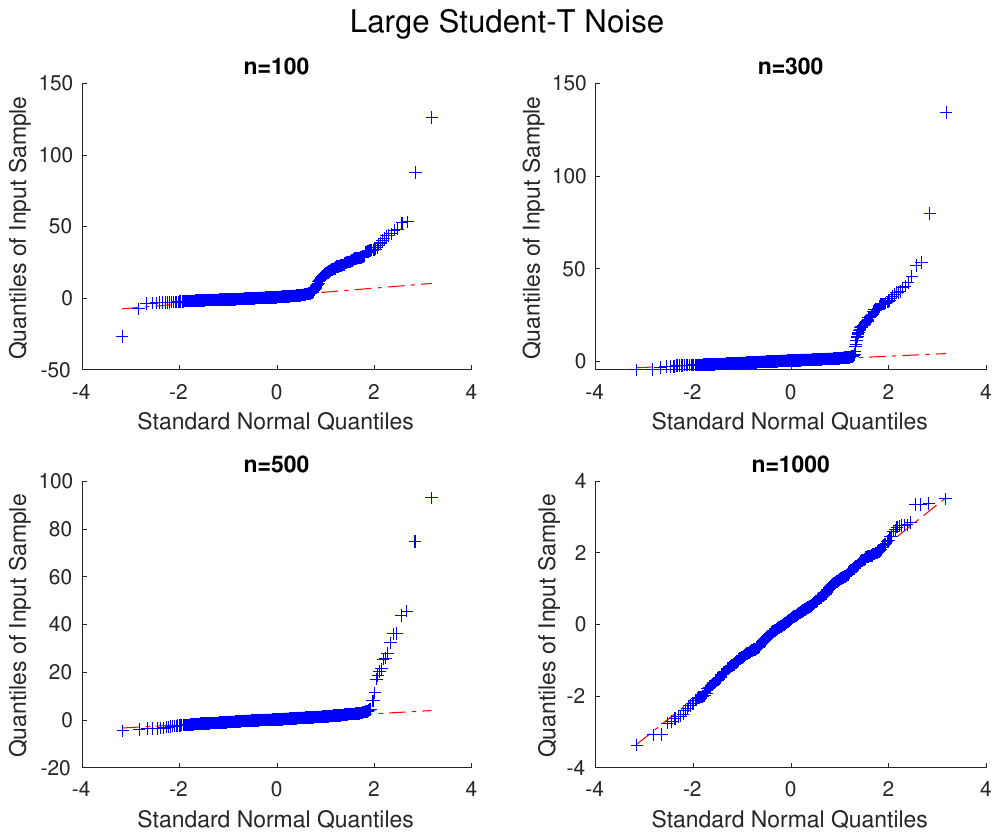}
\end{subfigure}
\vspace*{-26mm}
\caption{Results for Test Function $f_2$ with high-level noise $\sigma=5$ }
\label{fig:test2_large}
\end{figure}

\begin{figure}
\centering
\begin{subfigure}{.37\hsize}
  \centering
  \includegraphics[width=1.1\linewidth,height=1.8\linewidth]{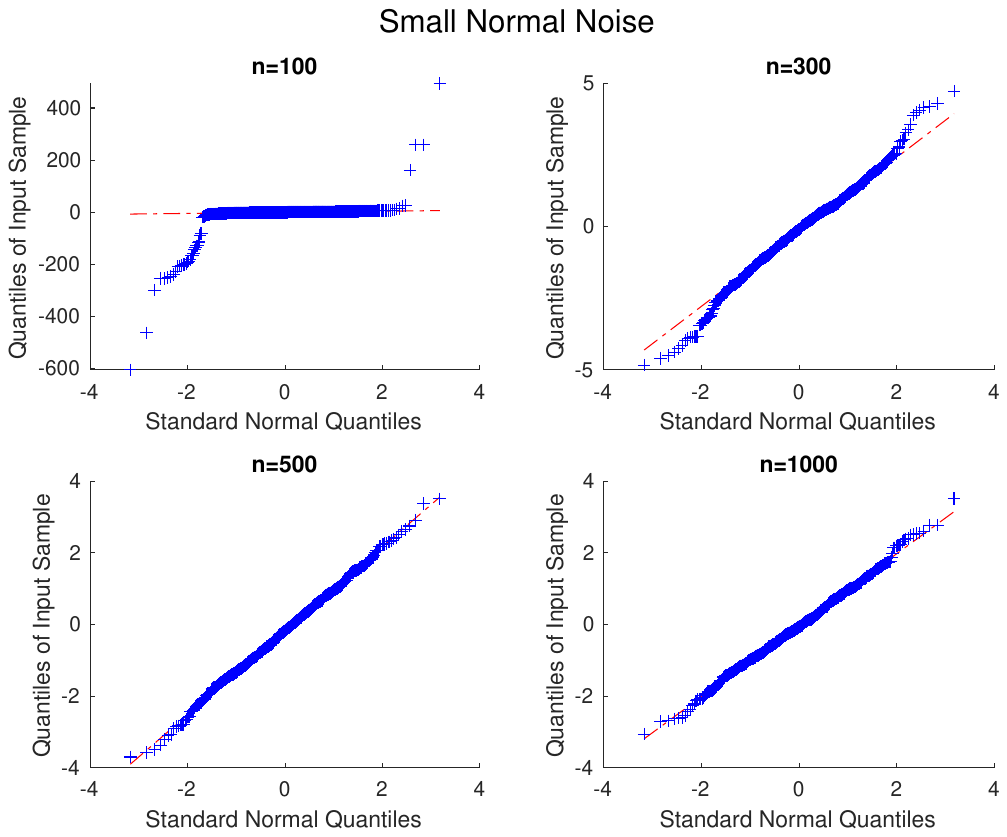}
\end{subfigure}%
\begin{subfigure}{.37\hsize}
  \centering
  \includegraphics[width=1.1\linewidth,height=1.8\linewidth]{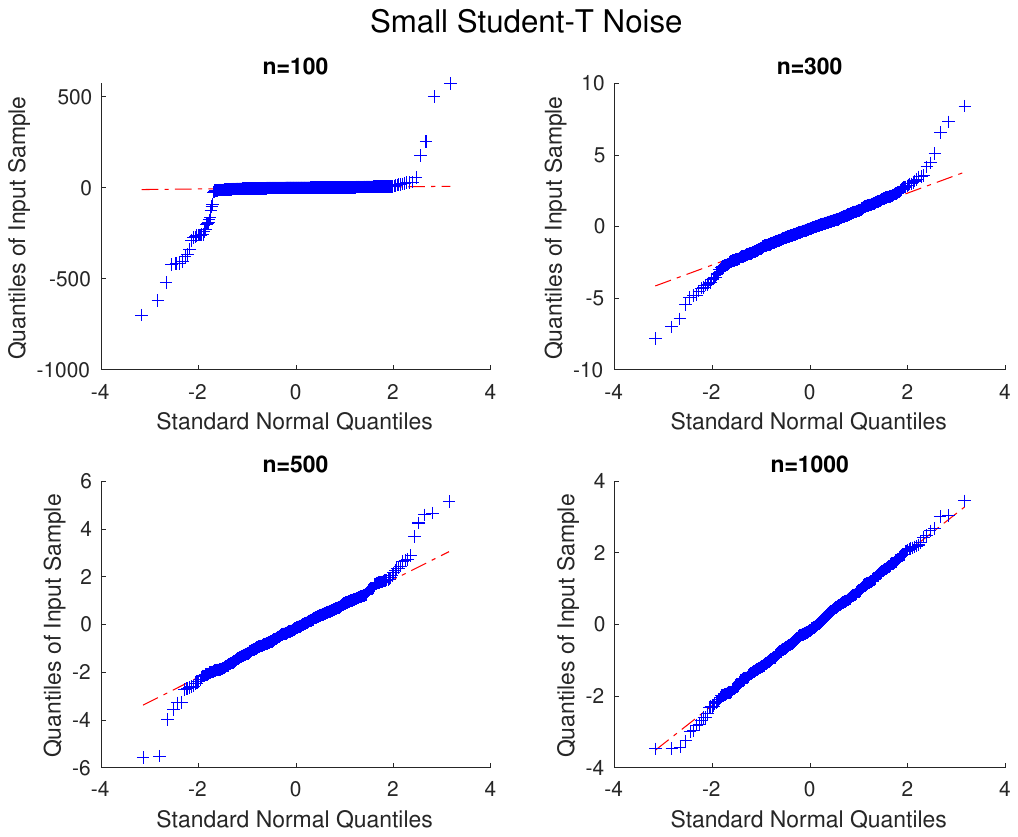}
\end{subfigure}
\vspace*{-26mm}
\caption{Results for Test Function $f_3$ with low-level noise $\sigma=0.5$ }
\label{fig:test3_small}
\end{figure}

\begin{figure}
\centering
\begin{subfigure}{.37\hsize}
  \centering
  \includegraphics[width=1.1\linewidth,height=1.8\linewidth]{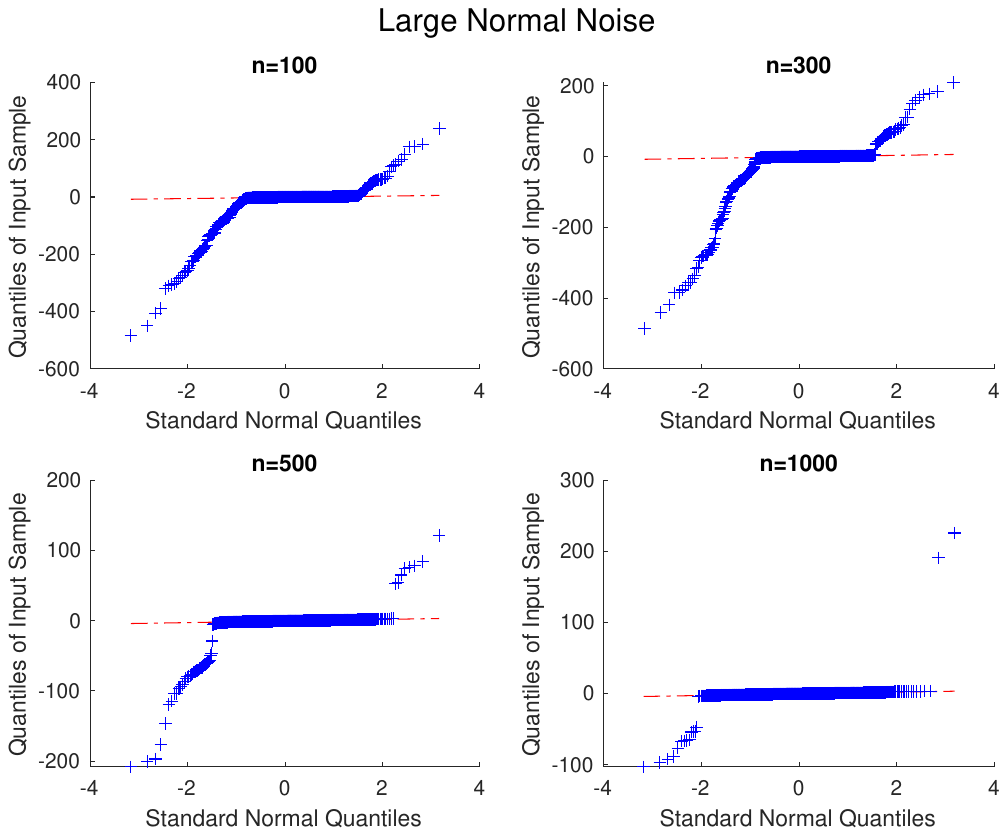}
\end{subfigure}%
\begin{subfigure}{.37\hsize}
  \centering
  \includegraphics[width=1.1\linewidth,height=1.8\linewidth]{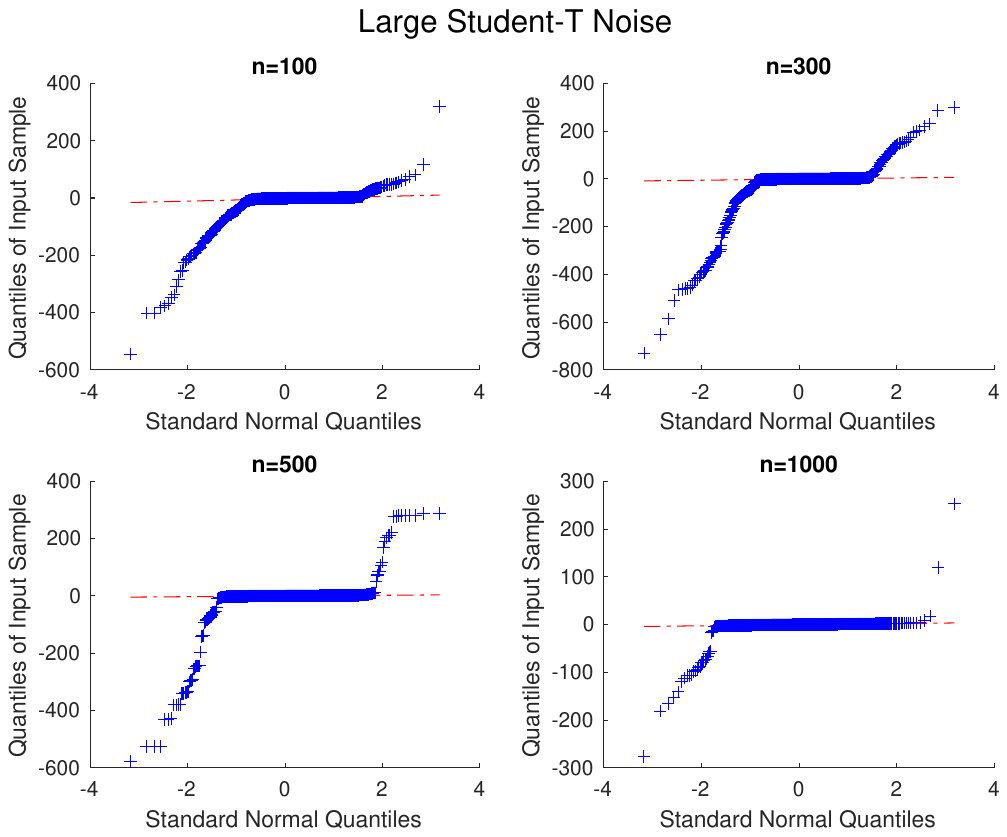}
\end{subfigure}
\vspace*{-26mm}
\caption{Results for Test Function $f_3$ with high-level noise $\sigma=5$ }
\label{fig:test3_large}
\end{figure}


\begin{table}
\centering
\begin{tabular}{ccccccc}
    \toprule
        & \multicolumn{6}{c}{Coverage Probability under Normal Noise with $\alpha=0.05$}                                              \\
        \cmidrule{2-7}
        & \mcc{f_1} & \mcc{f_2} & \mcc{f_3}     \\
    \cmidrule(r){2-3}\cmidrule(lr){4-5}\cmidrule(l){6-7}
\multicolumn{1}{c}{$n$} & $\sigma=0.5$ & $\sigma=5$ & $\sigma=0.5$ & $\sigma=5$ &$\sigma=0.5$ & $\sigma=5$ \\
    \midrule
100      & 0.9031 & 0.8010 & 0.8452     & 0.5872 &  0.5968     &  0.5978               \\
300   & 0.9317 &  0.8304   & 0.9178    & 0.7665   & 0.8386    & 0.6223               \\
500   & 0.9533 & 0.8821     & 0.9398   & 0.8415   & 0.9118   & 0.8344                 \\
1000   & 0.9543 & 0.9412    &  0.9577   & 0.9205    & 0.9441     & 0.8898               \\
1500   & 0.9573  & 0.9532   &  0.9470   & 0.9389   & 0.9407     & 0.9382                \\
    \bottomrule
\end{tabular}
\caption{\label{table:normal_cp} Estimated Coverage Probability for Normal Distributed Noise.}
    \end{table}

\begin{table}
\centering
\begin{tabular}{ccccccc}
    \toprule
        & \multicolumn{6}{c}{Coverage Probability under $t_3$ Noise with $\alpha=0.05$}                                              \\
        \cmidrule{2-7}
        & \mcc{f_1} & \mcc{f_2} & \mcc{f_3}     \\
    \cmidrule(r){2-3}\cmidrule(lr){4-5}\cmidrule(l){6-7}
\multicolumn{1}{c}{$n$} & $\sigma=0.5$ & $\sigma=5$ & $\sigma=0.5$ & $\sigma=5$ &$\sigma=0.5$ & $\sigma=5$ \\
    \midrule
100   &  0.9005 &  0.8101  & 0.8801   & 0.6006 &   0.5114    &  0.5578               \\
300   & 0.9329  &  0.8412  & 0.9217   & 0.7912   & 0.8359   &  0.5976            \\
500   & 0.9532  &  0.8897  & 0.9470   & 0.8584   & 0.9295   &  0.7716                 \\
1000  & 0.9402  &  0.9509  & 0.9501   &  0.9142    &  0.9310   & 0.8475            \\
1500  & 0.9472  & 0.9417   & 0.9629   &  0.9401   & 0.9389     & 0.9293                \\
    \bottomrule
\end{tabular}
\caption{ Estimated Coverage Probability for Student's-t Distributed Noise.}\label{table:student_cp}
\end{table}

Tables \ref{table:normal_cp} and \ref{table:student_cp} summarize the CP of our asymptotic CI over 800 replications. Tables \ref{table:normal_cp} and \ref{table:student_cp} imply that in the first two cases, the proposed asymptotic confidence intervals provide decent coverage rates (i.e., close to the nominal level $95\%$) for both functions, regardless of the type of the error distribution. For Case 3, we suffer from the under-coverage problem in high noise scenarios, KRR cannot accurately reconstruct the function and thus pinpoint the global minimum point. But such a problem is mitigated when the sample size is sufficiently large: when $n = 1500$, the proposed asymptotic CI has a CP close to 0.95. 

Figures \ref{fig:test1_small}-\ref{fig:test3_large} 
present the Q-Q plots of the aforementioned statistics over the replications. As shown in Figures \ref{fig:test1_small} and \ref{fig:test2_small}, when the error variance is small, the distribution of statistical quantities corresponding to two different error distributions is close to the normal distribution even under small sample sizes. However, in Case 3 with small noise, the statistical values associated with the normal distribution error closely align with the normal distribution under small sample sizes, in contrast to those associated with the $t$-distribution error. Nevertheless, as sample size increases, the statistics corresponding to both error distributions progressively approach the normal distribution. When the error variance is relatively large, as observed in Figures \ref{fig:test1_large}, \ref{fig:test2_large}, and \ref{fig:test3_large}, the Q-Q plots for both types of error distribution exhibit an S-shape, indicating that the statistics' distribution has heavier tails than the normal distribution, especially with a sample of less than 500. In particular, as demonstrated in Figure \ref{fig:test3_large},  the statistics with both the $t$-distributed errors and normally distributed errors severely  deviate from a normal distribution even under a sample size of 1000. As said before, this deviation is mainly due to the large \textit{uniform estimation errors}, so we cannot correctly pinpoint which local optimal is the global optimal. Nevertheless, as exhibited in Table \ref{table:normal_cp} and Table \ref{table:student_cp}, the coverage rates of the test statistics associated with a normal distribution are slightly better than those with $t$-distributed errors across all sample sizes.
In view of the different simulation results led by the noise distribution, these results support our hypothesis in Remark \ref{remark:uniform} that the uniform rate of convergence of KRR depends on the tail property of the random noise.

In summary, the simulation results show that the asymptotic confidence interval for the optimal point generally aligns with our asymptotic analysis. The CP uniformly approaches the desired confidence level as the sample size grows, showing the validity of the intervals. In addition, the resulting confidence intervals are not sensitive to the error distribution.

\subsection{Real Data Analysis}\label{subsec:data}

Event-related potentials (ERPs) are electroencephalogram (EEG) signals recorded in response to external stimuli, and the amplitude and latency of their characteristic waveform components are well known to reflect sensory and cognitive processes. For our real-data analysis, we use a publicly available ERP dataset (\url{http://dsenturk.bol.ucla.edu/supplements.html}) consisting of recordings from a single participant diagnosed with autism spectrum disorder (ASD) under one electrode and one experimental condition. The dataset contains 72 trials, each with 250 time points. Our study targets two well-established ERP components—N1, typically occurring between 100 and 250, and P3, between 190 and 370—both of which have been extensively investigated for their links to sensory and cognitive function. To capture both components, we restrict the analysis to the [100,370]. We then apply our method to construct confidence intervals for the optimal point of these component latencies, providing a calibrated assessment of their estimation uncertainty.

The aim is to estimate the optimal maximum values of the ERP signal, specifically the peak latencies of the N1 and P3 components, within the time window [100, 370]. Since EEG signals are inherently noisy, neuroscientists traditionally average the signals across trials to obtain a grand average ERP waveform. This averaged waveform is then used to estimate the amplitude and latency of the ERP components. The optimal points are estimated based on these averaged waveforms. In the supplementary material, Figure~\ref{fig:erp_data} plots the 72 individual ERP trial waveforms together with their grand average, with two vertical lines indicating the time window used as the search region for estimating the optimal point.

Figure~\ref{fig:real_data_uq} displays the Q–Q plot of the optimal point estimates for the real ERP data, showing close agreement between the empirical and theoretical quantiles. The empirical coverage rate of the 95\% confidence intervals is 0.948, consistent with the nominal level and indicating that the intervals effectively capture the true optimal points.

\begin{figure}
    \centering
    \resizebox{0.45\linewidth}{!}{%
        \includegraphics{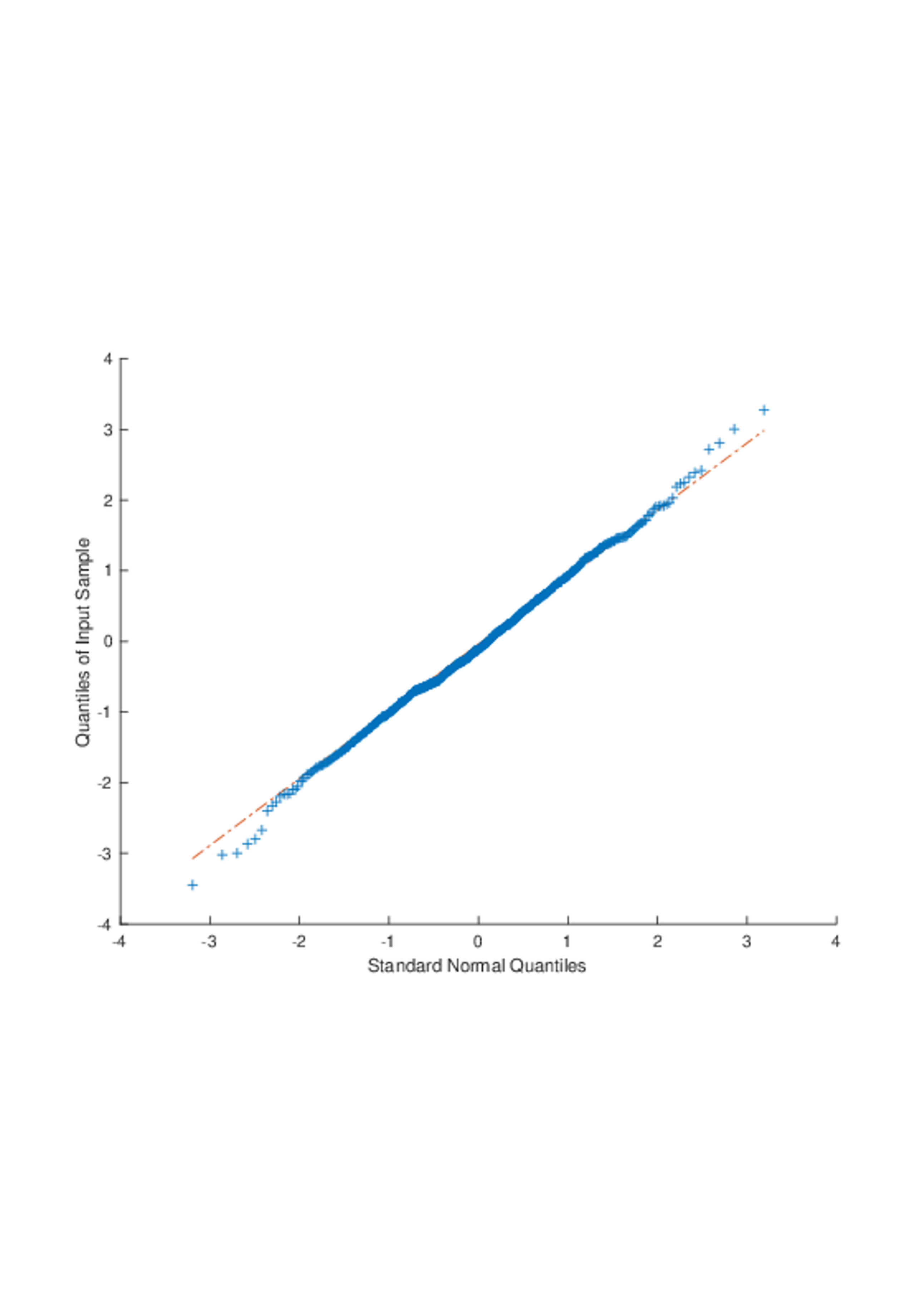}
    }
    \caption{Q-Q Plot of Optimal Point Estimations for Real ERP Data}
    \label{fig:real_data_uq}
\end{figure}

\subsection{ Comparison with existing Methods for Derivative Estimator}

We consider two regression functions: 

\begin{enumerate}
 \setlength\itemsep{3mm}
 
 \item $f_{4}(x)=5\exp{(-2(1 - 2x)^2)} (1 - 2x)$, with $x\in[0,1]$.
 
\item $f_{5}(x)=\sin(8.5x)+\cos(8.5x)+\log(2+x)$, with $x\in[-1,1]$.
\end{enumerate}

Random design points from the uniform distributions over the designated intervals are used with sample size $n = 500$.  The response $y$ is given by model (\ref{model}) after adding an independent and identically distributed Gaussian noise $\epsilon_i\sim N(0,2^2)$.


We consider  the first order derivative to accommodate competing methods, but note that the proposed method is readily available for any order.  We  construct a CI for each $\hat{f}^{\prime}(x)$ with a $95\%$ nominal level by applying Theorem \ref{Th:derivative}. The CP is estimated as the proportion of the CIs that cover the true value in a total of 800 replications.  For the plug-in KRR estimator, we adopt the same simulation setting as described in Section~\ref{subsection:optimal_point}. We compare the plug-in KRR estimator with three other methods: local polynomial regression with degree $p = 4$ (R package \texttt{nprobust} in \citep{calonico2019nprobust}, denoted as \texttt{locpol4} in the figures), smoothing spline (R package \texttt{lspartition} in \citep{cattaneo2020lspartition}) with higher-order-basis bias correction (denoted as \texttt{bspline1}) and with least squares bias correction (denoted as \texttt{bspline2}). For more details of the bias correction estimator, please refer to \citep{calonico2022coverage}.  

\begin{figure}[htbp]
    \centering
    \begin{subfigure}[b]{0.48\textwidth}
        \centering
        \includegraphics[width=\textwidth]{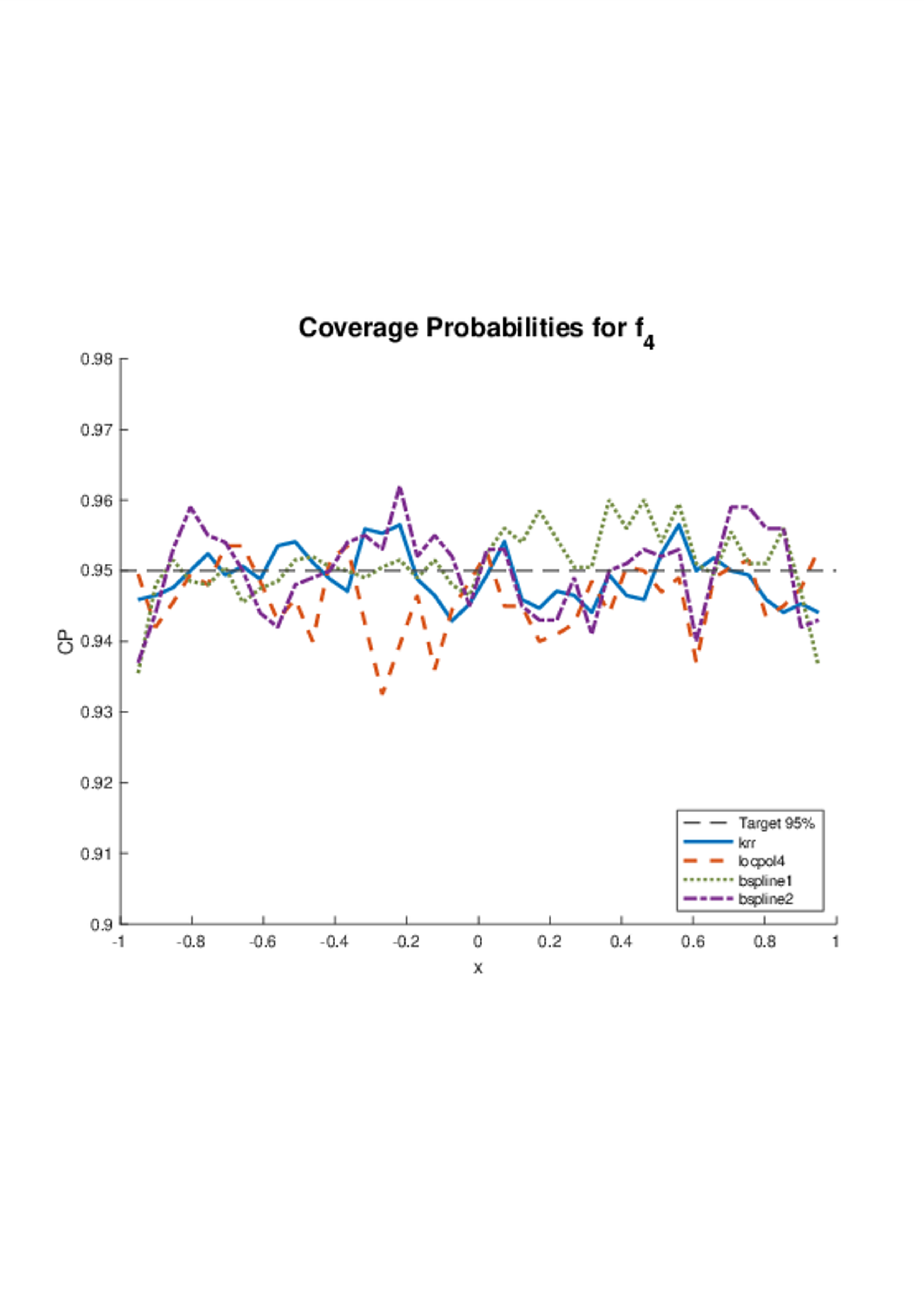}
    
        \label{fig:f4_der}
    \end{subfigure}
    \hfill
    \begin{subfigure}[b]{0.48\textwidth}
        \centering
        \includegraphics[width=\textwidth]{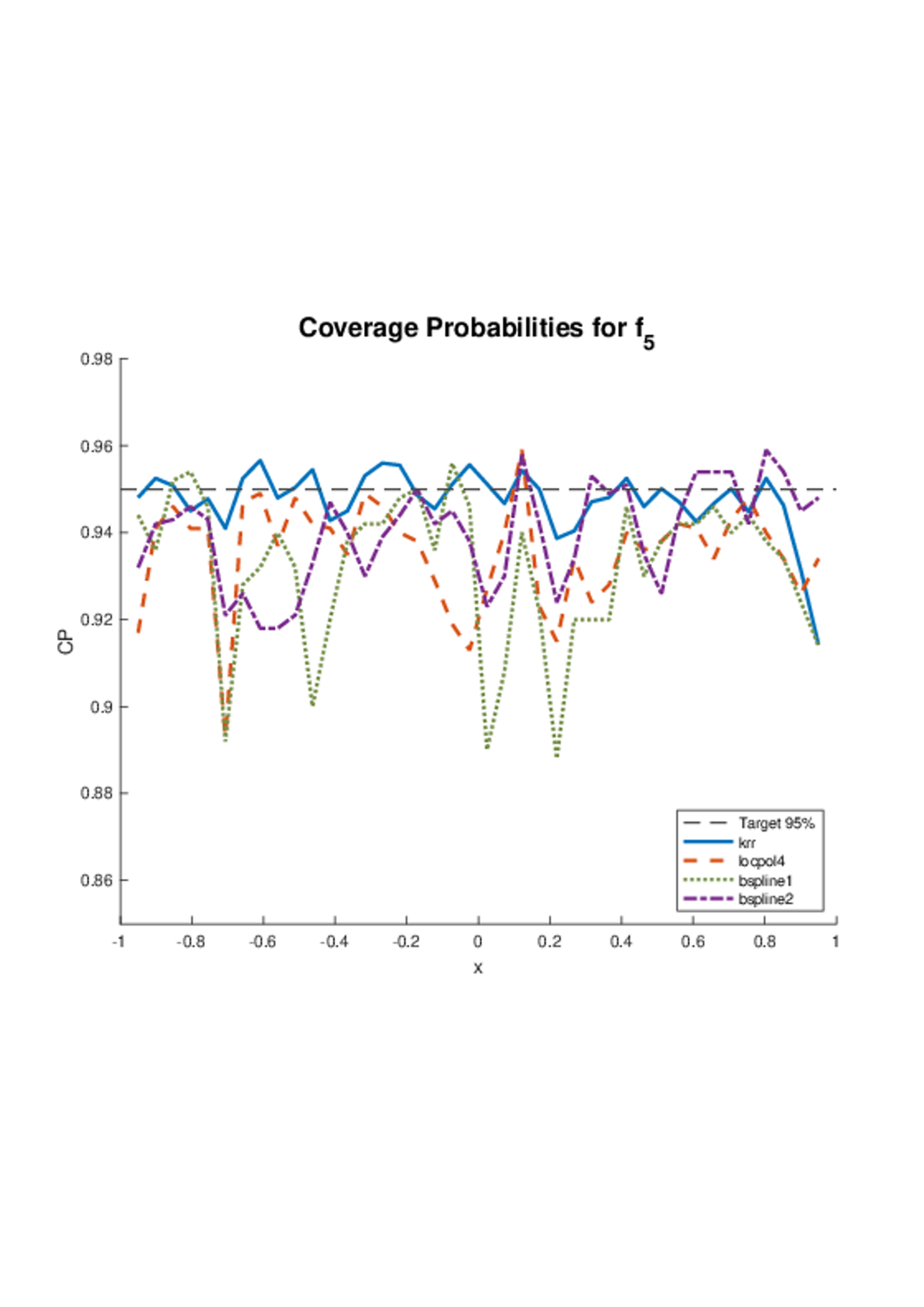}
       
        \label{fig:f5_der}
    \end{subfigure}
    
    \caption{Estimated Coverage Probability for Derivative}
    \label{fig:derivative}
\end{figure}

Figure~\ref{fig:derivative} presents the estimated coverage probabilities for $f_4$ (left) and $f_5$ (right) using the plug-in KRR estimator (\texttt{krr}), local polynomial regression with degree $p=4$ (\texttt{locpol4}), smoothing spline with higher-order-basis bias correction (\texttt{bspline1}), and smoothing spline with least squares bias correction (\texttt{bspline2}). For $f_4$, all methods produce similar results across the domain, with coverage probabilities close to the nominal 95\% level. For $f_5$, the proposed KRR method outperforms the alternative approaches over most of the domain, except near the left boundary where its coverage probability is slightly lower. For both functions, the KRR estimator exhibits relatively small fluctuations in coverage compared to other methods. Table~\ref{tab:avg_CI_length} summarizes the average confidence interval widths for the derivative estimates across all target functions. The proposed KRR method yields the narrowest intervals in both cases, demonstrating superior estimation efficiency while maintaining nominal coverage. Overall, these results indicate that the proposed method maintains stable and accurate coverage across different target functions.

\begin{table}[htbp]
\centering
\begin{tabular}{l @{\hskip 35pt} S[table-format=1.3]@{\hskip 35pt} S[table-format=1.3]}
\toprule
\textbf{Method} & {\textbf{$f_4$}} & {\textbf{$f_5$}} \\
\midrule
\texttt{krr}       &  12.5803  &  11.3918 \\
\texttt{locpol4}   &  17.2081 & 13.2785 \\
\texttt{bspline1}  &  15.6488 & 12.0351 \\
\texttt{bspline2}  &  16.8536 &  12.4222\\
\bottomrule
\end{tabular}
\caption{Average Lengths of the 95\% Confidence Intervals for Each Method.}\label{tab:avg_CI_length}
\end{table}

\section{Discussion}\label{Sec:discussion}

In this paper, we develop an asymptotic theory for a variety of linear functionals of kernel ridge regression. Our theory encompasses both upper and lower bounds for the estimator's performance and its asymptotic normality under both deterministic and random designs. 
We also demonstrate that our asymptotic theory on linear functionals can be utilized to obtain results for uniform errors and certain non-linear problems.

This article is based on the assumption that the true function $f$ resides within the RKHS ($\mathcal{H}$) associated with the kernel $K$. Our analysis can be extended to scenarios where the smoothness levels of $\mathcal{H}$ surpass those of the functional space in which the true function lies in \citep{fischer2020sobolev}. Additionally, deriving sharp and uniform confidence bands for the estimator, presenting another interesting direction for future research. The challenge in constructing sharp and uniform confidence bands arises from the reliance of existing methods for constructing  uniform confidence bands on expressing the KRR estimator through an orthonormal basis; see \cite{shang2013local, singhkernel}. Since linear functional estimators, such as derivatives, are typically non-orthogonal within this basis \citep{liu2023estimationanova}, existing testing procedures cannot be directly adapted to these estimators.



\appendix
\section*{Supplementary Material}

In this supplementary material, we provide the technical details of our theoretical results. An additional literature review is available in Section \ref{sec:literature}. Section \ref{sec:results} provides additional convergence results and discussion to complement the findings presented in the main article. Section \ref{sec:spaces} offers preliminary information on function spaces. In Section \ref{sec:sufficient}, we present the equivalent conditions for a key assumption. Section \ref{sec:proofs_main} contains the supporting lemmas and the proofs of the theorems in our main article. Section \ref{sec:numerical_plot} contains additional figures of the numerical results. 

\section{Additional Related Literature}\label{sec:literature}
KRR is a prevailing technique in machine learning and statistical modeling, demonstrating extensive utility across diverse areas, including predictive modeling \citep{ciliberto2020general,pourkamali2020kernel}, classification \citep{cortes2012algorithms,zien2007multiclass}, generative modeling \citep{dziugaite2015training,li2019implicit}, and statistical inference. In statistical inference areas,  KRR finds specific applications in tasks such as two-sample testing, independence testing \citep{bach2002kernel,gretton2005measuring,gretton2012kernel}, and causal inference \citep{singh2019kernel,singh2020kernel}.

\textit{Error bounds for KRR.} The minimax convergence rates for KRR in $L_2$  are thoroughly documented in the current literature. More recently, \cite{fischer2020sobolev} extended these rates to Sobolev norms without requiring the regression function to be contained in the hypothesis space.  For more recent work on the convergence rate for KRR, please refer to \cite{yang2017frequentist, wang2021inference, cui2021generalization, marteau2019beyond, talwai2022sobolev, zhang2023optimality}. In recent years, there has been significant interest in characterizing the learning curve for KRR, which captures the magnitude of the generalization error as it fluctuates in response to regularization parameters. Several works (e.g., \cite{bordelon2020spectrum, cui2021generalization}) depicted the learning curve of KRR under the Gaussian design. Subsequently, these results were extended to a more general random design; see \cite{ loureiro2021learning}. It has been discovered in practice and reported in the literature \citep{bauer2007regularization,gerfo2008spectral} that incorporating extra smoothness and refining the qualifications of the algorithm could yield a higher convergence rate for KRR. Recent research, including works by \citet{dicker2017kernel, li2022saturation, lian2021distributed, lin2020optimal, tuo2020improved}, further explores strategies for achieving this improved convergence rate.

Another line of research relevant to this paper explores linear functional regression, as detailed in \citep{li2020inference, lv2023kernel, sun2018optimal}. These studies focus on the linear functional defined as the 
$L_2$ inner product of the input data with a slope function. Recently, \cite{koltchinskii2023functional} demonstrated the asymptotic normality of smooth functionals with plug-in estimators, which relies on the assumption that the plug-in estimator can be well approximated by a normal random variable. For further literature on functional linear regression with special structures, please refer to \citep{cui2020partially, balasubramanian2022unified}.

\textit{Statistical inference for KRR.}  Another approach uses KRR for statistical inference, often investigating Gaussian approximation for KRR and its variants. More recently,  \cite{singhkernel} proposed a uniform confidence band for KRR, which also provided the pointwise asymptotic normality for KRR as a byproduct.  In econometric literature, exploring the linear functional form includes investigating other nonparametric regression estimators like B-spline and wavelet models.  \cite{cattaneo2013optimal} provided the uniform Bahadur representation for linear functionals of local polynomial partitioning estimators. These results are contingent upon H\"older conditions for both the underlying function and its derivatives. In a related context, \citep{belloni2015some,chen2015optimal} offered similar theoretical results under more general conditions.

\section{Additional Convergence Results and Discussion}\label{sec:results}

This section provides additional convergence results and discussion that supplement the findings presented in the main article.

\subsection{Supporting Lemmas for Bias and Variance in Section \ref{sec:BnV}}
In this part we introduce a major auxiliary problem that plays a central role in our theory. 

The first goal of this work is to quantify the bias and variance. It turns out that these quantities are intimately related to an auxiliary problem, called the \textit{noiseless kernel ridge regression}.

\begin{definition}
    Given KRR problem (\ref{KRR:1}) and function $g\in\mathcal{H}$, the associated noiseless KRR problem is defined as
    \begin{eqnarray}
        \hat{g}=\operatorname*{argmin}_{v\in\mathcal{H}}\frac{1}{n}\sum_{i=1}^n(g(x_i)-v(x_i))^2+\lambda \|v\|^2_{\mathcal{H}},
    \end{eqnarray}
    where $\lambda$ takes the same value as in (\ref{KRR:1}).
\end{definition}

Note that the target function of the noiseless KRR is $g$, not $f$.
The following lemma establishes the relationship between the bias and variance, and the noiseless KRR. For notational simplicity, we will denote $\|v\|^2_n=\frac{1}{n}\sum_{i=1}^n v^2(x_i)$ for any $v$.

\begin{lemma}\label{lemma:relationship}
The following formulas are true:
\begin{eqnarray}
    |\operatorname{BIAS}|&=&|\langle\hat{g}-g,f\rangle_\mathcal{H}| \leq\|\hat{g}-g\|_{\mathcal{H}}\|f\|_{\mathcal{H}},\label{cauchy}\\
    \operatorname{VAR}&=&\sigma^2n^{-1}\lambda^{-2}\|\hat{g}-g\|^2_n.\nonumber
\end{eqnarray}
\end{lemma}

It is worth noting that Lemma \ref{lemma:relationship} does not postulate any assumptions on the input points $X$. These points can be arbitrary: either deterministic or random.

To make Lemma \ref{lemma:relationship} useful, it is critical to establish the rates of convergence of $\|\hat{g}-g\|_n$ and $\|\hat{g}-g\|_{\mathcal{H}}$. Under a standard theory (in the sense of a minimax rate of convergence), we can only have $\|\hat{g}-g\|_{\mathcal{H}}=O(1)$, which is insufficient. The key here is: if $g$ is ``smoother'' than the baseline smoothness of $\mathcal{H}$, $\|\hat{g}-g\|_n$ and $\|\hat{g}-g\|_{\mathcal{H}}$ may decay faster than their minimax rates. Such a result is called an \textit{improved rate of convergence}. Improved rates are widely available for methodologies with a variational or optimization-based formulation, such as finite element methods \citep{brenner2008mathematical} and radial basis function approximation \citep{wendland2004scattered}. In statistics, it was also discovered long ago that extra smoothness and boundary conditions could yield a higher convergence rate for smoothing splines \citep{wahba1975periodic}. Such extra conditions are referred to as the source conditions in the machine learning literature \cite{bauer2007regularization,li2022saturation,rastogi2017optimal}. Recent advances have demonstrated the general ideas to pursue an improved convergence rate for KRR \citep{dicker2017kernel,fischer2020sobolev,guo2017learning,lin2017distributed,tuo2020improved}. 
In this work, we will adopt the approach of \cite{tuo2020improved} to derive the improved rates, which leads to results in terms of both the $\|\cdot\|_n$ and $\|\cdot\|_{\mathcal{H}}$ norms. 


We also highlight that the Cauchy-Schwarz inequality used in (\ref{cauchy}) is sharp: the equality holds if $f$ is a multiple of $\hat{g}-g$. This implies that $\|\hat{g}-g\|_\mathcal{H}$ is the \textit{worst-case bias} over the unit ball of $\mathcal{H}$. To be more precise, when referring to the worst-case bias, we imagine the application of KRR to a family of
models having the form of equation (\ref{model}), but with different $f$. Nevertheless, the same $g$ and parameter $\lambda$ are used for each model. For each $f$, denote the corresponding bias by $\operatorname{BIAS}_f$, and then we immediately have Corollary \ref{coro:worstcase}. 

\begin{corollary}\label{coro:worstcase}
    $\sup_{\|f\|_\mathcal{H}\leq 1}|\operatorname{BIAS}_f|=\|\hat{g}-g\|_\mathcal{H}.$
\end{corollary}


\subsection{Comments on Assumption \ref{assum:f}}\label{sec:A2}
Assumption \ref{assum:f} is a critical condition to ensure an improved rate of convergence for $\hat{g}-g$, by imposing an extra smoothness condition on $g$. Technically, Assumption \ref{assum:f} holds if $g$ lies in a function space $\mathcal{G}$ such that the dual space of $\mathcal{G}$ (with respect to the inner product of $\mathcal{H}$), denoted as $\mathcal{G}^*$, is an intermediate space between $L_2$ and $\mathcal{H}$, i.e., $L_2\supset \mathcal{G}^*\supset \mathcal{H}\supset \mathcal{G}$. In this case,
\[\langle g,v\rangle_\mathcal{H}\leq \|g\|_{\mathcal{G}}\|v\|_{\mathcal{G}^*}.\]
If an ``interpolation inequality'' with the form
\begin{eqnarray}\label{Ineq}
    \|v\|_{\mathcal{G}^*}\leq C\|v\|_{L_2}^\delta\|v\|_\mathcal{H}^{1-\delta}
\end{eqnarray}
holds for some $\delta\in (0,1]$, Assumption \ref{assum:f} is valid. In general, an interpolation inequality is an inequality of the form $\|v\|_{1}\leq C\|v\|_2^{1-\theta}\|v\|_3^{\theta}$ for $0<\theta<1$, which describes the relative strength of the norms $\|\cdot\|_1,\|\cdot\|_2$ and $\|\cdot\|_3$. For example, the following inequality, which follows simply from H\"older's inequality, links three $L_p$ norms: 
\begin{eqnarray}\label{LpInter}
    \|v\|_{L_{p_\theta}}\leq \|v\|^{1-\theta}_{L_{p_0}}\|v\|^\theta_{L_{p_1}},
\end{eqnarray}
where the indices $1\leq p_0<p_1\leq \infty$ and $0<\theta<1$ satisfy
\begin{eqnarray}\label{LpInter2}
   \frac{1}{p_\theta}=\frac{1-\theta}{p_0}+\frac{\theta}{p_1}. 
\end{eqnarray}
In view of (\ref{LpInter})-(\ref{LpInter2}), we can regard space $L_{p_\theta}$ as an ``interpolation'' of spaces $L_{p_0}$ and $L_{p_1}$, and this is where its name derives from. 
In Section \ref{sec:point}, we use the interpolation inequality (\ref{interpolation}) that links the $L_2,L_\infty$ and $H^m$ norms. A related field from functional analysis is referred to as ``interpolation theory'' (e.g., the Riesz-Thorin theorem). An interpolation inequality is usually a consequence of the corresponding interpolation theory.

Besides using interpolation inequalities, Assumption \ref{assum:f} can be verified directly when a series expansion is applied for $g$. See Proposition \ref{Prop:series} in Section \ref{sec: eigen_series}.

\subsection{Further Improvements in Bias}\label{sec:further}

In case $f$ also possesses an extra smoothness, the bias upper bound in Theorem \ref{Coro:rate} can be further improved. Assumption \ref{assum:further} is analogous to Assumption \ref{assum:f}.

\begin{assumption}\label{assum:further}
    There exist constants $C_f>0$ and $\gamma\in(0,1]$, such that for each $v\in\mathcal{H}$,
    \begin{eqnarray}\label{smoothnessf}
        |\langle f,v\rangle_{\mathcal{H}}|\leq C_f\|v\|^{\gamma}_{L_2}\|v\|^{1-\gamma}_{\mathcal{H}}.
    \end{eqnarray}
\end{assumption}

\begin{theorem}\label{Coro:further}
    Under the conditions and notation of Theorem \ref{Coro:rate} in addition to Assumption \ref{assum:further}, we have
    $|\operatorname{BIAS}|=O_\mathbb{P}(C_f\lambda^\frac{\gamma+\delta}{2}).$  
\end{theorem}

In view of Corollary \ref{Coro:further}, in the presence of Assumption \ref{assum:further}, the best order of magnitude of $\lambda$ to balance the bias and the variance is $\lambda\asymp n^{-\frac{1}{\gamma+1}}$. In particular, if $\gamma=1$, one can choose $\lambda\asymp n^{-\frac{1}{2}}$. However, as $\gamma$ is unknown in practice, it is difficult to take advantage of this improved rate in statistical inference.

\subsection{Discussion on the Semiparametric Effect}\label{sec:semiparametric}

As shown in Proposition \ref{prop:L2}, when $\delta=1$, there exists $h$ such that $\langle g,v\rangle_{\mathcal{H}}=\langle h,v\rangle_{L_2}$ for each $v\in\mathcal{H}$. In this section, we will discuss the known results from the standard semiparametric statistical theory through the lens of the proposed approach. In the literature, it is often assumed that the input points $x_i$'s are independent and identical random samples. Denote the probability density function of $x_1$ by $p_X$. With the techniques articulated in \cite{mammen1997penalized,tuo2015efficient}, one can prove 
\begin{eqnarray}\label{semi}
    \sqrt{n}\int_\Omega (\hat{f}-f)(x)h(x)dx\xrightarrow{\mathscr{L}} N\left(0,\sigma^2\int_\Omega h^2(x)/p_X(x)dx\right),
\end{eqnarray}
under $\lambda=o_p(n^{-1/2})$ in addition to some other conditions. Among these conditions, the most important one to our attention is 
\begin{eqnarray}\label{Hcondtion}
    h/p_X\in\mathcal{H}.
\end{eqnarray}
The objective of this part is to further understand (\ref{semi}) together with the condition (\ref{Hcondtion}).
Clearly, (\ref{semi}) implies that $\operatorname{BIAS}=o_\mathbb{P}(n^{-1/2})$, which cannot be obtained by simply applying Theorem \ref{Coro:rate} under the condition $\lambda=o(n^{-1/2})$. This implies that further improvement in the rate of convergence emerges. 

To explain the actual reason, we should take the perspective of numerical integration. Define
\begin{eqnarray}\label{NI}
    \mathscr{E}:=\int_\Omega (\mathbb{E}_E\hat{f}-f)(x)h(x)dx-\frac{1}{n}\sum_{i=1}^n (\mathbb{E}_E\hat{f}-f)(x_i)\frac{h(x_i)}{p_X(x_i)},
\end{eqnarray}
the error of approximating the integral $\int_\Omega (\hat{f}-f)(x)h(x)dx$ with the summation $n^{-1}\sum_{i=1}^n (\hat{f}-f)(x_i)h(x_i)/p_X(x_i)$. Under our setting, $x_i$'s are not necessarily random, and $p_X$ can be any function of our choice with the goal of making $|\mathscr{E}|$ small. 

We will first show that the second term in (\ref{NI}) is small if $h/p_X\in\mathcal{H}$.

\begin{theorem}\label{Prop:semi}
    If $h/p_X\in\mathcal{H}$,
    \[\left|\frac{1}{n}\sum_{i=1}^n (\mathbb{E}_E\hat{f}-f)(x_i)\frac{h(x_i)}{p_X(x_i)}\right|\leq \lambda \|f\|_\mathcal{H}\|h/p_X\|_{\mathcal{H}}.\]
\end{theorem}

Thus, regarding $\|f\|_\mathcal{H}$ and $\|h/p_X\|_{\mathcal{H}}$ as constants,
\begin{eqnarray*}
    |\operatorname{BIAS}|&=&\left|\int_\Omega (\mathbb{E}_E\hat{f}-f)(x)h(x)dx\right|\leq|\mathscr{E}|+\left|\frac{1}{n}\sum_{i=1}^n (\mathbb{E}_E\hat{f}-f)(x_i)\frac{h(x_i)}{p_X(x_i)}\right|\\
    &=&|\mathscr{E}|+O(\lambda).
\end{eqnarray*}
In case $x_i$'s are indeed independent copies with density $p_X$, the standard empirical process theory \citep{geer2000empirical} can show that $|\mathscr{E}|=o_\mathbb{P}(n^{-1/2})$, provided that $\int_\Omega(\mathbb{E}_E\hat{f}-f)^2(x)p_X(x)dx=o_\mathbb{P}(1)$. Hence we have recovered the results from the semiparametric statistical literature. 
If the input points $X$ are carefully chosen, the integration error can be much smaller than that from a Monte Carlo sampling. For example, when $\Omega=[0,1]$ and $X$ are evenly distributed, choosing $p_X=1$, then $|\mathscr{E}|$ can be as small as $O(n^{-2})$. In this situation, $|\mathscr{E}|$ can be smaller than $O(\lambda)$ when $\lambda$ is not too small, which implies that the lower bound in Theorem \ref{Coro:biasrate} can be reached.

\subsection{Expressions in terms of the Eigensystem}\label{sec: eigen_series}
Suppose Assumption \ref{assum:matern} is true. Let $\rho_1\geq \rho_2\geq \cdots$ and $\eta_1,\eta_2,\ldots$ be the eigenvalues and $L_2$-normalized eigenfunctions of the integral operator $L(v)=\int_\Omega K(\cdot,x)v(x)dx$. In this case, we have the representation
\begin{eqnarray}\label{RKHSnorm}
    \left\|\sum_{i=1}^\infty c_i \eta_i\right\|^2_\mathcal{H}=\sum_{i=1}^\infty\frac{c_i^2}{\rho_i},
\end{eqnarray}
for any $c_i\in\mathbb{R}$ such that the right side of (\ref{RKHSnorm}) is convergent. On the other hand, $\mathcal{H}$ is equal to all functions in the form of (\ref{RKHSnorm}) with a finite norm. Proposition \ref{Prop:series} links the series presentation of functions with Assumption \ref{assum:f}.

\begin{proposition}\label{Prop:series}
    Under Assumption \ref{assum:matern},
    suppose $w=\sum_{i=1}^n c_i \eta_i\in\mathcal{H}$ satisfies
    \begin{eqnarray}\label{Hkappanorm}
        \|w\|_{\mathcal{H},\kappa}^2:=\sum_{i=1}^\infty \frac{c_i^2}{\rho^{1+\kappa}}<\infty,
    \end{eqnarray}
    for some $\kappa\in(0,1]$. Then for any $v\in\mathcal{H}$,
    \[|\langle w,v\rangle_\mathcal{H}|\leq \|w\|_{\mathcal{H},\kappa}\|v\|^{\kappa}_{L_2}\|v\|^{1-\kappa}_{\mathcal{H}}.\]
\end{proposition}

\begin{remark}
    A condition equivalent to (\ref{Hkappanorm}) was considered in \cite{dicker2017kernel} to pursue an improved rate. When $\kappa=1$, $\langle w,\cdot\rangle_\mathcal{H}$ is equal to an $L_2$ inner product; see Proposition \ref{prop:L2} and \cite{tuo2020improved,wendland2004scattered}.
\end{remark}

Corollary \ref{Coro:series} is an immediate consequence of Proposition \ref{Prop:series} and Theorem \ref{Coro:rate}.

\begin{corollary}\label{Coro:series}
    Under the conditions of Theorem \ref{Coro:rate} and Proposition \ref{Prop:series}, we have
    \begin{eqnarray*}
        |\operatorname{BIAS}|&=&O_\mathbb{P}(\lambda^{\frac{\kappa}{2}}\|f\|_\mathcal{H}),\\
        \operatorname{VAR}&=&O_\mathbb{P}(\sigma^2n^{-1}\lambda^{\kappa-1}).
    \end{eqnarray*}
    In addition, if $\kappa>\frac{d}{2m}$, $\sigma^2>0$, and $\lambda=o(n^{-1})$, then 
    \begin{eqnarray}\label{seriesAN}
        (\operatorname{VAR})^{-\frac{1}{2}}\langle \hat{f}-f,g\rangle_\mathcal{H}\xrightarrow{\mathscr{L}}N(0,1).
    \end{eqnarray}
\end{corollary}

Because $\kappa$ can be arbitrarily small, (\ref{Hkappanorm}) is a relatively weak condition given that $w\in\mathcal{H}$. Thus we can say that improved rates are generally available for ``most'' functions in $\mathcal{H}$. A relevant conclusion is that the further improved rate for the bias term in Section \ref{sec:further} are also commonly available. Specifically, by Corollary \ref{Coro:further}, if $\|f\|_{\mathcal{H},\kappa}<\infty$, we have $\operatorname{BIAS}=O_\mathbb{P}(\lambda^{\frac{\delta+\kappa}{2}})$.

The asymptotic normality (\ref{seriesAN}) requires $\kappa>\frac{d}{2m}$, because we use $\tau=1$ in Assumption \ref{assum:tau}. We conjecture that this cannot be improved in general, when the magnitude of the coefficients $c_i$'s of $g$ in (\ref{RKHSnorm}) fluctuates wildly. 

\subsection{Uniform Bounds}\label{sec:uniformSup}

Recall that the goal is to determine the rate of convergence of the \textit{uniform bias} and the \text{uniform variance term}.

As we remarked in Section \ref{sec:upper}, the event $\Xi_\epsilon$ is independent of $g$. Therefore, the uniform bias is simply the largest bias on $\Xi_\epsilon$, which, together with the interpolation inequality (\ref{interpolationD}), leads to Corollary \ref{Coro:uniformbias}.

\begin{corollary}\label{Coro:uniformbias}
    Suppose Assumptions \ref{assum:matern} and \ref{assum:norms} are true. In addition, $m>d/2+|\alpha|$ and $\lambda\gtrsim n^{-2m/d}$. Then we have
    \begin{eqnarray}\label{uniformbias}
        \sup_{x\in\Omega}\left|\mathbb{E}_E D^\alpha \hat{f}(x)-D^\alpha f(x)\right|=O_\mathbb{P}(\lambda^{\frac{1}{2}-\frac{d+2|\alpha|}{4m}}\|f\|_\mathcal{H}).
    \end{eqnarray}
\end{corollary}
The uniform variance term is a supremum of a stochastic process, which seemingly depends on the random noise's tail property. Theorem \ref{Th:uniformvar} shows that, if the random noise has a sub-Gaussian tail, the uniform variance term has almost the same order of magnitude as the pointwise variance term, except for a logarithmic factor. 

\begin{theorem}\label{Th:uniformvar}
    Suppose Assumptions \ref{assum:matern} and \ref{assum:norms} are met, $m>d/2+|\alpha|$ and $n^{-2m/d}\lesssim\lambda\leq 1$.
    In addition, if the random error satisfies $\mathbb{E}\exp\{\vartheta e_1\}\leq \exp\{\vartheta^2\varsigma^2/2\}$ for all $\vartheta\in \mathbb{R}$ and some $\varsigma^2>0$, we have
    \begin{eqnarray}\label{uniformvarsubgaussian}
        \sup_{x\in\Omega}\left|D^\alpha \hat{f}(x)-\mathbb{E}_E D^\alpha \hat{f}(x)\right|= O_\mathbb{P}\left(\varsigma n^{-\frac{1}{2}}\lambda^{-\frac{d+2|\alpha|}{4m}}\sqrt{\log\left(\frac{C}{\lambda}\right)}\right),
    \end{eqnarray}
    for some $C>1$ independent of $\varsigma^2,\lambda$, and $n$.
\end{theorem}

Comparing Theorem \ref{Th:uniformvar} with the pointwise bound given by Theorem \ref{Th:derivative}, it can be seen that the uniform bound is inflated only by a logarithmic factor $\sqrt{\log(C/\lambda)}$. This factor cannot be improved in general, as shown in the lower bound in Theorem \ref{Th:uniformlower} under the assumption that the noise follows a normal distribution.

\begin{theorem}\label{Th:uniformlower}
    Suppose Assumptions \ref{assum:matern} and \ref{assum:norms} are met, $m>d/2+|\alpha|$ and $n^{-2m/d}\lesssim\lambda\leq 1$.
    In addition, if the random error follows a normal distribution, i.e., $e_1\sim N(0,\sigma^2)$ with $\sigma^2>0$, we have
   
    \begin{eqnarray*}
        \mathbb{P}\left(\sup_{x\in\Omega}\left|D^\alpha \hat{f}(x)-\mathbb{E}_E D^\alpha \hat{f}(x)\right|\geq C_1\sigma n^{-\frac{1}{2}}\lambda^{-\frac{d+2|\alpha|}{4m}}\sqrt{\log\left(\frac{C_2}{\lambda}\right)}\right)\geq \frac{1}{2}
    \end{eqnarray*}
    for some constants $C_1>0,C_2>1$ independent of $\sigma^2$, $\lambda$ and $n$.
\end{theorem}

\subsection{A Nonlinear Problem}\label{sec:nonlinearSup}
Consider the nonlinear functionals
$\min_{x\in\Omega} f(x) \text{ and } \operatorname*{argmin}_{x\in\Omega}f(x)$.
By plugging in the KRR estimator $\hat{f}$, we obtain intuitive estimators of $\min_{x\in\Omega} f(x)$ and $\operatorname*{argmin}_{x\in\Omega}f(x)$ as 
\[\hat{f}_{\min}:=\min_{x\in\Omega} \hat{f}(x) \text{ and } \hat{x}_{\min}:=\operatorname*{argmin}_{x\in\Omega}\hat{f}(x),\]
respectively. The goal is to study the asymptotic properties of these estimators. In order to linearize the problem, we make some regularity assumptions.

\begin{assumption}\label{assum:nonlinear}
    The function $f$ has a unique minimizer $x_{\min}$. 
    Besides, $x_{\min}$ is an interior point of $\Omega$, and $f$ is twice differentiable at $x_{\min}$ with a positive definite Hessian matrix $H:=\frac{\partial^2 f}{\partial x\partial x^T}(x_{\min})$.
\end{assumption}

Here we provide a rigorous result following the intuition provided in the main article. For simplicity, we only show the result under the optimal choice of the tuning parameter $\lambda\asymp n^{-1}$, which yields the best rate of convergence. The results are given in Theorem \ref{Th:nonlinear}.

\begin{theorem}\label{Th:nonlinear}
    Suppose Assumptions \ref{assum:matern}, \ref{assum:norms}, and \ref{assum:nonlinear} are true, $\sigma^2>0$, and $m>2+d/2$. The covariance matrix $\operatorname{COV}$ and its estimate $\widehat{\operatorname{COV}}$ are defined by (\ref{Covnonlinear}) and (\ref{Covestimate}), respectively. Then under the optimal choice of the tuning parameter $\lambda\asymp n^{-1}$, we have
    \begin{enumerate}
        \item $\|\hat{x}_{\min}-x_{\min}\|=O_\mathbb{P}(n^{-\frac{1}{2}+\frac{d+2}{4m}}),f(\hat{x}_{\min})-f(x_{\min})=O_\mathbb{P}(n^{-1+\frac{d+2}{2m}})$;
        \item $\widehat{\operatorname{COV}}^{-\frac{1}{2}}\hat{H}(\hat{x}_{\min}-x_{\min})\xrightarrow{\mathscr{L}} N(0,I)$.
    \end{enumerate}
\end{theorem}


\section{Review of Sobolev Spaces}\label{sec:spaces}

Let $\Omega\subset\mathbb{R}^d$ be a domain.
For a non-negative integer $k$, the Sobolev space $H^k(\Omega)$ is defined as the closure of sufficiently smooth functions over the norm
\[\|f\|_{H^k(\Omega)}^2=\sum_{|\alpha|\leq k}\|D^\alpha f\|_{L_2(\Omega)}^2.\]
To define $H^m(\Omega)$ for non-integer $m=k+s$ for $k\in\mathbb{N}$ and $s\in (0,1)$, there is a direct approach using the Sobolev–Slobodeckij semi-norm
\begin{eqnarray}\label{Sobolev}
    |f|_{W^{k+s}(\Omega)}^2:=\sum_{|\alpha|= k}\int_\Omega\int_\Omega \frac{|D^\alpha f(x)-D^\alpha f(y)|^2}{\|x-y\|^{d+2s}}dxdy,
\end{eqnarray}
and an equivalent norm of $H^{k+s}(\Omega)$ is given by
\begin{eqnarray*}
\|f\|_{H^{k+s}(\Omega)}^2:=\|f\|_{H^k(\Omega)}^2+|f|_{W^{k+s}(\Omega)}^2.
\end{eqnarray*}
For notational simplicity, we omit the domain $\Omega$ in notation like $H^m(\Omega)$ and $\|\cdot\|_{H^m(\Omega)}$ if $\Omega$ is the experimental region of the KRR problem of our main interest.

A reproducing kernel Hilbert space $\mathcal{H}$ is a Hilbert space of continuous functions over a domain $\Omega$, satisfying the reproducing property
\begin{eqnarray}\label{reproducing}
    f(x)=\langle f,K(\cdot,x)\rangle_\mathcal{H},
\end{eqnarray}
for each $f\in\mathcal{H}$ and $x\in\Omega$. Here $K(\cdot,\cdot)$ is a positive semi-definite function called the reproducing kernel.
Stationary kernels, i.e., $K(x,y)=\Phi(x-y)$ for some $\Phi:\mathbb{R}^d\mapsto\mathbb{R}$, are commonly used. When the Fourier transform of $\Phi$, denoted as $\tilde{\Phi}$, satisfies
\begin{eqnarray}\label{fourier}
    c_1(1+\|\omega\|^2)^{-m}\leq \tilde{\Phi}(\omega)\leq c_2(1+\|\omega\|^2)^{-m},
\end{eqnarray}
for $m>d/2$, some constants $0<c_1<c_2$, and all $\omega\in\mathbb{R}^d$, and $\Omega$ has a Lipschitz boundary, then $\mathcal{H}=H^m$ with equivalent norms; see \cite{wendland2004scattered}. A prominent example of kernels satisfying (\ref{fourier}) is the Mat\'ern correlation family \cite{stein2012interpolation} with smoothness $\nu=m-d/2$, defined as
\[\Phi(x;\nu,\phi):=\frac{1}{\Gamma(\nu)2^{\nu-1}}(2\sqrt{\nu}\phi\|x\|)^\nu K_\nu(2\sqrt{\nu}\phi\|x\|),\]
where $\phi,\nu>0$, and $K_\nu$ is the modified Bessel function of the second kind.

\section{Sufficient Conditions of Assumption \ref{assum:norms} in Section \ref{sec:improved}}\label{sec:sufficient}
In this section, we provide the equivalent conditions for Assumption \ref{assum:norms} under both random designs and fixed designs.

By Assumption \ref{assum:matern}, $\mathcal{H}$ and $H^m$ are equivalent. Therefore, we can replace the $\mathcal{H}$-norm by the $H^m$-norm in Assumption \ref{assum:norms}, with possibly different $C_1$ and $C_\epsilon$, to obtain an equivalent assumption. In this section, we shall show sufficient conditions for this equivalent assumption.

\subsection{Random Designs}


The goal of this section is to prove Theorems \ref{th:randomdesign1} and \ref{th:randomdesign2} under the random design Assumptions \ref{assum:random1} and \ref{assum:random2}, respectively. These results refine Lemma 5.16 of \cite{geer2000empirical}.

\begin{assumption}\label{assum:random1}
    The input sites $x_1,\ldots,x_n$ are independent and identically distributed random variables over $\Omega$ with density function $\mu(\cdot)$. In addition, $\inf_{x\in\Omega}\mu(x)=\mu_0>0$.
\end{assumption}

\begin{theorem}\label{th:randomdesign1}
Suppose $\Omega$ satisfies the conditions in Assumption \ref{assum:matern}.
    Under Assumption \ref{assum:random1}, there exist constants $C_1,C_2,C_3,C_4>0$ independent of $n$ and $X$, such that for each $t\geq 1$,
        \begin{multline}
         \mathbb{P}\left(\|v\|_{L_2}\leq \max\left\{ C_1 \|v\|_{n},C_2 t n^{-m/d}\|v\|_{H^m}\right\} \text{ for all } v\in H^m\right)\\ \geq 1-C_3\exp\{-C_4 t^{d/m}\},\label{l2_to_nTH}
    \end{multline}
\end{theorem}

\begin{assumption}\label{assum:random2}
    The input sites $x_1,\ldots,x_n$ are independent and identically distributed random variables over $\Omega$ with density function $\mu(\cdot)$. In addition, $\sup_{x\in\Omega}\mu(x)<\infty$.
\end{assumption}

\begin{theorem}\label{th:randomdesign2}
Suppose $\Omega$ satisfies the conditions in Assumption \ref{assum:matern}.
    Under Assumption \ref{assum:random2}, there exist constants $C_1,C_2,C_3,C_4>0$ independent of $n$ and $X$, such that for each $t\geq 1$,
        \begin{multline}
         \mathbb{P}\left(\|v\|_{n}\leq \max\left\{C_1 \|v\|_{L_2},C_2 t n^{-m/d}\|v\|_{H^m}\right\} \text{ for all } v\in H^m\right)\\ \geq 1-C_3\exp\{-C_4 t^{d/m}\}.\label{n_to_l2TH}
    \end{multline}
\end{theorem}


There are three major steps to prove Theorem \ref{th:randomdesign1}. Here we call $\|v\|_{L_2}\leq C_1\|v\|_{n}+C_2tn^{-m/d}\|v\|_{H^m}$ the ``norm inequality'' for simplicity.
\begin{enumerate}[noitemsep]
    \item Use a Bernstein inequality to show that the norm inequality is true with a high probability for each fixed $v$. This is given by Lemma \ref{lemma:bernstein}.
    \item Apply a ``peeling device'' \citep{geer2000empirical} with regard to the $L_2$ norm, and show that the norm inequality is true with a high probability for all $v$ satisfying $\|v\|_{L_2}\geq r$ and $\|v\|_{H^m}\leq R$ with fixed $(r,R)$. This is given by Lemma \ref{lemma:step2}.
    \item Use a normalization argument to show (\ref{l2_to_nTH}). The proof is given at the end of this section.
\end{enumerate}

Theorem \ref{th:randomdesign2} can be proved in a similar fashion, starting from the opposite side of the Bernstein inequality. Hence, we omit the proof of Theorem \ref{th:randomdesign2}.

\begin{lemma}\label{lemma:bernstein}
    Suppose $v$ is continuous over $\Omega$. Under Assumption \ref{assum:random1}, we have
    \[\mathbb{P}(\|v\|_{L_2}\leq \sqrt{2/\mu_0}\|v\|_n)\geq 1- \exp\left\{-\frac{3n\mu_0\|v\|^2_{L_2}}{28\|v\|^2_{L_\infty}}\right\}.\]
\end{lemma}

\begin{proof}
    We first note that $|v^2(x_1)-\mathbb{E}v^2(x_1)|\leq \max\{v^2(x_1),\mathbb{E}v^2(x_1)\}\leq \|v\|^2_{L_\infty}$, and
    \[\mathbb{E}[v^2(x_1)-\mathbb{E}v^2(x_1)]^2\leq \mathbb{E}v^4(x_1)\leq \|v\|^2_{L_\infty}\mathbb{E}v^2(x_1).\]
    Let $t=0.5$; we use the Bernstein inequality to obtain
    \begin{eqnarray*}
            \mathbb{P}(\|v\|^2_n-\mathbb{E}v^2(x_1)\leq -t\mathbb{E}v^2(x_1))&\leq& \exp\left\{-\frac{nt^2(\mathbb{E}v^2(x_1))^2/2}{\|v\|^2_{L_\infty}\mathbb{E}v^2(x_1)+\|v\|^2_{L_\infty}t\mathbb{E}v^2(x_1)/3}\right\}\nonumber\\
            &=&\exp\left\{-\frac{3nt^2\mathbb{E}v^2(x_1)}{2\|v\|^2_{L_\infty}(t+3)}\right\}\nonumber\\
            &=&\exp\left\{-\frac{3n\mathbb{E}v^2(x_1)}{28\|v\|^2_{L_\infty}}\right\},
    \end{eqnarray*}
    which, together with the property
    \[\mathbb{E}v^2(x_1)=\int_\Omega v^2(x)\mu(x)dx\geq \mu_0\|v\|^2_{L_2},\]
    yields the desired result.
\end{proof}

The above lemma works only for a specific $v$. To get a bound uniform for a range of $v$, we need to consider the covering number.

\begin{definition}[$d_\mathcal{V}$-covering number]
    Let $\mathcal{V}$ be a set of functions over $\Omega$, and $d_\mathcal{V}(\cdot,\cdot)$ be a semi-metric over $\mathcal{V}$. Define $N(\epsilon,\mathcal{V},d_\mathcal{V})$ the smallest integer $N$, such that there exist functions (also referred to as \textit{centers}) $v_1,\ldots,v_N$ satisfying $\sup_{v\in\mathcal{V}}\min_{1\leq i\leq N}d_\mathcal{V}(v,v_i)\leq \epsilon$. In particular, for the case $\mathcal{V}\subset L_\infty$, we denote $N(\epsilon,\mathcal{V},\|\cdot\|_{L_\infty})$ as $N(\epsilon,\mathcal{V})$ for simplicity.
\end{definition}

The following result can be found in \cite{edmunds2008function}. Define $H^m(R)=\{v\in H^m:\|v\|_{H^m}\leq R\}$ for $R\geq 0$. 

\begin{proposition}
Suppose $\Omega$ satisfies the conditions in Assumption \ref{assum:matern}.
    There exists a constant $A>0$ depending only on $\Omega,m,d$, such that all $r>0$,
    \begin{eqnarray}\label{entropy}
        \log N(r,H^m(R))\leq A(R/r)^{d/m}.
    \end{eqnarray}
\end{proposition}

\begin{lemma}\label{lemma:step2}
Suppose $\Omega$ satisfies the conditions in Assumption \ref{assum:matern}.
    Fix $R>0$. For any $r>0$ satisfying 
    \begin{eqnarray}
        \sqrt{n}(r/R)^{d/m}\geq 1,\label{epsilon}
    \end{eqnarray}
    under Assumption \ref{assum:random1},
   there exists constants $C_1,C_2,C_3,C_4$ depending only on $\Omega,m,d$ and $\mu_0$, such that
   \begin{eqnarray*}
       \mathbb{P}\left(\|v\|_{L_2}\leq C_1 \|v\|_n \text{ for all } v\in H^m(R) \text{ with } \|v\|_{L_2}\geq C_2r\right)\\\geq 1-C_3\exp\{-C_4n(r/R)^{d/m}\}
   \end{eqnarray*}
\end{lemma}

\begin{proof}
    The proof proceeds by applying a peeling device.
    Let $|\Omega|$ be the volume of $\Omega$, and $\mathcal{V}_s:=\{v\in\mathcal{V}(R):(s-1)|\Omega|^{1/2}r \leq \|v\|_{L_2}\leq s |\Omega|^{1/2}r\}$ for $s=1,2,\ldots$. By the definition of the covering number, we have $N(r,\mathcal{V})$ centers. For $v\in\mathcal{V}_s$, denote its associated center as $\operatorname{ctr}v$. Note that
    \[(s-2)|\Omega|^{1/2}r\leq\|v\|_{L_2}-|\Omega|^{1/2}r\leq \|\operatorname{ctr}v\|_{L_2}\leq\|v\|_{L_2}+|\Omega|^{1/2}r\leq (s+1)|\Omega|^{1/2}r.\] Define event
    \[E_v:=\{\|\operatorname{ctr}v\|_{L_2}\leq \sqrt{2/\mu_0}\|\operatorname{ctr}v\|_n\}.\]
    Then on $E_v$, we have 
    \begin{eqnarray*}
        \|v\|_{L_2}&\leq& s|\Omega|^{1/2}r\leq \frac{s}{s-2}\|\operatorname{ctr}v\|_{L_2}\leq \frac{s\sqrt{2/\mu}}{s-2}\|\operatorname{ctr}v\|_n\\
        &\leq&\frac{s\sqrt{2/\mu_0}}{s-2}(\|v\|_n+r)\leq \frac{s\sqrt{2/\mu_0}}{s-2}\left(\|v\|_n+\frac{|\Omega|^{-1/2}}{s-1}\|v\|_{L_2}\right).
    \end{eqnarray*}
    Then for $s> 2|\Omega|^{-1/2}+1$, we have $\|v\|_{L_2}\leq 4\sqrt{2/\mu_0}\|v\|_n$. This proves $E_v\subset \{\|v\|_{L_2}\leq 4\sqrt{2/\mu_0}\|v\|_n\}$. Therefore, by Lemma \ref{lemma:bernstein}, we have
    \begin{eqnarray*}
    &&\mathbb{P}(\|v\|_{L_2}> 4\sqrt{2/\mu_0}\|v\|_n)\leq \mathbb{P}(E_v^c)\leq \exp\left\{-\frac{3n\mu_0\|v\|^2_{L_2}}{28\|v\|^2_{L_\infty}}\right\}\\
    &\leq&\exp\left\{-\frac{3n\mu_0\|v\|^2_{L_2}}{28C\|v\|^{2-d/m}_{L_2}\|v\|^{d/m}_{H^m}}\right\}=\exp\left\{-\frac{3n\mu_0}{28C}\frac{\|v\|^{d/m}_{L_2}}{\|v\|^{d/m}_{H^m}}\right\}\\
    &\leq&\exp\left\{-\frac{3n\mu_0}{28C}\left(\frac{(s-1)|\Omega|^{1/2}r}{R}\right)^{d/m}\right\},
    \end{eqnarray*}
    where the third inequality follows from the interpolation inequality (\ref{interpolation}). Choose $S_0$ large enough such that we also have $3\mu_0(S_0-1)^{d/m}|\Omega|^{d/(2m)}/(28C)>(A+1)S_0^{d/(2m)}$, for $A$ defined in (\ref{entropy}). Now we arrive at
    \begin{eqnarray*}
        &&\mathbb{P}\left(\bigcup_{v\in \cup_{s\geq S_0}\mathcal{V}_s}\left\{\|v\|_{L_2}> 4\sqrt{2/\mu_0}\|v\|_n\right\}\right)\\
        &\leq&\sum_{s\geq S_0}\mathbb{P}\left(\bigcup_{v\in \mathcal{V}_s}\left\{\|v\|_{L_2}> 4\sqrt{2/\mu_0}\|v\|_n\right\}\right)\\
        &\leq&\sum_{s\geq S_0}\mathbb{P}\left(\bigcup_{v\in \mathcal{V}_s}E_v^c\right)\\
        &\leq&\sum_{s\geq S_0}\exp\left\{\log N(r,H^m(R))-\frac{3n\mu_0}{28C}\left(\frac{(s-1)|\Omega|^{1/2}r}{R}\right)^{d/m}\right\}\\
        &\leq&\sum_{s\geq S_0}\exp\left\{A(R/r)^{d/m}-(A+1)s^{d/(2m)}n(r/R)^{d/m}\right\}\\
        &\leq&\sum_{s\geq S_0}\exp\left\{-s^{d/(2m)}n(r/R)^{d/m}\right\},
    \end{eqnarray*}
    where the last inequality follows from (\ref{epsilon}). This completes the proof.
\end{proof}

Now we are ready to prove (\ref{l2_to_nTH}).

\begin{proof}[Proof of Theorem \ref{th:randomdesign1}]
Because $\|0\|_{L_2}\leq \max\left\{A_1 \|0\|_n, A_2 n^{-m/d}\|0\|_{H^m}\right\}$ is certainly true, we only need to consider the $v\neq 0$ case.
In this case,
\begin{eqnarray}\label{normineq}
    \|v\|_{L_2}\leq \max\left\{A_1 \|v\|_{n},A_2n^{-m/d}\|v\|_{H^m}\right\}
\end{eqnarray}
is equivalent to
\[\left\|\frac{v}{\|v\|_{H^m}}\right\|_{L_2}\leq \max\left\{A_1 \left\|\frac{v}{\|v\|_{H^m}}\right\|_{n},A_2n^{-m/d}\left\|\frac{v}{\|v\|_{H^m}}\right\|_{H^m}\right\}.\]
This implies that we only need to show (\ref{normineq}) for $v$ with $\|v\|_{H^m}=1$. Now we invoke Lemma \ref{lemma:step2} with $R=1$ and $r=t n^{-m/d}$ for $t\geq 1$, which fulfills the condition (\ref{epsilon}). Let $C_1,C_2,C_3,C_4$ be constants suggested by Lemma \ref{lemma:step2}. We consider two cases.

\textit{Case 1)}. If $\|v\|_{L_2}<C_2 r=C_2tn^{-m/d}$, then $\|v\|_{L_2}<C_2tn^{-m/d}\|v\|_{H^m}$ is automatically true.

\textit{Case 2)}. If $\|v\|_{L_2}\geq C_2 r$. Lemma 4 implies that on an event $\Xi$ independent of $v$, we have 
    $\|v\|_{L_2}\leq C_1\|v\|_n$
and $\mathbb{P}(\Xi)\geq 1-C_3\exp\{-t^{d/m}\}$.

Combining the above two cases, we get
\begin{eqnarray*}
      \mathbb{P}\left(\|v\|_{L_2}\leq \max\left\{C_1 \|v\|_{n},C_2 t n^{-m/d}\|v\|_{H^m}\right\} \text{ for all } v\in H^m\right)\\ \geq 1-C_3\exp\{-C_4 t^{d/m}\}
\end{eqnarray*}
which completes the proof.
\end{proof}

\begin{remark}
   Theorem \ref{th:randomdesign1} improves Lemma 5.16 of \cite{geer2000empirical}. As we examine the proof, it can be seen that this improvement is mainly due to the use of the interpolation inequality (\ref{interpolation}).
\end{remark}

\subsection{Fixed Designs}

For fixed designs, we assume they are \textit{quasi-uniform}, defined as below. For a set of design points $X = \{x_1, x_2,...,x_n\}\subset \Omega$, define the fill distance as
$$h_{X, \Omega}=\sup_{x\in \Omega} \inf_{x_j \in X}||x-x_j||,$$
and the separation radius as
$$q_X = \min_{1\leq j\neq k\leq n}||x_j-x_k||/2.$$
    A set of input points $X$ is said to be \textit{quasi-uniform} in $\Omega$ if
    \[h_{X,\Omega}/q_X\leq A,\]
for some $A>0$ independent of $n$.

Suppose $\Omega$ satisfies the conditions in Assumption \ref{assum:matern}.
For quasi-uniform designs, Assumption \ref{assum:norms} is a consequence of Theorems 3.3 and 3.4 of \cite{utreras1988convergence}. Of course, here we do not need a probabilistic statement, and (\ref{l2_to_n})-(\ref{n_to_l2}) can be simplified to:
    \begin{eqnarray*}
        &&\|v\|_{L_2}\leq \max\left\{C_1 \|v\|_{n},C_2 n^{-m/d}\|v\|_{H^m}\right\},\\
        &&\|v\|_n\leq \max\left\{C_1\|v\|_{L_2},C_2 n^{-m/d}\|v\|_{H^m}\right\}.
    \end{eqnarray*}
    for all $v\in H^m$ and constants $C_1,C_2$ depending only on $\Omega,d,m$ and the quasi-uniform constant $A$.


\section{Supporting Lemmas and Technical Details}\label{sec:proofs_main}
In this section, we present supporting lemmas used in our main theorems and  proofs of the main theorems and lemmas presented in the main article.

\subsection{Proofs for Section \ref{sec:BnV}}\label{sub:gen_bias_var}

\begin{proof}[Proof of Lemma \ref{lemma:relationship}]
Denote $(u_1,\ldots,u_n)^T:=(K(X,X)+\lambda nI)^{-1}g(X)$. Use the representation 
\begin{eqnarray}\label{ghat}
    \hat{g}(\cdot)=K(\cdot,X)(K(X,X)+\lambda nI)^{-1}g(X),
\end{eqnarray}
we have
\begin{eqnarray*}
    \operatorname{BIAS}&=&g^T(X)(K(X,X)+\lambda n I)^{-1}F-\langle f,g\rangle_\mathcal{H}\\
    &=&\sum_{i=1}^n u_i f(x_i)-\langle f,g\rangle_\mathcal{H}\\
    &=&\left\langle f, \sum_{i=1}^n u_i K(\cdot,x_i)\right\rangle_\mathcal{H}-\langle f,g\rangle_\mathcal{H}\\
    &=&\left\langle f, g^T(X)(K(X,X)+\lambda n I)^{-1}K(X,\cdot)-g\right\rangle_{\mathcal{H}}\\
    &=&\langle f, \hat{g}-g\rangle_{\mathcal{H}}\\
    &\leq&\|f\|_{\mathcal{H}}\|\hat{g}-g\|_{\mathcal{H}},
\end{eqnarray*}
where the third equality follows from the reproducing property (\ref{reproducing}).
This proves the bias part.

For the variance part, it suffices to note from (\ref{ghat}) that
\begin{eqnarray}\label{u}
  \hat{g}(x_i)-g(x_i)=-\lambda n u_i.  
\end{eqnarray}
Therefore
\[\frac{1}{n}\sum_{i=1}^n (\hat{g}(x_i)-g(x_i))^2=n\lambda^2 g^T(X)(K(X,X)+\lambda nI)^{-2}g(X),\]
which proves the variance part.
\end{proof}

\subsection{Supporting Lemmas for Upper Bound in Section \ref{sec:upper}}

The following lemma states the results on the improved rates for the noiseless KRR.

\begin{lemma}\label{Th:improved}
    Under Assumptions \ref{assum:matern}-\ref{assum:norms}, on the event $\Xi_\epsilon$ introduced in Assumption \ref{assum:norms}, we have
    \begin{align*}
          &  \begin{cases}
        \|\hat{g}-g\|_n\leq 2C_gC_1^\delta\lambda^{\frac{1+\delta}{2}}\\
        \|\hat{g}-g\|_\mathcal{H}\leq 2C_gC_1^\delta\lambda^{\frac{\delta}{2}}
    \end{cases}, & \text{if } C_1\lambda^{\frac{1}{2}}\geq C_\epsilon n^{-\frac{m}{d}} & \text{ (\textbf{smoothing regime})},\\
          &  \begin{cases}
        \|\hat{g}-g\|_n\leq 2C_gC_\epsilon^\delta\lambda^{\frac{1}{2}}n^{-\frac{\delta m}{d}}\\
        \|\hat{g}-g\|_\mathcal{H}\leq 2C_gC_\epsilon^\delta n^{-\frac{\delta m}{d}}
    \end{cases}, & \text{if } C_1\lambda^{\frac{1}{2}}< C_\epsilon n^{-\frac{m}{d}} & \text{ (\textbf{interpolation regime})}.
    \end{align*}
\end{lemma}

Lemma \ref{Th:improved} shows that, depending on the choice of $\lambda$, there are two types of upper bounds. Given $\epsilon$, we say that $\lambda$ lies in the \textit{interpolation regime} if $\lambda<(C_\epsilon/C_1)^{2}n^{-2m/d}$; and otherwise, say that $\lambda$ lies in the \textit{smoothing regime}. In the interpolation regime, $\hat{g}$ behaves similarly as the kernel interpolant, i.e., the KRR estimator with $\lambda=0$. Specifically, we have seen that as $\lambda$ decreases, $\|\hat{g}-g\|_\mathcal{H}$ also decreases as $O_\mathbb{P}(\lambda^{\delta/2})$ until $\lambda$ enters the interpolation regime, and thereafter $\|\hat{g}-g\|_\mathcal{H}$ stays as $O_\mathbb{P}(n^{-\delta m/d})$. This is not surprising, as $O_\mathbb{P}(n^{-\delta m/d})$ is the limit of this estimation: it is the rate of convergence of the kernel interpolants under the same conditions; see \cite{tuo2023privacy,wendland2004scattered}.

It is important to note that the event $\Xi_\epsilon$, introduced in Assumption \ref{assum:norms}, is independent of the target function $g$. In other words, the inequalities in Lemma \ref{Th:improved} hold \textit{simultaneously} for all $g$ satisfying Assumption \ref{assum:f}. This property enables us to quantify the uniform errors in terms of $\sup_{g\in\mathcal{G}}|\langle \hat{f}-f,g\rangle_\mathcal{H}|$. Further details will be given in Section \ref{sec:uniform} in the main article.

It is also worth noting that, Lemma \ref{Th:improved} concerns noiseless KRR, which has only the bias but no variance. So there is no downside to using a small $\lambda$. In the presence of random noise, however, it is of no practical interest to choose $\lambda$ inside the interpolation regime (say, $\lambda=o(n^{-2m/d})$), because doing so will result in way too large variances! Therefore, hereafter we only consider results in the smoothing regime in an asymptotic sense, i.e., $\lambda^{-1}=O(n^{2m/d})$, for simplicity.

 Theorem \ref{Coro:rate} is an immediate consequence of Lemmas \ref{lemma:relationship} and \ref{Th:improved}.

\subsection{Proof for Section \ref{sec:assumptions}} \label{sub:general_assum}
\begin{proof}[Proof of Lemma \ref{lemma:var}]
By the definition
\[\operatorname{VAR}=\sigma^2g^T(X)(K(X,X)+\lambda n I)^{-2}g(X),\]
together with the condition $\sigma^2>0$ and that $(K(X,X)+\lambda n I)^{-2}$ is positive definite,
$\operatorname{VAR}=0$ if and only if $g(X)=0$, which implies $\|g\|_n=0$.
For any $\epsilon>0$, let $C_\epsilon$ and $\Xi_\epsilon$ be defined in Assumption \ref{assum:norms}. Because $\|g\|_{L_2}\neq 0$, for sufficiently large $n$, we have $\|g\|_{L_2}>C_\epsilon n^{-m/d}\|g\|_\mathcal{H}$. Then by (\ref{l2_to_n}), on the event $\Xi_\epsilon$, $\|g\|_{L_2}\leq C_1\|g\|_n$. This shows that $\operatorname{VAR}\neq 0$ on $\Xi_\epsilon$, and the desired result follows.
\end{proof}


\subsection{Proof for Upper Bound in Section \ref{sec:upper}} \label{sub:upp}
 \begin{proof}[Proof of Lemma \ref{Th:improved}]
    By the definition of noiseless KRR, we have the basic inequality
\[\|\hat{g}-g\|_n^2+\lambda \|\hat{g}\|^2_{\mathcal{H}}\leq \|g-g\|^2_n + \lambda \|g\|^2_\mathcal{H}=\lambda \|g\|^2_\mathcal{H},\]
which is equivalent to
\begin{eqnarray*}
    \|\hat{g}-g\|_n^2+\lambda \|\hat{g}-g\|^2_{\mathcal{H}}\leq 2 \lambda \langle g, \hat{g}-g\rangle_{\mathcal{H}}.
\end{eqnarray*}
Plugging in Assumptions \ref{assum:f}-\ref{assum:norms}, on $\Xi_\epsilon$, we have
\begin{eqnarray*}
    &&\|\hat{g}-g\|_n^2+\lambda \|\hat{g}-g\|^2_{\mathcal{H}}\leq 2 \lambda C_g \|\hat{g}-g\|_{L_2}^{\delta}\|\hat{g}-g\|_{\mathcal{H}}^{1-\delta}.\\
    &\leq& 2 \lambda C_g\max\left\{C_1^\delta\|\hat{g}-g\|^{\delta}_n\|\hat{g}-g\|_{\mathcal{H}}^{1-\delta},C_\epsilon^\delta n^{-\delta m/d}\|\hat{g}-g\|_{\mathcal{H}}\right\},
\end{eqnarray*}
which can be broken down into two cases.

\textit{Case 1)}: $\|\hat{g}-g\|_n^2+\lambda \|\hat{g}-g\|^2_{\mathcal{H}}\leq 2 \lambda C_gC_1^\delta\|\hat{g}-g\|^{\delta}_n\|\hat{g}-g\|_{\mathcal{H}}^{1-\delta}$, which implies
\begin{eqnarray}\label{needed_for_lower}
    \begin{cases}
        ~~\|\hat{g}-g\|_n^2\leq 2 \lambda C_gC_1^\delta\|\hat{g}-g\|^{\delta}_n\|\hat{g}-g\|_{\mathcal{H}}^{1-\delta},\\
        \lambda \|\hat{g}-g\|^2_{\mathcal{H}}\leq 2 \lambda C_gC_1^\delta\|\hat{g}-g\|^{\delta}_n\|\hat{g}-g\|_{\mathcal{H}}^{1-\delta}.
    \end{cases}
\end{eqnarray}
The above system can be solved with elementary algebra. The solution is
\begin{eqnarray}\label{case1}
    \begin{cases}
        \|\hat{g}-g\|_n\leq 2C_gC_1^\delta\lambda^{\frac{1+\delta}{2}},\\
        \|\hat{g}-g\|_\mathcal{H}\leq 2C_gC_1^\delta\lambda^{\frac{\delta}{2}}.
    \end{cases}
\end{eqnarray}

\textit{Case 2)}: $\|\hat{g}-g\|_n^2+\lambda \|\hat{g}-g\|^2_{\mathcal{H}}\leq 2 \lambda C_gC_\epsilon^\delta n^{-\delta m/d}\|\hat{g}-g\|_{\mathcal{H}}$, which implies
\begin{eqnarray*}
    \begin{cases}
        ~~\|\hat{g}-g\|_n^2\leq 2 \lambda C_gC_\epsilon^\delta n^{-\delta m/d}\|\hat{g}-g\|_{\mathcal{H}},\\
        \lambda \|\hat{g}-g\|^2_{\mathcal{H}}\leq 2 \lambda C_gC_\epsilon^\delta n^{-\delta m/d}\|\hat{g}-g\|_{\mathcal{H}}.
    \end{cases}
\end{eqnarray*}
The solution is
\begin{eqnarray}\label{case2}
    \begin{cases}
        \|\hat{g}-g\|_n\leq 2C_gC_\epsilon^\delta\lambda^{\frac{1}{2}}n^{-\frac{\delta m}{d}},\\
        \|\hat{g}-g\|_\mathcal{H}\leq 2C_gC_\epsilon^\delta n^{-\frac{\delta m}{d}}.
    \end{cases}
\end{eqnarray}
Clearly, if $C_1^\delta\lambda^{\delta/2}\geq C_\epsilon^\delta n^{-\delta m/d}$, (\ref{case2}) is implied by (\ref{case1}); otherwise, (\ref{case1}) is implied by (\ref{case2}). This completes the proof.
\end{proof}

\subsection{Supporting Lemmas for Lower Bound in Section \ref{sec:lower}}
In view of Lemma \ref{lemma:relationship}, we have analogous lower bounds for the noiseless KRR.

\begin{lemma}\label{Th:KRRlower}
       Suppose Assumptions \ref{assum:matern}-\ref{assum:tau} hold. Then for each $\epsilon>0$, there exist constants $A_1,A_2,A_3>0$ depending only on $C_0,C_1,C_g,C_\epsilon,R_0,\delta$, and $\tau$, such that, on the event $\Xi_\epsilon$ introduced in Assumption \ref{assum:norms}, for any $n$ and $\lambda$ satisfying $A_1n^{-2m/d}\leq \lambda \leq A_2$, we have
       \begin{eqnarray}
           \|\hat{g}-g\|_n&\geq& A_3 \lambda^{\frac{\delta\tau-\delta+2\tau}{2\tau}},\label{KRRlowern}\\
            \|\hat{g}-g\|_\mathcal{H}&\geq& 
            \begin{cases}
                A_3 \lambda^{\frac{2\tau-2\delta+\delta^2-\delta^2\tau}{2\tau(1-\delta)}} & \text{ if } \delta<1\\
                A_3 \lambda     & \text{ if } \delta=1
            \end{cases}.\label{KRRlowerH}
       \end{eqnarray}
       In particular, if $\delta=\tau$, we have       
       \begin{eqnarray}
           \|\hat{g}-g\|_n&\geq& A_3 \lambda^{\frac{1+\delta}{2}},\nonumber\\
           \|\hat{g}-g\|_\mathcal{H}&\geq& 
           \begin{cases}
               A_3 \lambda^{\frac{\delta}{2}} & \text{ if } \delta<1\\
               A_3 \lambda& \text{ if } \delta=1
           \end{cases}.\label{KRRlowerHequal}
       \end{eqnarray}
\end{lemma}

When $\delta=\tau<1$, the noiseless KRR's convergence rate is completely known.

\subsection{Proofs for Lower Bound in Section \ref{sec:lower}}
\begin{proof}[Proof of Proposition \ref{Prop:equivalence}]
    Suppose $\sup_{v\in\mathcal{H}}\frac{|\langle g,v\rangle_\mathcal{H}|}{\|v\|_{L_2}^\delta\|v\|^{1-\delta}_\mathcal{H}}= A$. Then for each $R>0$,
\begin{eqnarray*}
    \sup_{\|v\|_\mathcal{H}\leq R\|v\|_{L_2}}\frac{|\langle g,v\rangle_\mathcal{H}|}{\|v\|_{L_2}}&\leq& \sup_{\|v\|_\mathcal{H}\leq R\|v\|_{L_2}}\frac{A\|v\|_{L_2}^\delta\|v\|^{1-\delta}_\mathcal{H}}{\|v\|_{L_2}}\\
    &=& A \sup_{\|v\|_\mathcal{H}\leq R\|v\|_{L_2}}\frac{\|v\|^{1-\delta}_\mathcal{H}}{\|v\|_{L_2}^{1-\delta}}\\
    &\leq& CR^{1-\delta}.
\end{eqnarray*}
Conversely, suppose $\sup_{\|v\|_\mathcal{H}\leq R\|v\|_{L_2}}\frac{|\langle g,v\rangle_\mathcal{H}|}{\|v\|_{L_2}}\leq C R^{1-\delta}$ for each $R>0$. First we note that, under Assumption \ref{assum:matern}, $\|\cdot\|_\mathcal{H}$ is stronger than $\|\cdot\|_{L_2}$, which means $\|v\|_{L_2}/\|v\|_\mathcal{H}\leq A_1$ for all $v\in\mathcal{H}$.
Then for each $v\in\mathcal{H}$ satisfying $\|v\|_\mathcal{H}\leq \|v\|_{L_2}$, we have
\begin{eqnarray}\label{v01}
    \frac{|\langle g,v\rangle_\mathcal{H}|}{\|v\|_{L_2}^\delta\|v\|^{1-\delta}_\mathcal{H}}=\frac{|\langle g,v\rangle_\mathcal{H}|}{\|v\|_{L_2}}\cdot\frac{\|v\|_{L_2}^{1-\delta}}{\|v\|^{1-\delta}_\mathcal{H}}\leq C A_1^{1-\delta}.
\end{eqnarray}
Next, for each $i=1,2,\ldots$ and $v\in\mathcal{H}$ satisfying $2^{i-1}\leq \|v\|_\mathcal{H}/\|v\|_{L_2}\leq 2^i$,
\begin{eqnarray}\label{vi}
    \frac{|\langle g,v\rangle_\mathcal{H}|}{\|v\|_{L_2}^\delta\|v\|^{1-\delta}_\mathcal{H}}&=&\frac{|\langle g,v\rangle_\mathcal{H}|}{\|v\|_{L_2}}\cdot\frac{\|v\|_{L_2}^{1-\delta}}{\|v\|^{1-\delta}_\mathcal{H}}\leq C (2^i)^{1-\delta}\cdot (2^{1-i})^{1-\delta}\nonumber
    \\&=&2^{1-\delta}C.
\end{eqnarray}
Combining (\ref{v01}) and (\ref{vi}) leads to
\[\sup_{v\in\mathcal{H}}\frac{|\langle g,v\rangle_\mathcal{H}|}{\|v\|_{L_2}^\delta\|v\|^{1-\delta}_\mathcal{H}}\leq \max\{2^{1-\delta},A_1^{1-\delta}\}C<+\infty.\]
This completes the proof.
\end{proof}

\begin{proof}[Proof of Theorem \ref{Th:lower}]
    Note that $\operatorname{VAR}=\sigma^2g^T(X)(K(X,X)+\lambda n I)^{-2}g(X)$.
    By Cauchy-Schwarz inequality, we have
    \begin{eqnarray}\label{lower}
        [\mathbf{v}(K(X,X)+\lambda n I)^{-1}g(X)]^2&\leq &\mathbf{v}^T\mathbf{v}\cdot g^T(X)\cdot\nonumber\\
        &&(K(X,X)+\lambda n I)^{-2}g(X),
    \end{eqnarray}
    for each $\mathbf{v}\in\mathbb{R}^n$. Now take $\mathbf{v}=v(X)$ for some $v\in\mathcal{H}$. By (\ref{ghat}), 
    \[\mathbf{v}(K(X,X)+\lambda n I)^{-1}g(X)=\langle v,\hat{g}\rangle_\mathcal{H};\]
    thus (\ref{lower}) implies
    \begin{eqnarray}\label{lowersup}
        g^T(X)(K(X,X)+\lambda n I)^{-2}g(X)\geq \sup_{v\in\mathcal{H}}\frac{\langle v,\hat{g}\rangle^2_\mathcal{H}}{n\|v\|^2_n}.
    \end{eqnarray}
    (Actually, the equality holds as $v(X)$ can go over the entire $\mathbb{R}^n$.) In view of (\ref{lowersup}), the strategy is to bound $\langle v,\hat{g}\rangle_\mathcal{H}/\|v\|_n$ from below with a carefully chosen $v$.


    
    To proceed, we need to get rid of the annoying $\|\cdot\|_n$ norm and the KRR estimator. This can be done by invoking the bounds in Assumption \ref{assum:norms} and Lemma \ref{Th:improved}, which state that on the event $\Xi_\epsilon$,
    \begin{eqnarray*}
        \|v\|_n/\|v\|_{L_2}\leq \max\left\{C_1,C_\epsilon n^{-m/d}\|v\|_{\mathcal{H}}/\|v\|_{L_2}\right\},
    \end{eqnarray*}
    and
    \begin{eqnarray*}
        |\langle v,\hat{g}-g\rangle_\mathcal{H}|\leq \|v\|_\mathcal{H}\|\hat{g}-g\|_\mathcal{H} \leq 2C_gC_1^\delta\lambda^{\frac{\delta}{2}}\|v\|_\mathcal{H}.
    \end{eqnarray*}
    Therefore, for any $v$ satisfying 
    \begin{eqnarray}\label{vcondition}
        C_\epsilon n^{-m/d}\|v\|_{\mathcal{H}}/\|v\|_{L_2}\leq C_1,
    \end{eqnarray}
    we have
    \begin{eqnarray}\label{lower2}
        \frac{\langle v,\hat{g}\rangle_\mathcal{H}}{\|v\|_n}&&\geq \frac{\langle v,\hat{g}\rangle_\mathcal{H}}{C_1\|v\|_{L_2}}=\frac{\langle v,g\rangle_\mathcal{H}}{C_1\|v\|_{L_2}}+\frac{\langle v,\hat{g}-g\rangle_\mathcal{H}}{C_1\|v\|_{L_2}}\nonumber\\
        &&\geq \frac{\langle v,g\rangle_\mathcal{H}}{C_1\|v\|_{L_2}}-2C_gC_1^{\delta-1}\lambda^{\frac{\delta}{2}}\frac{\|v\|_\mathcal{H}}{\|v\|_{L_2}}.
    \end{eqnarray}

    Assumption \ref{assum:tau} implies that for each $R\geq R_0$, there exists $v\in\mathcal{H}$ such that $\|v\|_{\mathcal{H}}/\|v\|_{L_2}\leq R$ and $R^{1-\tau}\langle v,g\rangle_\mathcal{H}/\|v\|_{L_2}>C_0$. Using this specific $v$ in (\ref{lower2}) leads to
    \begin{eqnarray}\label{lower3}
        \frac{\langle v,\hat{g}\rangle_\mathcal{H}}{\|v\|_n}\geq \frac{C_0 R^{1-\tau}}{C_1}-2C_gC_1^{\delta-1}\lambda^{\frac{\delta}{2}}R.
    \end{eqnarray}

    Our goal is to make the right-hand side of (\ref{lower3}) no less than $\frac{C_0R^{1-\tau}}{2C_1}$, which requires taking $R$ no more than $4^{-1/\tau}C_0^{1/\tau}C_1^{-\delta/\tau}C_g^{-1/\tau}\lambda^{-\delta/(2\tau)}$. Clearly, we can find suitable constants $A_1,A_2$ depends only on $C_0,C_1,C_g,C_\epsilon,R_0,\delta,\tau$, such that for each $\lambda$ satisfying
    \[A_1 n^{-\frac{2m}{d}}\leq \lambda\leq A_2,\]
    we have
    \begin{eqnarray}\label{R0}
        R_0\leq 4^{-1/\tau}C_0^{1/\tau}C_1^{-\delta/\tau}C_g^{-1/\tau}\lambda^{-\delta/(2\tau)}\leq \frac{C_1n^{\frac{m}{d}}}{C_\epsilon}.
    \end{eqnarray}
    This implies that the choice $R=4^{-1/\tau}C_0^{1/\tau}C_1^{-\delta/\tau}C_g^{-1/\tau}\lambda^{-\delta/(2\tau)}$ fulfills the conditions $R\geq R_0$ and (\ref{vcondition}), and leads to
    \[\frac{\langle v,\hat{g}\rangle_\mathcal{H}}{\|v\|_n}\geq \frac{C_0 R^{1-\tau}}{2C_1},\]
    which, together with (\ref{lowersup}), completes the proof.
\end{proof}

 \begin{proof}[Proof of Lemma \ref{Th:KRRlower}]
        The relationship between $\|\hat{g}-g\|_n$ and $\operatorname{VAR}$ in Lemma \ref{lemma:relationship} and Theorem \ref{Th:lower} lead to (\ref{KRRlowern}) immediately. To bound $\|\hat{g}-g\|_\mathcal{H}$ from below, we first make a possible adjustment of $A_1$ from that given by Theorem \ref{Th:lower}, so that $\lambda$ lies in the smoothing regime defined in Lemma \ref{Th:improved}. Then we resort to the first inequality in (\ref{needed_for_lower}) from the proof of Lemma \ref{Th:improved}, which states
        \begin{eqnarray}\label{lowerKRR1}
            \|\hat{g}-g\|_n^2\leq 2 \lambda C_gC_1^\delta\|\hat{g}-g\|^{\delta}_n\|\hat{g}-g\|_{\mathcal{H}}^{1-\delta}.
        \end{eqnarray}
        When $\delta<1$, we substitute the lower bound of $\|\hat{g}-g\|_n$ to (\ref{lowerKRR1}), and arrive at the first part of (\ref{KRRlowerH})
        by elementary algebraic calculations. For $\delta=1$, we note that Assumption \ref{assum:f} is also true for any $\delta'<1$. We then invoke the first part of (\ref{KRRlowerH}), which we just proved, by substituting $\delta\leftarrow\delta'$ and $\tau\leftarrow 1$. The resulting lower bound is $A_3 \lambda$, regardless of the choice of $\delta'$.
    \end{proof}

Combining Lemmas \ref{Th:improved}, \ref{Th:KRRlower} and Corollary \ref{coro:worstcase}, we obtain Theorem \ref{Coro:biasrate}.

\subsection{Proofs for Improved Results on $\operatorname{BIAS}$}
\begin{proof}[Proof of Theorem \ref{Th:oBIAS}]
    We shall use the eigensystem representation of the RKHS norm in this proof. We follow the notation introduced in Section \ref{sec: eigen_series} and denote $f=\sum_{i=1}^\infty c_i \eta_i$ and $\hat{g}-g=\sum_{i=1}^\infty a_i \eta_i$. By Lemma \ref{lemma:relationship},
    \begin{eqnarray}\label{bias_series}
        |\operatorname{BIAS}|=|\langle \hat{g}-g,f\rangle_\mathcal{H}|=\left|\sum_{i=1}^\infty \frac{a_ic_i}{\rho_i}\right|.
    \end{eqnarray}
    Basic calculus suggests that we can find an infinite series which converges slower than    $
        \sum_{i=1}^\infty c_i^2/\rho_i$. In other words, there exists a sequence $\xi_i\downarrow 0$, such that
        $\sum_{i=1}^\infty \frac{c_i^2}{\rho_i\xi_i}<\infty$. Now we apply the Cauchy-Schwarz inequality to (\ref{bias_series}) to find
        \[|\operatorname{BIAS}|\leq \left(\sum_{i=1}^\infty\frac{a_i^2\xi_i}{\rho_i}\right)^{1/2}\left(\frac{c_i^2}{\rho_i\xi_i}\right)^{1/2}.\]
        Now it suffices to prove that $\sum_{i=1}^\infty a_i^2\xi_i/\rho_i=o_\mathbb{P}(\lambda^\delta)$. Because $\xi_i\downarrow 0$, for any $\epsilon>0$, there exists $N$ such that $\xi_i<\epsilon$ for all $i>N$. We now write
        \begin{eqnarray*}
            &&\sum_{i=1}^\infty \frac{a_i^2 \xi_i}{\rho_i}=\left(\sum_{i=1}^N+\sum_{i=N+1}^\infty\right) \frac{a_i^2 \xi_i}{\rho_i}\\
            &\leq &\left(\max_{1\leq i\leq N}\frac{\xi_i}{\rho_i}\right)\sum_{i=1}^\infty a_i^2+\epsilon \sum_{i=1}^\infty \frac{a_i^2}{\rho_i}\\
            &=&\left(\max_{1\leq i\leq N}\frac{\xi_i}{\rho_i}\right)\|\hat{g}-g\|^2_{L_2}+\epsilon\|\hat{g}-g\|^2_\mathcal{H}.
        \end{eqnarray*}
        Employing Assumption \ref{assum:norms} and Theorem \ref{Coro:rate}, together with the condition $\lambda=o(1)$, on $\Xi_\epsilon$, the above equation is no more than
        \begin{eqnarray*}
            &&\left(\max_{1\leq i\leq N}\frac{\xi_i}{\rho_i}\right)\max\left\{C_1^2 \|\hat{g}-g\|_{n}^2,C_\epsilon^2 n^{-m/d}\|\hat{g}-g\|^2_{\mathcal{H}}\right\}+\epsilon\|\hat{g}-g\|^2_\mathcal{H}\\
            &=&\left(\max_{1\leq i\leq N}\frac{\xi_i}{\rho_i}\right)o(\lambda^\delta)+\epsilon O(\lambda^\delta).
        \end{eqnarray*}
        This proves $\sum_{i=1}^\infty a_i^2\xi_i/\rho_i=o_\mathbb{P}(\lambda^\delta)$ as $\epsilon$ is arbitrary.
\end{proof}

\begin{proof}[Proof of Theorem \ref{Coro:further}]
    The desired result follows from
using (\ref{smoothnessf}) instead of the Cauchy-Schwarz inequality in (\ref{cauchy}), together with Lemma \ref{Th:improved}.
\end{proof}

 \subsection{Supporting Lemmas for Asymptotic Normality in Section \ref{sec:AN}} 

The following Lemma \ref{lemma:CLT} is a consequence of the Lindeberg central limit theorem, and it is the key lemma for proving the asymptotic normality of our estimator. We use the notion ``$\xrightarrow{\mathscr{L}}$'' to denote the convergence in distribution.

\begin{lemma}\label{lemma:CLT}
Suppose $\sigma^2\in(0,\infty)$ is independent of $n$, and $g\neq 0$. The design points $X$ are either deterministic, or random but independent of the random error $E$.
    If
    \begin{eqnarray}\label{Lindeberg}
        n^{-1/2}\frac{\|\hat{g}-g\|_{L_\infty}}{\|\hat{g}-g\|_n}\xrightarrow{p} 0, \text{as } n\rightarrow\infty,
    \end{eqnarray}
   then we have the central limit theorem
   \begin{eqnarray}
       \frac{1}{\sqrt{\operatorname{VAR}}}g^T(X)(K(X,X)+\lambda n I)^{-1}E\xrightarrow{\mathscr{L}} N(0,1),  \text{ as } n\rightarrow\infty.
   \end{eqnarray}
\end{lemma}

We can verify (\ref{Lindeberg}) provided that we have an upper bound of $\|\hat{g}-g\|_{L_\infty}$ and a lower bound of $\|\hat{g}-g\|_n$.
The final result is given in Theorem \ref{Th:AN}.

    \subsection{Proofs for Asymptotic Normality in Section \ref{sec:AN}} \label{sub:normlity}

\begin{proof}[Proof of Lemma \ref{lemma:CLT}]
    For clarity, we shall reinstate the subscript $n$ for each term depending on $n$ in this proof. For instance, we will denote $X$ by $X_n$ to emphasize its dependence on $n$.
    Again, we define $(u_{1n},\ldots,u_{nn})^T:=(K(X_n,X_n)+\lambda_n nI_n)^{-1}g(X_n)$. Then
    \[g^T(X_n)(K(X_n,X_n)+\lambda_n n I_n)^{-1}E_n=\sum_{i=1}^n u_{in} e_{i}.\]
    First, we regard $X_n$ as a fixed sequence, i.e., we set conditioning on $X_n$ if the design is random. Then $u_{in}$'s are fixed. Then we shall have the central limit theorem
    \[\frac{1}{\sigma\sqrt{\sum_{i=1}^n u_{in}^2}}\sum_{i=1}^n u_{in}e_i\xrightarrow{d} N(0,1),\]
    provided that the Lindeberg condition
    \begin{eqnarray}\label{lindeberg}
        \lim_{n\rightarrow \infty}\frac{1}{\sigma^2\sum_{i=1}^n u_{in}^2}\sum_{i=1}^n \mathbb{E}\left[u_{in}^2 e_i^2 \mathbf{1}_{\{u_{in}^2e_i^2>\epsilon^2 \sigma^2 \sum_{j=1}^n u_{jn}^2\}}\right]=0
    \end{eqnarray}
    is fulfilled. It is easily seen that a sufficient condition of (\ref{lindeberg}) is (see also Lemma 3.1 of \cite{huang2003local}, and the proof of Theorem \ref{Th:multivariateAN})
    \begin{eqnarray}\label{lindebergsufficient}
        \max_{1\leq i\leq n}u_{in}^2/\sum_{i=1}^n u_{in}^2\rightarrow 0, \text{ as } n\rightarrow \infty.
    \end{eqnarray}
    This is equivalent to, by (\ref{u}),
    \begin{eqnarray}\label{pointwise}
        \frac{\max_{1\leq i\leq n}(\hat{g}_n-g)^2(x_{in})}{\sum_{i=1}^n(\hat{g}_n-g)^2(x_{in})}\rightarrow 0, \text{ as } n\rightarrow\infty,
    \end{eqnarray}
    which is ensured by (\ref{Lindeberg}), except that (\ref{Lindeberg}) converges only in probability. 
    In fact, the convergence in probability in the Lindeberg condition still leads to the central limit theorem, because $X$ is independent of $E$; see, e.g., Theorem 1 on page 171 of \cite{pollard1984convergence}.
\end{proof}

\begin{proof}[Proof of Theorem \ref{Th:AN}]
    The interpolation inequality (\ref{interpolation}), together with Assumption \ref{assum:matern}, implies that
    \[\|\hat{g}-g\|_{L_\infty}\leq C\|\hat{g}-g\|_{L_2}^{1-\frac{d}{2m}}\|\hat{g}-g\|_{\mathcal{H}}^{\frac{d}{2m}}.\]
    Invoke Assumption \ref{assum:norms}, we know on $\Xi_\epsilon$,
    \begin{eqnarray}
        \|\hat{g}-g\|_{L_\infty}&\leq& C \max\left\{C_1 \|\hat{g}-g\|_n^{1-\frac{d}{2m}}\|\hat{g}-g\|_{\mathcal{H}}^{\frac{d}{2m}}, C_\epsilon n^{-\frac{m}{d}\left(1-\frac{d}{2m}\right)}\|\hat{g}-g\|_{\mathcal{H}}\right\}\nonumber\\
        &\leq& 2 C C_g C_1 \max\left\{C_1(\lambda^{\frac{1+\delta}{2}})^{1-\frac{d}{2m}}\lambda^{\frac{\delta d}{4m}},C_\epsilon n^{-\frac{m}{d}\left(1-\frac{d}{2m}\right)} \lambda^{\frac{\delta}{2}}\right\}\nonumber\\
        &=& 2 C C_g C_1\max\{C_1\lambda^{\frac{1+\delta}{2}-\frac{d}{4m}},n^{-\frac{2m-d}{2d}}\lambda^{\frac{\delta}{2}}\}\nonumber\\
        &=& O(\lambda^{\frac{1+\delta}{2}-\frac{d}{4m}}),\label{deltachoice}
    \end{eqnarray}
    where the second inequality follows from Lemma \ref{Th:improved}, and the last equality follows from the condition $\lambda^{-1}=O(n^{2m/d})$. 
    Combining the above upper bound of $\|\hat{g}-g\|_{L_\infty}$ with the lower bound given in Lemma \ref{Th:KRRlower} and condition (\ref{conditionlambda}) leads to (\ref{Lindeberg}). Then we invoke Lemma \ref{lemma:CLT} to arrive at the desired result.
\end{proof}

\subsection{Proofs for Examples in Section \ref{sec:examples}} \label{sub:example}
To prove Theorems \ref{Th:point} and \ref{Th:derivative}, it suffices to show that $\delta=\tau$ in both cases, which is implied by 
Proposition \ref{Prop:point}.
\begin{proposition}\label{Prop:point}
    Let $\alpha$ be a multi-index. Suppose $m>d/2+|\alpha|$.
    For each interior point of $\Omega$, denoted as $x_0$, there exist $A_1,A_2>0$, such that
    \[\sup_{\|v\|_{H^m}\leq R\|v\|_{L_2}}\frac{D^\alpha v(x_0)}{\|v\|_{L_2}}\geq A_1 R^{\frac{d+2|\alpha|}{2m}},\]
    for all $R\geq A_2$.
\end{proposition}

\begin{proof}
    Choose a function $B(x)$ such that $B(x)\in C^\infty(\mathbb{R}^d)$ and $B(x)$ is supported in the unit ball of $\mathbb{R}^d$. The function $D^\alpha B$ must be nonzero at some point. Without loss of generality, we assume that $D^\alpha B(0)=1$, because if otherwise, we can translate, dilate, and rescale $B$ to make this happen.
    Define $w(x)=B((x-x_0)/\rho)$ for  $\rho\in(0,1)$. For each multi-index $\beta$, the chain rule implies
    \begin{eqnarray}\label{dbetav}
        D^\beta w(x)=\rho^{-|\beta|}D^\beta B((x-x_0)/\rho).
    \end{eqnarray}
    Thus for each sufficiently small $\rho$ such that $w$ is supported in $\Omega$, we have
    \[\int_\Omega [D^\beta w (x)]^2dx\leq \rho^{-2|\beta|+d}\int_{\mathbb{R}^d} [D^\beta B(x)]^2dx,\]
    which implies that $\|w\|_{H^k}\leq \rho^{-(k-d/2)}\|B\|_{H^k(\mathbb{R}^d)}$ for integer $k$. In particular, we have 
    \[\|w\|_{L_2}=\rho^{d/2}\|B\|_{L_2(\mathbb{R}^d)}.\]
    If $m$ is not an integer, direct calculations show the Sobolev-Slobodeckij semi-norm in (\ref{Sobolev}) satisfies
    \begin{eqnarray*}
        |w|_{W^m}=\rho^{-(m-d/2)}|B|_{W^m(\mathbb{R}^d)},
    \end{eqnarray*}
    which, again, implies $\|w\|_{H^m}\leq \rho^{-(m-d/2)}\|B\|_{H^m(\mathbb{R}^d)}$. Besides, (\ref{dbetav}) also shows $D^\alpha w(x_0)=\rho^{-|\alpha|}$. In summary, for sufficiently small $\rho$, we have
    \begin{eqnarray*}
        \sup_{\frac{\|v\|_{H^m}}{\|v\|_{L_2}}\leq\rho^{-m}\frac{\|B\|_{H^m(\mathbb{R}^d)}}{\|B\|_{L_2(\mathbb{R}^d)}}}\frac{D^\alpha v(x_0)}{\|v\|_{L_2}}\geq \frac{D^\alpha w(x_0)}{\|w\|_{L_2}}=\frac{\rho^{-|\alpha|-d/2}}{\|B\|_{L_2(\mathbb{R}^d)}},
    \end{eqnarray*}
    which leads to the desired result by replacing $R=\rho^{-m}$.    
\end{proof}

\begin{proof}[Proof of Theorem \ref{Th:multivariateAN}]
    Without loss of generality, we assume that $\sigma^2=1$. First, we prove that $\operatorname{COV}$ is invertible with probability tending to one. It suffices to prove that the smallest eigenvalue of $\operatorname{COV}$, denoted by $\underline{\lambda}$, is positive. Note that
    \begin{eqnarray}\label{mineigen}
        \underline{\lambda}=\min_{\|\mathbf{a}\|=1}\mathbf{a}^T\operatorname{COV}\mathbf{a},
    \end{eqnarray}
    where $\mathbf{a}=(a_1,\ldots,a_{d_0})^T$. By the definition of $\operatorname{COV}$ in (\ref{cov}),
    \begin{eqnarray}\label{covquadratic}
        \mathbf{a}^T\operatorname{COV}\mathbf{a}&=&\left(\sum_{i=1}^{d_0}a_iD^{\alpha_i} K(z_i,X)\right)(K+\lambda n I)^{-2}\left(\sum_{i=1}^{d_0}a_iD^{\alpha_j}K(X,z_j)\right)\nonumber\\
        &=&g_a^T(X)(K+\lambda n I)^{-2}g_a(X),
    \end{eqnarray}
    for $g_a=\sum_{i=1}^{d_0}a_iD^{\alpha_j}K(\cdot,z_j)$. Because $(\alpha_i,z_i)$'s are distinct, $D^{\alpha_j}K(\cdot,z_j)$'s are linearly independent, and therefore $g_a\neq 0$ for any $\|a\|=1$. Because $\|g\|_{L_2}/\|g_a\|_\mathcal{H}$, as a function of $a$, is continuous over the unit sphere $\{a:\|a\|=1\}$, $\|g\|_{L_2}/\|g_a\|_\mathcal{H}$ has an attainable infimum, denoted as $\underline{r}>0$. Now for any $\epsilon>0$, let $C_\epsilon$ and $\Xi_\epsilon$ be defined as in Assumption \ref{assum:norms}. Then for $n>(C_\epsilon/\underline{r})^{d/m}$, $\|g_a\|_{L_2}>C_\epsilon n^{-m/d}\|g_a\|_\mathcal{H}$ for all $\|a\|=1$. Then by (\ref{l2_to_n}), on the event $\Xi_\epsilon$, $\|g_a\|_{L_2}\leq C_1\|g_a\|_n$ for all $\|a\|=1$. This shows that $\mathbf{a}^T\operatorname{COV}\mathbf{a}\neq 0$ for all $\|a\|=1$, and implies that $\operatorname{COV}$ is invertible.

    Now assume that $\alpha_i$'s are homogenous and denote $k:=|\alpha_i|$.
    Let us establish a lower bound of $\underline{\lambda}_n$ for the future use. By (\ref{mineigen}) and (\ref{covquadratic}), it suffices to find a lower bound of the variance term of $\langle \hat{f}-f,g_a\rangle_\mathcal{H}$. In order to invoke Theorem \ref{Th:lower}, we need to verify Assumption \ref{assum:tau} for $g_a$. The idea is similar to the proof of Proposition \ref{Prop:point} but with more involved details. Without loss of generality, we assume that $a_i\neq 0$ for each $i$.

        First, we group the triads $(a_i,\alpha_i,z_i)$'s based on the value of $z_i$: each group has a common $z_i$, and different groups have different $z_i$. Denote the groups by $\mathcal{G}_1,\ldots,\mathcal{G}_J$. Again, each group consists of triads $(a_i,\alpha,z_i)$ with the same $z_i$ value. Then the linear functional $\langle g_a,\cdot\rangle_\mathcal{H}$ can be rewritten as
    \begin{eqnarray}\label{vinner}
        \langle g_a,v\rangle_\mathcal{H}=\sum_{j=1}^J\sum_{(a_i,\alpha_i,z_i)\in\mathcal{G}_j}a_iD^{\alpha_i}v(z_i).
    \end{eqnarray}
    The goal is to construct $v$ under the condition $\|v\|_\mathcal{H}/\|v\|_{L_2}\leq R$, such that $\langle g_a,v\rangle_\mathcal{H}/\|v\|_{L_2}$ reaches the optimal order of magnitude. For a moment, suppose that, for each $j=1,\ldots,J$,  we can find a function $B_j\in C^\infty(\mathbb{R}^d)$, such that
    \begin{eqnarray}\label{Bbound}
        \sum_{(a_i,\alpha_i,z_i)\in\mathcal{G}_j}a_iD^{\alpha_i}B_j(0)\geq \frac{1}{2}\sum_{(a_i,\alpha_i,z_i)\in\mathcal{G}_j}a_i^2,
    \end{eqnarray}
    and $B_j$ is supported in the unit ball of $\mathbb{R}^d$. Now define $v_j(x):=B_j((x-z_j)/\rho)$ for $\rho\in (0,1)$ and $v:=\sum_{j=1}^J v_j$. Clearly, if $\rho$ is sufficiently small, $v_j$'s have disjoint supports, and thus
    \[\|v\|_{L_2}^2=\sum_{i=1}^{J}\|v_j\|_{L_2}^2, ~~~ \|v\|_{H^m}^2=\sum_{i=1}^{J}\|v_j\|_{H^m}^2.\]
    By the calculations in the proof of Proposition \ref{Prop:point}, we have
    \begin{eqnarray*}
        D^{\alpha_i}v_j(z_i)&=&\rho^{-k}D^{\alpha_i}B_j(0),\\
        \|v\|_{L_2}&=&\rho^{d/2}\left(\sum_{i=1}^{J}\|B_j\|_{L_2}^2\right)^{1/2}=:\rho^{d/2} A_1,\\
        \|v\|_{H^m}&\leq&\rho^{-(m-d/2)}\left(\sum_{i=1}^{J}\|B_j\|_{H^m}^2\right)^{1/2}:=\rho^{-(m-d/2)} A_2.
    \end{eqnarray*}
    On the other hand, we have
    \begin{eqnarray*}
        \frac{\langle g_a,v\rangle_\mathcal{H}}{\|v\|_{L_2}}&=&\frac{\sum_{j=1}^J\sum_{(a_i,\alpha_i,z_i)\in\mathcal{G}_j}a_iD^{\alpha_i}v_j(z_i)}{\rho^{d/2} A_1}\\
        &=&\frac{\rho^{-k}\sum_{j=1}^J\sum_{(a_i,\alpha_i,z_i)\in\mathcal{G}_j}a_iD^{\alpha_i}B_j(0)}{\rho^{d/2} A_1}\\
        &\geq&\frac{1}{2}\rho^{-k-d/2}\sum_{i=1}^n a_i^2 /A_1=\frac{1}{2A_1}\rho^{-k-d/2}.
    \end{eqnarray*}
    Setting $R=\rho^m$ implies Assumption \ref{assum:tau} with $\tau=1-\frac{2k+d}{2m}$.

    Now we prove the existence of $B_j$'s subject to (\ref{Bbound}) and the compact supportedness condition. A simple configuration that fulfills (\ref{Bbound}) is to ensure 
    \begin{eqnarray}\label{hermite}
        D^{\alpha_i}B_j(0)=a_i, \text{ whenever } (a_i,\alpha_i,z_i)\in\mathcal{G}_j.
    \end{eqnarray}
    Building a function $B_j$ subject to (\ref{hermite}) can be done by a multivariate Hermite interpolation. For example, we can use kring \cite{MorrisMitchellYlvisaker93,zongmin1992hermite} with a Gaussian kernel to produce a function in $C^\infty(\mathbb{R}^d)$ that satisfies (\ref{hermite}). Denote such a function by $B_{j1}$. To introduce the compact supportedness, define
    \[
    B_{j2}(x):=
    \begin{cases}
        B_{j1}(x) & \text{if } \|x\|\leq 1/2,\\
        0       & \text{otherwise.}
    \end{cases}
    \]
    Then we smooth $B_{j2}$ via a convolution. Choose $\varphi\in C^\infty(\mathbb{R}^d)$ supported in the unit ball of $\mathbb{R}^d$ with $\int_{\mathbb{R}^d}\varphi(x)dx=1$. Let $\varphi_\rho:=\rho^{-d}\varphi(\cdot/\rho)$, and $B_{j3}(x;\rho)=\int_{\mathbb{R}^d}B_{j2}(x-t)\varphi_\rho(t)dt$ for small $\rho$. Then $B_{j3}(\cdot;\rho)\in C^\infty(\mathbb{R}^d)$, $B_{j3}(\cdot;\rho)$ is supported in the unit ball of $\mathbb{R}^d$, and $\lim_{\rho\downarrow 0}D^{\alpha_i}B_{j3}(0;\rho)=D^{\alpha_i}B_{j1}(0)$. Therefore, we can set $B_j=B_{j3}(\cdot;\rho)$ for sufficiently small $\rho$ such that (\ref{hermite}) is satisfied.

    To verify Assumption \ref{assum:f}, we use the interpolation inequality (\ref{interpolationD}) to show that
    \begin{eqnarray*}
        \left|\sum_{i=1}^{d_0}a_i D^{\alpha_i}g(z_i)\right|&\leq& \left(\sum_{i=1}^{d_0}a_i^2\right)^{1/2}\left(\sum_{i=1}^{d_0}[D^{\alpha_i}g(z_i)]^2\right)^{1/2}\\
        &\leq&d_0^{\frac{1}{2}}A\|v\|^{1-\frac{2k+d}{2m}}_{L_2}\|v\|^{\frac{2k+d}{2m}}_{H^m},
    \end{eqnarray*}
    where $A$ is given in (\ref{interpolationD}).
    Thus we have verified Assumption \ref{assum:f} with $\delta=1-\frac{2k+d}{2m}$.
    Since $\delta=\tau$, we are ready to invoke Corollary \ref{Coro:varrate} to obtain that $\mathbf{a}^T\operatorname{COV}\mathbf{a}\geq A_3 n^{-1}\lambda^{-\frac{2k+d}{2m}}$, for some $A_3>0$. Note that the constants we established for Assumptions \ref{assum:f} and \ref{assum:tau} are independent of $\mathbf{a}$. Thus $A_3$ is also independent of $\mathbf{a}$, which implies that on the event $\Xi_\epsilon$,
    \begin{eqnarray}\label{eigenlower}
        \underline{\lambda}\geq A_3n^{-1}\lambda^{-\frac{2k+d}{2m}}.
    \end{eqnarray}
    
    Next, we move to the central limit theorem. We shall use the notation similar to the proof of Lemma \ref{lemma:CLT}, by reinstating the subscript $n$. Again, we assume that $X_n$ is fixed, which is equivalent to conditioning on $X_n$.
    Denote $\left(u_{1n}^{(i)},\ldots,u_{nn}^{(i)}\right)^T:=(K(X_n,X_n)+\lambda_n n I_n)^{-1}D^{\alpha_i}K(X_n,z_i)$. Define
    \[\mathbf{u}_{n,i}:=\left(u_{in}^{(1)},\ldots,u_{in}^{(d_0)}\right)^T.\]
    Then
    \[\begin{pmatrix}
        D^{\alpha_1} K(z_1,X)\\
        \vdots\\
        D^{\alpha_{d_0}}K(z_{d_0},X)
    \end{pmatrix}(K(X_n,X_n)+\lambda_n n I)^{-1}E=\sum_{i=1}^n \mathbf{u}_{n,i}e_i.\]
    We now use a version of the multivariate Lindeberg central limit theorem \cite{hansen2022econometrics}, which ensures the desired result provided that
    \begin{eqnarray}\label{MVLindeberg}
        \lim_{n\rightarrow \infty}\frac{1}{\underline{\lambda}_n}\sum_{i=1}^n\mathbb{E}\left[\|\mathbf{u}_{n,i}\|^2e_i^2\mathbf{1}_{\{\|\mathbf{u}_{n,i}\|^2e_i^2\geq \varepsilon^2 \underline{\lambda}_n\}}\right]=0,
    \end{eqnarray}
    for each $\varepsilon>0$, where $\underline{\lambda}_n$ denotes the minimum eigenvalue of $\operatorname{COV}_n$. In view of (\ref{eigenlower}), on the event $\Xi_{\epsilon,n}$
    \begin{eqnarray}
        &&\frac{1}{\underline{\lambda}_n}\sum_{i=1}^n\mathbb{E}\left[\|\mathbf{u}_{n,i}\|^2e_i^2\mathbf{1}_{\{\|\mathbf{u}_{n,i}\|^2e_i^2\geq \varepsilon^2 \underline{\lambda}_n\}}\right]\nonumber\\
        &\leq& A_3^{-1}n\lambda_n^{\frac{2k+d}{2m}}\sum_{i=1}^n\mathbb{E}\left[\|\mathbf{u}_{n,i}\|^2e_i^2\mathbf{1}_{\left\{\|\mathbf{u}_{n,i}\|^2e_i^2\geq \varepsilon^2 A_3n^{-1}\lambda_n^{-\frac{2k+d}{2m}}\right\}}\right].\label{MVe1}
    \end{eqnarray}
    On the other hand, let $g_j=D^{\alpha_j} K(z_j,X)$ for $j=1,\ldots,d_0$. Then, on $\Xi_{\epsilon,n}$, we have
    \begin{eqnarray}
        \|\mathbf{u}_{n,i}\|^2&=&\sum_{j=1}^{d_0}\left[u^{(j)}_{in}\right]^2=\lambda_n^{-2}n^{-2}\sum_{j=1}^{d_0}\left[\hat{g}_{jn}(x_{in})-g_j(x_{in})\right]^2\nonumber\\
        &\leq&\lambda_n^{-2}n^{-2}\sum_{j=1}^{d_0}\|\hat{g}_{jn}-g_j\|^2_{L_\infty}\nonumber\\
        &\leq & A_4 n^{-2}\lambda_n^{-\frac{k+d}{m}},\label{MVe2}
    \end{eqnarray}
    where the second equality follows from (\ref{u}); and the last inequality follows from (\ref{deltachoice}) and $A_4>0$ is a constant independent of $n$ and $\lambda_n$. Combining (\ref{MVe1}) and (\ref{MVe2}), we obtain that on $\Xi_{\epsilon,n}$,
    \begin{eqnarray}
        &&\frac{1}{\underline{\lambda}_n}\sum_{i=1}^n\mathbb{E}\left[\|\mathbf{u}_{n,i}\|^2e_i^2\mathbf{1}_{\{\|\mathbf{u}_{n,i}\|^2e_i^2\geq \varepsilon^2 \underline{\lambda}_n\}}\right]\nonumber\\
        &\leq&A_3^{-1}n\lambda_n^{\frac{2k+d}{2m}}\sum_{i=1}^n\mathbb{E}\left[\|\mathbf{u}_{n,i}\|^2 e_i^2\mathbf{1}_{\left\{e_i^2\geq \varepsilon^2 A_3A_4^{-1}n\lambda_n^{\frac{d}{2m}}\right\}}\right]\nonumber\\
        &=&A_3^{-1}n\lambda_n^{\frac{2k+d}{2m}}\mathbb{E}\left[ e_1^2\mathbf{1}_{\left\{e_1^2\geq \varepsilon^2 A_3A_4^{-1}n\lambda_n^{\frac{d}{2m}}\right\}}\right]\sum_{i=1}^n\|\mathbf{u}_{n,i}\|^2.\label{MVe3}
    \end{eqnarray}
    Note that on $\Xi_{\epsilon,n}$,
    \begin{eqnarray}
        \sum_{i=1}^n\|\mathbf{u}_{n,i}\|^2&=&\sum_{j=1}^{d_0}\sum_{i=1}^n \left[u_{in}^{(j)}\right]^2=\sum_{j=1}^{d_0}\lambda^{-2}_nn^{-2}\sum_{i=1}^{n}\left[\hat{g}_{jn}(x_{in})-g(x_{in})\right]^2\nonumber\\
        &=&\lambda^{-2}_n n^{-1}\sum_{j=1}^{d_0}\|\hat{g}_{jn}-g_j\|^2_n\nonumber\\
        &\leq& A_5 n^{-1}\lambda_n^{-\frac{2k+d}{2m}},\label{MVe4}
    \end{eqnarray}
    where the second equality follows from (\ref{u}); and the inequality follows from Lemma \ref{Th:improved} and $A_5>0$ is a constant independent of $n$ and $\lambda_n$. Combining (\ref{MVe3}) and (\ref{MVe4}) yields that, on $\Xi_{\epsilon,n}$,
    \begin{eqnarray*}
        \frac{1}{\underline{\lambda}_n}\sum_{i=1}^n\mathbb{E}\left[\|\mathbf{u}_{n,i}\|^2e_i^2\mathbf{1}_{\{\|\mathbf{u}_{n,i}\|^2e_i^2\geq \varepsilon^2 \underline{\lambda}_n\}}\right]\leq A_3^{-1}A_5\mathbb{E}\left[ e_1^2\mathbf{1}_{\left\{e_1^2\geq \varepsilon^2 A_3A_4^{-1}n\lambda_n^{\frac{d}{2m}}\right\}}\right],
    \end{eqnarray*}
    which tends to zero due to the condition $\lim_{n\rightarrow\infty}n\lambda_n^{\frac{d}{2m}}=\infty$ and the dominated convergence theorem.

    Hence we have proven the Lindeberg condition, where the convergence is in probability. It can be argued that, similar to that in the proof of Lemma \ref{lemma:CLT}, such a condition still ensures the central limit theorem.
\end{proof}

\begin{proof}[Proof of Proposition \ref{prop:L2}]
    
 Consider the linear functional $l(v):\mathcal{H}\mapsto \mathbb{R}$ with $l(v)=\langle v,g\rangle_\mathcal{H}$. Assumption \ref{assum:matern} ensures that $\mathcal{H}$ is dense in $L_2$, which, together with the condition
 \[\sup_{v\in\mathcal{H}}\frac{l(v)}{\|v\|_{L_2}}<\infty,\]
 implies that $l$ can be continuously and uniquely extended to $L_2$. Then the Riesz representation theorem asserts there exists a unique $h\in L_2$, such that $l(v)=\langle v,h\rangle_{L_2}$.
\end{proof}

\begin{proof}[Proof of Proposition \ref{Prop:series}]
    Under Assumption \ref{assum:matern}, $L_2$ is dense in $\mathcal{H}$. Consequently, we know that 1) $\rho_i>0$ for each $i$, and 2) $\{\eta_i\}_{i=1}^\infty$ forms an orthonormal basis of $L_2$. Let $v=\sum_{i=1}^s a_i\eta_i$. Theorem 10.29 of \cite{wendland2004scattered} shows the representation of the RKHS inner product as
    \[\left\langle \sum_{i=1}^\infty a_i\eta_i,\sum_{i=1}^\infty c_i\eta_i\right\rangle_\mathcal{H}=\sum_{i=1}^\infty\frac{a_ic_i}{\rho_i}.\]
    This implies
    \begin{eqnarray*}
        |\langle g,v\rangle_\mathcal{H}|=\left|\sum_{i=1}^\infty\frac{c_ia_i}{\rho_i}\right| &\leq&\left(\sum_{i=1}^\infty \frac{c_i^2}{\rho^{1+\kappa}_i}\right)^{1/2}\left(\sum_{i=1}^\infty \frac{a_i^2}{\rho^{1-\kappa}_i}\right)^{1/2}\\
        &=&\|g\|_{\mathcal{H},\kappa}\left(\sum_{i=1}^\infty (a_i^2)^\kappa\left(\frac{a_i^2}{\rho_i}\right)^{1-\kappa}\right)^{1/2}\\
        &\leq&\|g\|_{\mathcal{H},\kappa}\left(\sum_{i=1}^\infty a_i^2\right)^{\frac{\kappa}{2}}\left(\sum_{i=1}^\infty \frac{a_i^2}{\rho_i}\right)^{\frac{1-\kappa}{2}}\\
        &=&\|g\|_{\mathcal{H},\kappa}\|v\|^{\kappa}_{L_2}\|v\|^{1-\kappa}_{\mathcal{H}},
    \end{eqnarray*}
    where the first inequality follows from the Cauchy-Schwarz inequality, the second inequality follows from the H\"older's inequality with $(p,q)=(\frac{1}{\kappa},\frac{1}{1-\kappa})$.
\end{proof}

\subsection{Proofs for Uniform Bound in Section \ref{sec:uniform}} \label{sub:uniform}

\begin{proof}[Proof of Theorem \ref{Th:uniformvar}]
The main idea is to invoke Dudley's theorem \citep{VaartWCofEP00,vershynin2018high}, which states that a zero-mean sub-Gaussian process with respect to a semi-metric $d_Z$, i.e., a stochastic process $Z(x)$ satisfying $\mathbb{E}\exp\{\theta(Z(x_1)-Z(x_2))\}\leq \exp\{\vartheta^2 d^2_Z(x_1,x_2)/2\}$ for all possible $\vartheta,x_1,x_2$, is subject to the following uniform bound
\begin{eqnarray}\label{Dudley}  \mathbb{E}\left[\sup_{t\in\mathcal{T}}|Z(t)|\right]\leq \mathbb{E}|Z(t_0)|+ A \int_0^D \sqrt{\log N(\epsilon, \mathcal{T},d_Z)}d\epsilon,
\end{eqnarray}
for any $t_0\in \mathcal{T}$, where $D$ is the $d_Z$-diameter of $\mathcal{T}$, and $A$ is a universal constant.

Denote $g_x(\cdot)=D^\alpha K(\cdot,x)$. Then by (\ref{u}), we have
\begin{eqnarray}\label{Zx}
    D^\alpha \hat{f}(x)-D^\alpha f(x)=-\lambda^{-1}n^{-1}\sum_{i=1}^n (\hat{g}_x-g_x)(x_i)e_i:=Z(x).
\end{eqnarray}

Because $e_i$ is $\varsigma^2$-sub-Gaussian, conditional on $X$, $Z(x)$ is a zero-mean sub-Gaussian process with respect to the semi-metric 
\begin{eqnarray}\label{natural_distance1}
    d_\Omega(x_1,x_2)=\varsigma\lambda^{-1}n^{-\frac{1}{2}}\|\hat{g}_{x_1}-g_{x_1}-\hat{g}_{x_2}+g_{x_2}\|_n.
\end{eqnarray}
Then we can find an upper bound of $d_\Omega(x_1,x_2)$ by using the triangle inequality
\begin{eqnarray*}
    d_\Omega(x_1,x_2)\leq \varsigma\lambda^{-1}n^{-\frac{1}{2}}\left(\|\hat{g}_{x_1}-g_{x_1}\|_n+\|\hat{g}_{x_2}+g_{x_2}\|_n\right).
\end{eqnarray*}
Thus, by Lemma \ref{Th:improved}, on the event $\Xi_\epsilon$ defined in Assumption \ref{assum:norms}, we have
\[d_\Omega(x_1,x_2)\leq C_\Omega \varsigma\lambda^{-1}n^{-\frac{1}{2}} \lambda^{\frac{1+1-\frac{d+2|\alpha|}{2m}}{2}}=C_\Omega \varsigma n^{-\frac{1}{2}}\lambda^{-\frac{d+2|\alpha|}{4m}}, \]
for some constant $C_\Omega>0$.
On the other hand, let $g_{x_1,x_2}:=g_{x_1}-g_{x_2}$. Because KRR is linear, we have $\hat{g}_{x_1,x_2}=\hat{g}_{x_1}-\hat{g}_{x_2}$, and thus
\begin{eqnarray}\label{natural_distance2}
    d_\Omega(x_1,x_2)
    =\varsigma\lambda^{-1}n^{-\frac{1}{2}}\|\hat{g}_{x_1,x_2}-g_{x_1,x_2}\|_n.
\end{eqnarray}
Now we verify Assumption \ref{assum:f} for $g_{x_1,x_2}$. Note that, for $v\in H^m$
\[\langle g_{x_1,x_2},v\rangle_\mathcal{H}=D^\alpha v(x_1)-D^\alpha v(x_2).\]
Clearly, $D^\alpha v\in H^{m-|\alpha|}$. Noting $m>d/2+|\alpha|$, we can find $m>m'>d/2+|\alpha|$. Because $m'-|\alpha|>d/2$, the Sobolev embedding theorem (see, e.g., Theorem 4.47 of \cite{demengel2012functional}) claims the embedding relationship $H^{m'-|\alpha|}\hookrightarrow  C^{0,\tau}$ for $\tau:=\min(m'-|\alpha|-d/2,1)$, where $C^{0,\tau}$ denotes the H\"older space with the norm
\[\|f\|_{C^{0,\tau}}:=\sup_{x\neq x'}\frac{|f(x)-f(x')|}{\|x-x'\|^\tau}.\]
Thus, 
\[\langle g_{x_1,x_2},v\rangle_\mathcal{H}\leq \|D^\alpha v\|_{C^{0,\tau}}\|x_1-x_2\|^\tau\leq \|D^\alpha v\|_{H^{m'-|\alpha|}}\|x_1-x_2\|^\tau\leq \|v\|_{H^{m'}}\|x_1-x_2\|^\tau.\]
Next we use the interpolation inequality
\[\|v\|_{H^{m'}}\leq A \|v\|_{L_2}^{1-\frac{m'}{m}}\|v\|_{H^m}^{\frac{m'}{m}}.\]
This implies Assumption \ref{assum:f} for $C_{g_{x_1,x_2}}=A \|x_1-x_2\|^\tau$ and $\delta=1-\frac{m'}{m}$. Thus, by Lemma \ref{Th:improved} and (\ref{natural_distance2}), on the event $\Xi_\epsilon$ defined in Assumption \ref{assum:norms}, we find
\[d_\Omega(x_1,x_2)\leq C \varsigma\lambda^{-1}n^{-\frac{1}{2}}\lambda^{\frac{1+1-\frac{m'}{m}}{2}}\|x_1-x_2\|^\tau=C_1\varsigma n^{-\frac{1}{2}}\lambda^{-\frac{m'}{2m}}\|x_1-x_2\|^\tau,\]
for some $C_1>0$. Using the fact that $\Omega$ is a $d$-dimensional bounded region, we obtain that
\[N(\varepsilon,\Omega,d_\Omega)\leq N\left(\left(\epsilon/ (C_1\varsigma n^{-\frac{1}{2}}\lambda^{-\frac{m'}{2m}})\right)^{1/\tau},\Omega,\|\cdot\|\right).\]
Thus, by Lemma 4.1 of Pollard,
\[\log N(\varepsilon,\Omega,d_\Omega)\leq d \log \left(16D_\Omega\left(\frac{C_1\varsigma n^{-\frac{1}{2}}\lambda^{-\frac{m'}{2m}}}{\epsilon}\right)^{1/\tau}+1\right),\]
where $D_\Omega$ is the Euclidean diameter of $\Omega$.

Therefore, the integral in Dudley's theorem (\ref{Dudley}) has the upper bound
\begin{eqnarray*}
    &&
    \int_0^{C_\Omega n^{-\frac{1}{2}}\lambda^{-\frac{d+2|\alpha|}{4m}}}\sqrt{\log N(\varepsilon,\Omega,d_{\Omega})}d\varepsilon\\
    &\lesssim&\int_0^{C_\Omega \varsigma n^{-\frac{1}{2}}\lambda^{-\frac{d+2|\alpha|}{4m}}}\sqrt{\log \left(16D_\Omega\left(\frac{C_1 \varsigma n^{-\frac{1}{2}}\lambda^{-\frac{m'}{2m}}}{\epsilon}\right)^{1/\tau}+1\right)}d\varepsilon\\
    &=&\varsigma n^{-\frac{1}{2}}\lambda^{-\frac{d+2|\alpha|}{4m}}\int_0^{C_\Omega }\sqrt{\log \left(16D_\Omega\left(\frac{C_1\lambda^{\frac{d+2|\alpha|}{4m}-\frac{m'}{2m}}}{\epsilon}\right)^{1/\tau}+1\right)}d\varepsilon\\
    &\lesssim&\varsigma n^{-\frac{1}{2}}\lambda^{-\frac{d+2|\alpha|}{4m}}\sqrt{\log\left(\frac{C}{\lambda}\right)},
\end{eqnarray*}
for some $C>0$; where the last inequality is based on algebraic calculations similar to (33)-(36) in \cite{tuo2020kriging}. This term would dominate the first term of (\ref{Dudley}), which is given by Theorem \ref{Th:derivative}. Hence, we prove the desired result as $\mathbb{P}(\Xi_\epsilon)$ tends to one as $n\rightarrow \infty$.
\end{proof}

\begin{proof}[Proof of Theorem \ref{Th:uniformlower}]
    Let $Z(x)$ be the same as (\ref{Zx}). We can see that conditional on $X$, $Z(x)$ is a zero-mean Gaussian process with the natural distance
    \begin{eqnarray*}
        d_{\Omega}(x_1,x_2):=\left(\mathbb{E}_E [Z(x_1)-Z(x_2)]^2\right)^{1/2}=\sigma\lambda^{-1}n^{-\frac{1}{2}}\|\hat{g}_{x_1}-g_{x_1}-\hat{g}_{x_2}+g_{x_2}\|_n.
    \end{eqnarray*}
    Now the idea is to invoke the Sudakov's lower bound \citep{vershynin2018high}, which states that
    \begin{eqnarray}\label{sudakov}
        \mathbb{E}_E\left[\sup_{x\in\Omega}|Z(x)|\right]\gtrsim \sup_{\varepsilon>0}\varepsilon\sqrt{\log N(\varepsilon,\Omega,d_\Omega)}.
    \end{eqnarray}

    The boundary effect of $\Omega$ may cause some problems to our proof. So we define $\Omega'$ as a subset of $\Omega$ such that each $x\in\Omega'$ is distant from the boundary of $\Omega$ by at least $\eta$ in the Euclidean distance, where $\eta$ is sufficiently small such that $\Omega'$ contains an open set. Because $\sup_{x\in\Omega}|Z(x)|\geq \sup_{x\in\Omega'}|Z(x)|$, we will work only on a lower bound of $\sup_{x\in\Omega'}|Z(x)|$.
    
    Let $C_2>0$ be a constant to be determined, and let $M:=\lceil (C_2/\lambda)^{\frac{d}{2m}} \rceil$. In view of the lower bound for the covering number of a Euclidean compact set \citep{pollard1990empirical}, when $M\geq 2$, for any $M$ points $\{\xi_1,\ldots,\xi_M\}\subset\Omega'$, there exists two points, say $\{\xi_1,\xi_2\}$ without loss of generality, such that $\|\xi_1-\xi_2\|\leq C M^{-1/d}$ for some constant $C$ depending only on $\Omega'$.    
    Because $\lambda\rightarrow 0$, we shall assume that $C M^{-1/d}<2C (\lambda/C_2)^{1/(2m)}<\eta$ without loss of generality.

    Now let us consider $d_\Omega(\xi_1,\xi_2)$. The goal is to show that 
    \begin{eqnarray}\label{lowergoal}
    d_\Omega(\xi_1,\xi_2)\gtrsim \sigma n^{-\frac{1}{2}}\lambda^{-\frac{d+2|\alpha|}{4m}}.
    \end{eqnarray}
    If (\ref{lowergoal}) is true, we essentially prove that for $\varepsilon\leq C_1 \sigma n^{-\frac{1}{2}}\lambda^{-\frac{d+2|\alpha|}{4m}}$ for some constant $C_1>0$, we have $N(\varepsilon,\Omega',d_\Omega)\geq M\geq (C_2/\lambda)^{d/2m}$, which, together with (\ref{sudakov}), imples
    \begin{eqnarray}\label{GPlower}
        \mathbb{E}_E\left[\sup_{x\in\Omega'}|Z(x)|\right]\gtrsim C_1 \sigma n^{-\frac{1}{2}}\lambda^{-\frac{d+2|\alpha|}{4m}}\sqrt{\frac{d}{2m}\log\left(\frac{C_2}{\lambda}\right)}.
    \end{eqnarray}

    To prove (\ref{lowergoal}), we use Lemma \ref{Th:KRRlower}. We have to be mindful that the constants in Lemma \ref{Th:KRRlower} may depend on $n$ as $\xi_1,\xi_2$ are dependent on $n$. This necessities a closer look at the proof of Theorem \ref{Th:lower}, which Lemma \ref{Th:KRRlower} primarily relies on. First, we note that the interpolation inequality (\ref{interpolationD}) gives $\delta=1-\frac{d+2|\alpha|}{2m}$, independent of $n$. The constants from Assumption \ref{assum:norms} are also constants. However, constants in Assumption \ref{assum:tau} should be examined carefully. Our goal is to ensure Assumption \ref{assum:tau} with $\tau=\delta=1-\frac{d+2|\alpha|}{2m}$. To this end, we consider the function constructed in Proposition \ref{Prop:point}. In the proof of Proposition \ref{Prop:point}, we have constructed a function $\phi_\rho$ for each $\rho<C M^{-1/d}<\eta$ that satisfies the following properties:
    \begin{enumerate} [noitemsep]
        \item $\phi_\rho(\xi)=0$ if $\|\xi-\xi_1\|\geq \rho$;
        \item $D^\alpha \phi_\rho(\xi_1)=1$;
        \item $\|\phi_\rho\|_{L_2}=C_3\rho^{d/2}$ for some constant $C_3>0$ depending only on $m$.
        \item $\|\phi_\rho\|_{H^m}/\|\phi_\rho\|_{L_2}\leq C_4\rho^{-m}$ for some $C_4>0$ depending only on $m$.
    \end{enumerate}
    Hence, we have
    \[\langle \phi_\rho,g_{x_1}-g_{x_2}\rangle_\mathcal{H}=\phi_\rho(x_1)-\phi_\rho(x_2)=1,\]
    whenever $\rho<C M^{-1/d}$.
    This ensures Assumption \ref{assum:tau} for $g=g_{x_1}-g_{x_2}$ with $\tau=\delta$ independent of $n$, $C_0=C_3$ independent of $n$, $R_0=C'(C M^{-1/d})^{-m}\geq 2^{-m} C^{-m} C' (C_2/\lambda)^{1/2}$. Because only $R_0$ depends on $n$ (or $\lambda$), by examining the proof of Theorem \ref{Th:lower}, we can see that we only need to ensure (\ref{R0}), that is,
    \begin{eqnarray}\label{C2}
      2^{-m} C^{-m} C' (C_2/\lambda)^{1/2}\leq 4^{-1/\delta}C_0^{-1/\delta}C_1^{-1}C_g^{-1/\delta}\lambda^{-1/2},  
    \end{eqnarray}
    for the validity of Theorem \ref{Th:lower} and consequently, Lemma \ref{Th:KRRlower}. and (\ref{C2}) can be ensured provided that $C_2$ is sufficiently small. 
    Now we are ready to use Lemma \ref{Th:KRRlower}, which states that under the event $\Xi_\epsilon$, (\ref{lowergoal}) is true.
    Therefore we have proven (\ref{GPlower}), under $\Xi_\epsilon$.

    Because $\epsilon$ can be chosen arbitrarily small, the desired result is a direct consequence of the above statement together with the concentration inequality of Gaussian processes \citep{vershynin2018high}.
\end{proof}

\subsection{Proof for Nonlinear Problem in Section \ref{sec:nonlinear}}\label{sub:nonlinear}
\begin{proof}[Proof of Theorem \ref{Th:nonlinear}]
    It is well known that, 
    \begin{eqnarray}\label{nonlineruniform}
        \sup_{x\in\Omega}|D^\alpha\hat{f}(x)-D^\alpha f(x)|=o_\mathbb{P}(1),
    \end{eqnarray}
    for all $\alpha\in\mathbb{N}^d$ with $|\alpha|=0,1,2$ under the condition $\lambda\sim n^{-1}$; see \cite{geer2000empirical}.
    The uniform convergence of $\hat{f}-f$ implies the consistency of $\hat{f}_{\min}$ and $\hat{x}_{\min}$.

    Next, we study the rates of convergence of the estimators. Because $x_{\min}$ and $\hat{x}_{\min}$ minimize $f$ and $\hat{f}$, respectively, we have
    \begin{eqnarray}
        0&=&\frac{\partial\hat{f}}{\partial x}(\hat{x}_{\min})=\frac{\partial\hat{f}}{\partial x}(x_{\min})+\frac{\partial^2\hat{f}}{\partial x\partial x^T}(x^*)(\hat{x}_{\min}-x_{\min})\nonumber\\
        &=&\frac{\partial(\hat{f}-f)}{\partial x}(x_{\min})+\frac{\partial^2\hat{f}}{\partial x\partial x^T}(x^*)(\hat{x}_{\min}-x_{\min})\label{Taylor}
    \end{eqnarray}
    where $x^*$ lies between $x_{\min}$ and $\hat{x}_{\min}$. The consistency of $\hat{x}_{\min}$ and (\ref{nonlineruniform}) implies that $\frac{\partial^2\hat{f}}{\partial x\partial x^T}(x^*)$ converges weakly to $H$, which, together with the condition that $H$ is positive definite, implies that $\frac{\partial^2\hat{f}}{\partial x\partial x^T}(x^*)$ is invertible with probability tending to one. Therefore, (\ref{Taylor}) implies
    \begin{eqnarray}\label{nonlinearerror}
        \hat{x}_{\min}-x_{\min}=-\left[\frac{\partial^2\hat{f}}{\partial x\partial x^T}(x^*)\right]^{-1}\frac{\partial(\hat{f}-f)}{\partial x}(x_{\min}).
    \end{eqnarray}
    By Theorem \ref{Th:derivative}, under the optimal choice $\lambda\asymp n^{-1}$, $\|\frac{\partial(\hat{f}-f)}{\partial x}(x_{\min})\|=O_\mathbb{P}(n^{-\frac{1}{2}+\frac{d+2}{4m}})$. Thus $\|\hat{x}_{\min}-x_{\min}\|=O_\mathbb{P}(n^{-\frac{1}{2}+\frac{d+2}{4m}})$.

    To show the rate of convergence of $f(\hat{x}_{\min})-f(x_{\min})$, we use the Taylor expansion of $f$ at $x_{\min}$ to obtain
    \[f(\hat{x}_{\min})-f(x_{\min})=(\hat{x}_{\min}-x_{\min})^T\frac{\partial^2{f}}{\partial x\partial x^T}(x_*)(\hat{x}_{\min}-x_{\min}),\]
    for some $x_*$ lying between $x_{\min}$ and $\hat{x}_{\min}$. Again, we have that $\frac{\partial^2{f}}{\partial x\partial x^T}(x_*)$ converges to $H$ weakly, and therefore $f(\hat{x}_{\min})-f(x_{\min})=O_\mathbb{P}(n^{-1+\frac{d+2}{2m}})$.

    By (\ref{nonlinearerror}) and Theorems \ref{Th:oBIAS} and \ref{Th:multivariateAN}, we have
    \begin{eqnarray}
        \operatorname{COV}^{-\frac{1}{2}}\frac{\partial^2\hat{f}}{\partial x\partial x^T}(x^*)\left(\hat{x}_{\min}-x_{\min}\right)\xrightarrow{\mathscr{L}}N(0,I).
    \end{eqnarray}
    To prove the desired result, it suffices to show that
    \begin{eqnarray}\label{nonlineare1}
        \left[\frac{\partial^2\hat{f}}{\partial x\partial x^T}(\hat{x}_{\min})\right]^{-1}\widehat{\operatorname{COV}}^{\frac{1}{2}}\operatorname{COV}^{-\frac{1}{2}}\frac{\partial^2\hat{f}}{\partial x\partial x^T}(x^*)\xrightarrow{p} I.
    \end{eqnarray}

 Define 
    \[\mathscr{C}(\cdot):=\frac{\partial K}{\partial x}(\cdot,X)(K(X,X)+\lambda n I)^{-2}\frac{\partial K}{\partial x^T}(X,\cdot).\]
    Because both $\frac{\partial^2\hat{f}}{\partial x\partial x^T}(\hat{x}_{\min})$ and $\frac{\partial^2\hat{f}}{\partial x\partial x^T}(x^*)$ converges to $H$ weakly and $\hat{\sigma}^2\xrightarrow{p}\sigma^2$, (\ref{nonlineare1}) is equivalent to
    \begin{eqnarray}
        [\mathscr{C}(\hat{x}_{\min})]^{\frac{1}{2}}[\mathscr{C}(x_{\min})]^{-\frac{1}{2}}\xrightarrow{p} I.
    \end{eqnarray}
    We shall use the operator norm over $\mathbb{R}^{d\times d}$, given by
    \[\|M\|_{op}:=\sup_{\mathbf{v}\in\mathbb{R}^d}\frac{\|M\mathbf{v}\|}{\|\mathbf{v}\|},\]
    which is equal to the greatest absolute eigenvalue of $M$. By the sub-multiplicativity of the operator norm,
    \begin{eqnarray*}
        &&\left\|[\mathscr{C}(\hat{x}_{\min})]^{\frac{1}{2}}[\mathscr{C}(x_{\min})]^{-\frac{1}{2}}-I\right\|_{op}\\
        &=&\left\|\left([\mathscr{C}(\hat{x}_{\min})]^{\frac{1}{2}}-[\mathscr{C}(x_{\min})]^{\frac{1}{2}}\right)[\mathscr{C}(x_{\min})]^{-\frac{1}{2}}\right\|_{op}\\
        &\leq&\left\|[\mathscr{C}(\hat{x}_{\min})]^{\frac{1}{2}}-[\mathscr{C}(x_{\min})]^{\frac{1}{2}}\right\|_{op}\left\|[\mathscr{C}(x_{\min})]^{-\frac{1}{2}}\right\|_{op}.
    \end{eqnarray*}
    Now let $\mathbf{a}$ be a unit eigenvector of $[\mathscr{C}(\hat{x}_{\min})]^{\frac{1}{2}}-[\mathscr{C}(x_{\min})]^{\frac{1}{2}}$ corresponding to an eigenvalue $\lambda_0$ such that $|\lambda_0|=\|[\mathscr{C}(\hat{x}_{\min})]^{\frac{1}{2}}-[\mathscr{C}(x_{\min})]^{\frac{1}{2}}\|_{op}$. Then, we have
    \begin{eqnarray*}
        &&\mathbf{a}^T\left(\mathscr{C}(\hat{x}_{\min})-\mathscr{C}(x_{\min})\right)\mathbf{a}
        =\mathbf{a}^T[\mathscr{C}(\hat{x}_{\min})]^{\frac{1}{2}}\left([\mathscr{C}(\hat{x}_{\min})]^{\frac{1}{2}}-[\mathscr{C}(x_{\min})]^{\frac{1}{2}}\right)\mathbf{a}\\&&+\mathbf{a}^T\left([\mathscr{C}(\hat{x}_{\min})]^{\frac{1}{2}}-[\mathscr{C}(x_{\min})]^{\frac{1}{2}}\right)[\mathscr{C}(x_{\min})]^{\frac{1}{2}}\mathbf{a}\\
        &=&\lambda_0\mathbf{a}^T\left([\mathscr{C}(\hat{x}_{\min})]^{\frac{1}{2}}+[\mathscr{C}(x_{\min})]^{\frac{1}{2}}\right)\mathbf{a}.
    \end{eqnarray*}
    Therefore,
    \begin{eqnarray*}
        \left\|[\mathscr{C}(\hat{x}_{\min})]^{\frac{1}{2}}-[\mathscr{C}(x_{\min})]^{\frac{1}{2}}\right\|_{op}&=&\frac{\left|\mathbf{a}^T\left(\mathscr{C}(\hat{x}_{\min})-\mathscr{C}(x_{\min})\right)\mathbf{a}\right|}{\mathbf{a}^T\left([\mathscr{C}(\hat{x}_{\min})]^{\frac{1}{2}}+[\mathscr{C}(x_{\min})]^{\frac{1}{2}}\right)\mathbf{a}}.
    \end{eqnarray*}

 Denote $\mathbf{a}=(a_1,\ldots,a_d)^T$, and 
    \[\mathbf{g}(x):=\sum_{i=1}^d a_i\frac{\partial K}{\partial \chi_i}(x,X), \text{ and } h(x):=\mathbf{g}(x)(K(X,X)+\lambda n I)^{-2}\mathbf{g}^T(x).\]
    Then by the mean value theorem, there exists $\tilde{x}$ between $\hat{x}_{\min}$ and $x_{\min}$, such that
    \begin{eqnarray*}
        &&\left|\mathbf{a}^T\left(\mathscr{C}(\hat{x}_{\min})-\mathscr{C}(x_{\min})\right)\mathbf{a}\right|=|h(\hat{x}_{\min})-h(x_{\min})|\\
        &=&\left|\frac{\partial h}{\partial x^T}(\tilde{x})(\hat{x}_{\min}-x_{\min})\right|\leq \left\|\frac{\partial h}{\partial x^T}(\tilde{x})\right\|\|\hat{x}_{\min}-x_{\min}\|
    \end{eqnarray*}
    By Cauchy-Schwarz inequality,
    \begin{eqnarray*}
               \left\|\frac{\partial h}{\partial x^T}(\tilde{x})\right\| &=&\left\|2\frac{\partial \mathbf{g}}{\partial x}(\tilde{x})(K(X,X)+\lambda n I)^{-2}\mathbf{g}^T(\tilde{x})\right\|\\
        &\leq& 2\left\|\frac{\partial \mathbf{g}}{\partial x}(\tilde{x})(K(X,X)+\lambda n I)^{-2}\frac{\partial \mathbf{g}^T}{\partial x}(\tilde{x})\right\|_{op}^{1/2}\cdot\\
        &&\left(\mathbf{g}(\tilde{x})(K(X,X)+\lambda n I)^{-2}\mathbf{g}^T(\tilde{x})\right)^{1/2}
    \end{eqnarray*}
    In the proof of Theorem \ref{Th:multivariateAN}, we proved the upper and lower bounds of the maximum and the minimum eigenvalues of the covariance matrices. Note that these bounds do not depend on the choice of $x$, and thus are also true for $\hat{x}_{\min}$ and $\tilde{x}$. Specifically, we have
    \begin{eqnarray*}
       && \left\|\frac{\partial \mathbf{g}}{\partial x}(\tilde{x})(K(X,X)+\lambda n I)^{-2}\frac{\partial \mathbf{g}^T}{\partial x}(\tilde{x})\right\|_{op}\lesssim n^{-\frac{2m-4-d}{2m}}.\\
       && \lambda_{\min}(\mathscr{C}(\hat{x}_{\min}))\gtrsim n^{-\frac{2m-2-d}{2m}}.\\
       && \lambda_{\min}(\mathscr{C}(x_{\min}))\gtrsim n^{-\frac{2m-2-d}{2m}}.
    \end{eqnarray*}
    Hence, we obtain
    \begin{eqnarray*}
        \left\|[\mathscr{C}(\hat{x}_{\min})]^{\frac{1}{2}}[\mathscr{C}(x_{\min})]^{-\frac{1}{2}}-I\right\|_{op}\lesssim n^{-\frac{2m-3-d}{2m}} n^{-\frac{2m-2-d}{4m}}n^{\frac{2m-2-d}{2m}}=n^{-\frac{2m-4-d}{4m}}\rightarrow 0,
    \end{eqnarray*}
    as $n\rightarrow\infty$. This completes the proof.
\end{proof}

\subsection{Proof of Theorem \ref{Prop:semi}} \label{sub:semipar}

First, note that similar to (\ref{u}),
\[(f-\mathbb{E}_E\hat{f})(X)=\lambda n (K(X,X)+\lambda n I)^{-1}f(X).\]
Therefore, by Cauchy-Schwarz inequality,
\begin{eqnarray}
    &&~~~\left|\frac{1}{n}\sum_{i=1}^n (\mathbb{E}_E \hat{f}-f)(x_i)(h/p_X)(x_i)\right|\nonumber\\
    &&=\left|\lambda  f^T(X)(K(X,X)+\lambda n I)^{-1}(h/p_X)(X)\right|\nonumber\\
    &&\begin{aligned}
        \leq~&\lambda \left(f^T(X)(K(X,X)+\lambda n I)^{-1}f(X)\right)^{1/2}\cdot\\&\left((h/p_X)^T(X)(K(X,X)+\lambda n I)^{-1}(h/p_X)(X)\right)^{1/2}.
    \end{aligned}
    \label{semi1}
\end{eqnarray}
The two factors on the right hand of the above inequality are related to the noiseless KRR. For notational clarity, we denote the noiseless KRR for $f$ as $\mathscr{A}f$, i.e.,
\begin{eqnarray}\label{Af}
    \mathscr{A}f:=\operatorname*{argmin}_{v\in\mathcal{H}}\|f-v\|_n^2+\lambda\|v\|^2_\mathcal{H}.
\end{eqnarray}
Using the solution $\mathscr{A}f=K(\cdot,X)(K(X,X)+\lambda n I)^{-1} f(X)$ and by direct calculations, we have
\begin{eqnarray}
    &&\|f-\mathscr{A}f\|_n^2+\lambda\|\mathscr{A}f\|^2_\mathcal{H}\nonumber\\
    &=&\lambda^2nf^T(X)(K(X,X)+\lambda n I)^{-2}f(X)+\nonumber\\
    &&\lambda f^T(X)(K(X,X)+\lambda n I)^{-1}K(X,X)(K(X,X)+\lambda n I)^{-1}f(X)\nonumber\\
    &=&\lambda f^T(X)(K(X,X)+\lambda n I)^{-1}f(X).\label{semi2}
\end{eqnarray}
Also, (\ref{Af}) implies that
\[\|f-\mathscr{A}f\|_n^2+\lambda\|\mathscr{A}f\|^2_\mathcal{H}\leq \lambda\|f\|^2_\mathcal{H},\]
which, together with (\ref{semi2}), leads to 
\begin{eqnarray}\label{semi3}
    f^T(X)(K(X,X)+\lambda n I)^{-1}f(X)\leq \|f\|^2_\mathcal{H}.
\end{eqnarray}
Similarly, we have
\begin{eqnarray}\label{semi4}
    (h/p_X)^T(X)(K(X,X)+\lambda n I)^{-1}(h/p_X)(X)\leq \|h/p_X\|^2_\mathcal{H}.
\end{eqnarray}
Combining (\ref{semi1}), (\ref{semi3}) and (\ref{semi4}) yields the desired result.

\section{Additional Figures for Numerical Results}\label{sec:numerical_plot}
This section presents additional experimental results that complement the main content.

Figure~\ref{fig:testfun_plot} depicts the test functions used in the experiments in Subsection \ref{subsection:optimal_point}, providing a visual reference for the simulation settings.

Figure~\ref{fig:erp_data} displays the ERP dataset, consisting of 72 single-trial waveforms and their grand average. The two vertical lines indicate the search window used to estimate the optimal point in our real-data analysis.

\begin{figure}
\centering
\includegraphics[width=13cm,height=10cm]{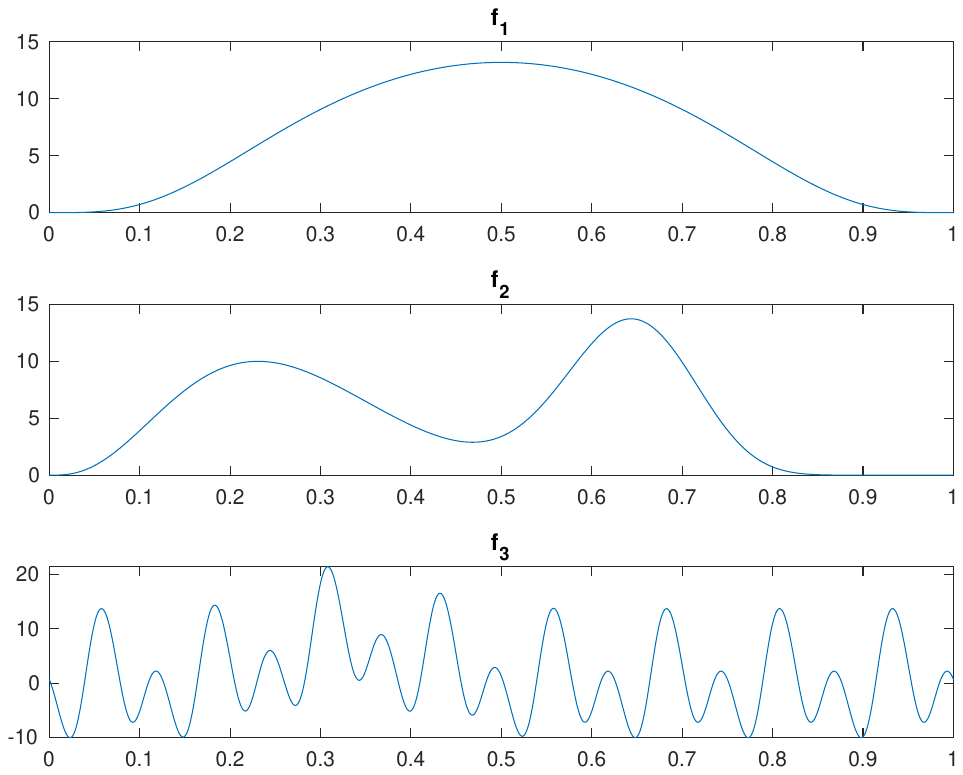}
\vspace*{-26mm}
\caption{Plots of Three Test Functions $f_1,f_2,f_3$ }
\label{fig:testfun_plot}
\end{figure}

\begin{figure}
    \centering
    \includegraphics[width=1.0\linewidth]{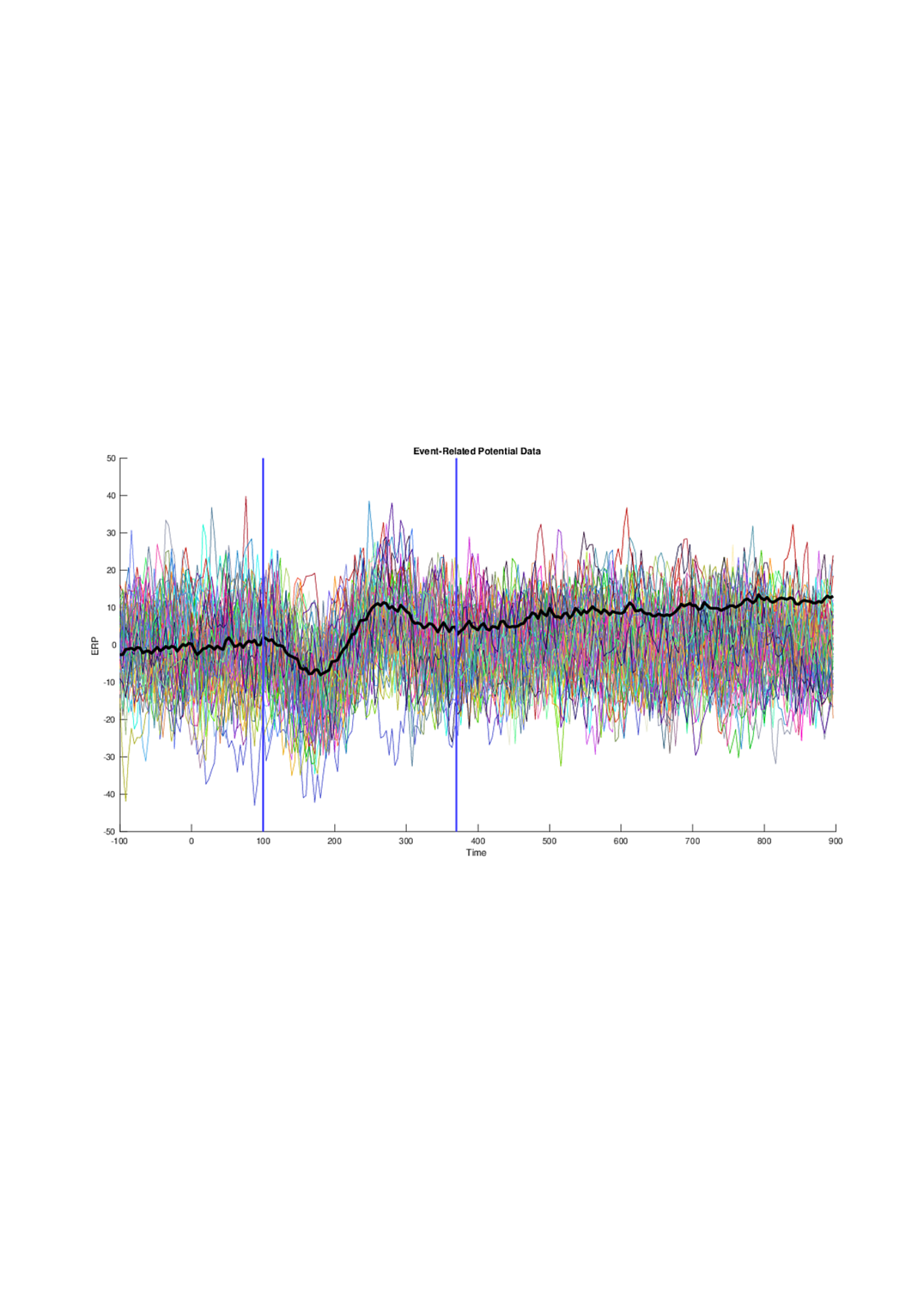}
    \vspace*{-76mm}
    \caption{The ERP data consists of 72 individual time series, each representing a single trial, along with the grand average time course computed from all trials. The time window between the two vertical lines indicates the search region.}
    \label{fig:erp_data}
\end{figure}



\bibliography{references}

\end{document}